\documentclass[12pt]{amsart}
\usepackage{amsmath,amssymb,mathrsfs,amsthm}
\usepackage{mathtools}
\usepackage{mathabx}
\usepackage[utf8]{inputenc}
\usepackage[T1]{fontenc}
\usepackage{color}
\usepackage{graphicx}
\usepackage{float}
\usepackage[bottom]{footmisc}
\usepackage[colorlinks,hyperindex,linkcolor=red]{hyperref}
\usepackage{enumerate}
\usepackage{cleveref}
\usepackage{tikz-cd}
\tikzset{
  symbol/.style={
    draw=none,
    every to/.append style={
      edge node={node [sloped, allow upside down, auto=false]{$#1$}}}
  }
}
\usepackage{enumitem}
\usepackage{wrapfig}
\usepackage{overpic}

\title{Vanishing cycles of symplectic foliations}
\author[F.\ Gironella]{Fabio Gironella}
\address[F.\ Gironella]{CNRS, Laboratoire de Mathématiques Jean Leray (UMR6629), Nantes Université, 44322 Nantes, France}
\email[F.\ Gironella]{fabio.gironella@cnrs.fr}

\author[K.\ Niederkrüger]{Klaus Niederkrüger}
\address[K.\ Niederkrüger]{ICJ, Université Claude Bernard Lyon 1, CNRS, Ecole Centrale de Lyon, INSA Lyon, Université Jean Monnet, ICJ UMR5208,
69622 Villeurbanne, France}
\email[K.\ Niederkrüger]{niederkruger@math.univ-lyon1.fr}

\author[L.\ Toussaint]{Lauran Toussaint}
\address[L.\ Toussaint]{Vrije Universiteit Amsterdam, De Boelelaan 1111, 1081 HV Amsterdam, The Netherlands}
\email[L.\ Toussaint]{l.e.toussaint@vu.nl}

\date{}

\newcommand{\cB}{\mathcal{B}}

\newcommand{\C}{{\mathbb{C}}}
\newcommand{\cC}{\mathcal{C}}

\newcommand{\D}{{\mathbb{D}}}
\newcommand{\cE}{{\mathcal{E}}}
\newcommand{\cF}{\mathcal{F}}

\newcommand{\cJ}{{\mathcal{J}}}
\newcommand{\cK}{\mathcal{K}}
\newcommand{\Jmod}{{J_{\mathrm{mod}}}}
\newcommand{\cM}{{\mathcal{M}}}
\newcommand{\cMtilde}{\widetilde{\cM}}
\newcommand{\cMline}{\overline{\cM}}
\newcommand{\N}{\mathbb{N}}
\newcommand{\bfq}{{\mathbf{q}}}
\newcommand{\bfp}{{\mathbf{p}}}
\newcommand{\bfx}{{\mathbf{x}}}
\newcommand{\bfy}{{\mathbf{y}}}

\newcommand{\R}{{\mathbb{R}}}
\newcommand{\fS}{\mathfrak{S}}

\newcommand{\cS}{\mathcal{S}}
\renewcommand{\S}{{\mathbb{S}}}

\newcommand{\T}{{\mathbb{T}}}
\newcommand{\univ}{{\operatorname{univ}}}
\newcommand{\cU}{\mathcal{U}}
\newcommand{\Umod}{{U_{\mathrm{mod}}}}
\newcommand{\X}{{\mathfrak{X}}}
\newcommand{\Z}{{\mathbb{Z}}}
\newcommand{\cCext}{\cC_{\mathrm{ext}}}

\newcommand{\delbar}{\overline{\partial}}

\newcommand{\antiClinear}{{\overline{\Hom}_\C}}

\newcommand{\Op}{{\mathcal{O}p}}

\newcommand{\Id}{{\operatorname{Id}}}
\DeclareMathOperator{\rank}{rank}
\renewcommand{\d}{{\operatorname{d}}}
\newcommand{\pr}{{\operatorname{pr}}}
\newcommand{\wtd}{\widetilde}
\DeclareMathOperator{\Aut}{Aut}

\DeclareMathOperator{\Cpx}{Cpx}
\DeclareMathOperator{\Diff}{Diff}
\DeclareMathOperator{\End}{End}
\DeclareMathOperator{\ev}{ev}
\DeclareMathOperator{\Hol}{Hol}
\DeclareMathOperator{\Hom}{Hom}
\DeclareMathOperator{\Image}{Image}
\DeclareMathOperator{\ind}{ind}

\DeclareMathOperator{\Vol}{Vol}
\DeclareMathOperator{\transv}{transv}

\newcommand{\defin}[1]{\textbf{#1}}

\newcommand{\Fabio}[1]{{\color{orange} Fabio: #1 }}

\newtheorem{lemma}{Lemma}
\newtheorem*{lemma*}{Lemma}
\newtheorem{proposition}[lemma]{Proposition}
\newtheorem{theorem}[lemma]{Theorem}
\newtheorem{corollary}[lemma]{Corollary}
\newtheorem{definition}[lemma]{Definition}
\newtheorem*{theorem*}{Theorem}
\newtheorem*{question*}{Question}
\newtheorem*{proposition*}{Proposition}

\newtheorem*{predef*}{Preliminary Definition}

\theoremstyle{remark}
\newtheorem{remark}[lemma]{Remark}

\newtheorem{example}[lemma]{Example}

\newtheorem{observation}[lemma]{Observation}

\theoremstyle{definition}

\newtheorem{convention}[lemma]{Convention}

\begin{document}

\maketitle

\begin{abstract}
    Several results in recent years have shown that the usual generalizations of taut foliations to higher dimensions, based only on topological concepts, lead to a theory that lacks the complexity of its $3$-dimensional counterpart.
    Instead, we propose strong symplectic foliations as natural candidates for such a generalization and we prove in this article that they do yield some interesting rigidity results, such as potentially topological obstructions on the underlying ambient manifold.

    We introduce a high-dimensional generalization of $3$-dimensional vanishing cycles for symplectic foliations, which we call \emph{Lagrangian vanishing cycles}, and prove that they prevent a symplectic foliation from being strong, just as vanishing cycles prevent tautness in dimension~$3$ due to the classical result of Novikov from 1964.

    We then describe, in every codimension, examples of symplectically foliated manifolds which admit Lagrangian vanishing cycles, but for which more classical arguments fail to obstruct strongness.
    In codimension~$1$, this is achieved by a rather explicit modification of the symplectic foliation, which allows us to open up closed leaves having non-trivial holonomy on both sides, and is thus of independent interest.
    
    Since there is no comprehensive source on holomorphic curves with boundary in symplectic foliations, we also give a detailed introduction to much of the analytic theory, in the hope that it might serve as a reference for future work in this direction.
\end{abstract}

\section{Introduction}

Thurston showed in his foundational work \cite{Thu} that general codimension~one foliations on closed manifolds are surprisingly flexible. 
This motivates the study of special subclasses of such foliations, which are more rigid and reflect the topology of the ambient manifold on which they live.

\subsection{Taut foliations in dimension 3}

One such class consists of \textbf{taut} foliations. 
While there are several equivalent definitions, in the $1$-codimensional case the most common one is to say that a foliation~$\cF$ on $M$ is taut if at any point of $M$ there exists a closed loop transverse to $\cF$.
These foliations have many remarkable properties.
For instance, Sullivan showed that a foliation is taut if and only if it admits a Riemannian metric for which the leaves are minimal surfaces, and the following theorem of Novikov illustrates the close interaction with the topology of the ambient manifold:

\begin{theorem}[Novikov \cite{Nov64,Nov65}]\label{thm:novikov}
    Let $\cF$ be a taut foliation on a closed $3$-manifold $M$. Then:
    \begin{itemize}
        \item the leaves of $\cF$ are incompressible hypersurfaces of $M$, that  is, the natural inclusion $\iota\colon F\to M$ of any leaf~$F$ induces an injective homorphism $\iota_*\colon \pi_1(F) \to \pi_1(M)$;
        \item every transverse loop $\gamma$ has infinite order in $\pi_1(M)$.
    \end{itemize}
\end{theorem} 

Note that the second conclusion implies that there is no taut foliation on $\S^3$.

\subsection{Reeb components and vanishing cycles}

The main source of flexibility for codimension~one foliations are so called Reeb components \cite{Thu,Eynard}.
A Reeb component consists of a solid torus $\S^1 \times \D^2$ foliated in such a way that its boundary is a (closed, separating) leaf, and its interior is foliated by copies of $\R^2$ accumulating on the boundary, see \Cref{fig:Reeb_component}.

\begin{wrapfigure}{r}{0.5\textwidth}
     \centering
     \def\svgwidth{0.3\textwidth}
    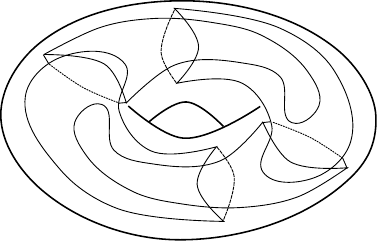  
     \def\svgwidth{0.5\textwidth}
     \vspace{7pt}
  \caption{The Reeb foliation on $\S^1\times \D^2$.} 
  \label{fig:Reeb_component}
\end{wrapfigure}

The Reeb foliation $\cF_{Reeb}$ induces a foliation by concentric circles on the disc $\{*\} \times \D^2 \subset \S^1 \times \D^2$. 
All of these circles are contractible inside their respective leaves of $\cF_{Reeb}$, except for the boundary circle $\{*\} \times \partial\D^2$, which is not trivial in $\pi_1(\S^1 \times \partial \D^2$).
We call such a disc (or its boundary circle, see \Cref{rmk:vanishing_cycle}) a \textbf{(non-trivial) vanishing cycle}. 
Novikov showed that the whole Reeb component can in fact be recovered from this disc. 
That is, any vanishing cycle of a foliation $(M,\cF)$ is a section of a Reeb component. 
As such, Reebless foliations are precisely those foliations without vanishing cycles, and we obtain another manifestation of rigidity for taut foliations:

\begin{theorem}[Novikov \cite{Nov64,Nov65}]\label{thm:Novikov_vanishingcycles}
A taut foliation does not contain any (non-trivial) vanishing cycles.
\end{theorem}

\begin{wrapfigure}[11]{l}{0.45\textwidth}
\begin{center}
\begin{overpic}[width = 0.45\textwidth]{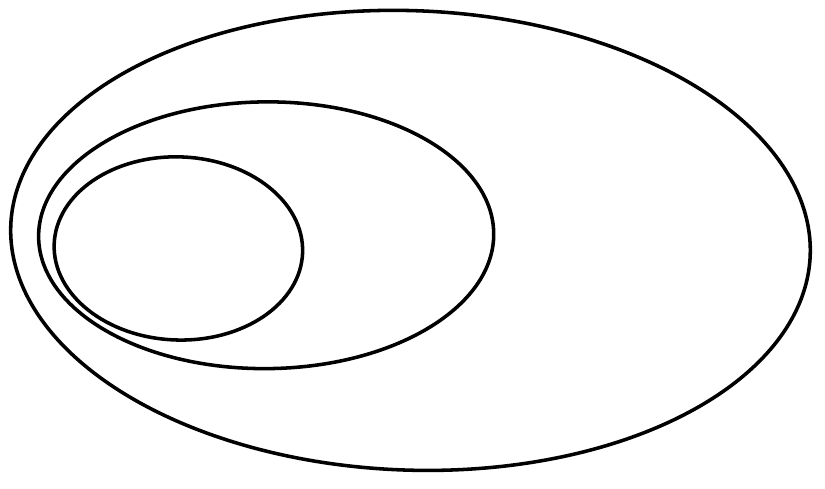}
\put(13,28){Taut}
\put(30,39){Reebless}
\put(51,48){Foliations}
\end{overpic}
\caption{The various classes of foliations on a $3$-dimensional manifold.}
\label{fig:situation_dim3}
\end{center}
\end{wrapfigure}

\begin{remark}\label{rmk:vanishing_cycle}
The usual definition of a vanishing cycle is slightly different, not requiring the existence of a full disc, but rather just a neighborhood of its boundary. 
However, by Novikov's result both definitions are equivalent, and the one we gave above is the one we will generalize to higher dimensions, see \Cref{def:lagr_van_cycle}. 
\end{remark}

\subsection{Taut and Reebless foliations}

Taut and Reebless foliations are closely related.
Any transverse curve inherits an orientation from the coorientation of $\cF$.
This implies that a taut foliation cannot contain any closed separating leaves: 
indeed, a transverse loop would have to intersect such a leaf both positively and negatively, leading to a contradiction. 
Since the boundary of a Reeb component is a closed separating leaf, this in particular implies that taut foliations are Reebless. 

The converse is not true.
Novikov \cite{Nov65} showed that a Reebless foliation~$(M,\cF)$ is taut if and only if there is no positive linear combination of torus leaves giving zero in $H_2(M)$. 
Thus, the situation in dimension-$3$ can be summarized as in \Cref{fig:situation_dim3}.

\subsection{Taut foliations in higher dimensions}

In higher dimensions the above definition of tautness 
does not imply the same rigidity as in dimension~$3$, even when considering the codimension-$1$ case.
It was shown by Meigniez that in dimension $4$ and higher, every hyperplane distribution is homotopic to a (codimension $1$) foliation all of whose leaves are dense.
These foliations are in particular taut, as it can easily be seen from the definition of tautness via transverse loops; in particular, the Euler characteristic is the only obstruction to the existence of taut foliations.
Moreover, it is not hard to see that the foliations from Meigniez's construction in general do not satisfy the conclusion of \Cref{thm:novikov}, for instance on the high dimensional sphere.
Thus, if we hope to extend the theory to higher dimensions while retaining some of the rigid behavior found in dimension~$3$, we need to consider a more restrictive class of foliations agreeing in dimension~$3$ with taut foliations.

\subsection{Strong symplectic foliations}

It is a classical result \cite[Proposition~10.4.1]{CandelConlon} that a codimension $1$ foliation is taut (in the sense that it admits a closed transverse loop through every point) if and only if there is a differential form $\mu\in\Omega^{\rank\cF}(M)$ such that $\d\mu=0$ on $M$ and that is a volume form on every leaf.
In order to give the definition of the class of foliations that we are interested in, as we will not work just with codimension~$1$ foliations, we recall a natural high-codimensional generalization of tautness in the style of the just mentioned equivalent formulation: a foliation (of arbitrary codimension) is taut if it admits a leafwise volume form whose ``leafwise variation'' is zero. Let us briefly explain what this means.

We let $\Omega^\bullet(\cF)$ denote the space of leafwise differential forms, and $\Omega^\bullet(\cF,\nu^*)$ the differential forms with values in the conormal bundle $\nu^* = (TM/\cF)^*$. 
We define the transverse differential:
\begin{equation*}
    \d_\nu\colon H^k(\cF) \to H^{k}(\cF,\nu^*), \quad [\alpha] \mapsto [\d_\nu \alpha] \;, 
\end{equation*}
where $\d_\nu\alpha$ is defined as:
\begin{equation*}
    (\d_\nu \alpha)(X_1,\dotsc,X_k)(\overline{V}) := (\d \wtd{\alpha})(X_1,\dotsc,X_k,V)  \;,  
\end{equation*}
for any choice of extension $\wtd{\alpha} \in \Omega^k(M)$ of $\alpha$, $X_1,\dotsc,X_k$ vectors tangent to $\cF$, and $\overline{V}$ the image of $V \in TM$ under the quotient map $TM \to \nu = TM/\cF$.
It is not hard to show \cite{GMP,TorresThesis} that $\d_\nu[\alpha]$ does not depend on the representative $\alpha$ or the choice of extension $\wtd{\alpha}$.

\begin{remark}\label{rmk:StrongAlternative}
    Suppose that $\mu$ is a leafwise volume form on a codimension-$k$ foliation $\cF$ on $M^{n+k}$. Then, $\d_\nu[\mu] = 0$ if and only if there exists an extension $\wtd{\mu} \in \Omega^n(M)$ such that
    \[ \d\wtd{\mu}(X_1,\dots,X_{k+1}) = 0,\]
    if at least $k$ of the $X_i$ are tangent to $\cF$. That is, $\mu$ is $\cF$-closed as defined classically e.g.\ in \cite[Definition~10.5.8]{CandelConlon}.
\end{remark}

A foliation $\cF$ of arbitrary codimension on $M$ is \textbf{taut} if there exists a $\mu\in\Omega^{\rank\cF}(M)$ such that $\d_\nu[\mu]=0$ and with $\mu$ a leafwise volume form.

\medskip

As in the non-foliated case, we can pass from the flexible realm to the rigid one by replacing volume forms with symplectic structures. 
We say that a foliation $\cF$ on a $(2n+1)$-dimensional manifold $M$ is \textbf{strong symplectic} if there exists $[\omega] \in H^2(\cF)$ satisfying
\begin{equation*}
    [\omega]^n \neq 0 \in H^{2n}(\cF),\quad \text{and} \quad \d_\nu[\omega ]  = 0  \; .
\end{equation*}
For a codimension~$1$ foliation, this condition is equivalent to the existence of a globally closed form $\Omega \in \Omega^2(M)$ whose restriction to $\cF$ defines a leafwise symplectic form.

Note that strong symplectic foliations are taut, since $\omega^n$ defines a leafwise volume form whose leafwise variation is zero.
In particular, in dimension~$3$ a foliation is taut if and only if it is strong symplectic.   

Strong symplectic foliations are good candidates to observe rigidity phenomena.
Two of the most powerful techniques yielding rigidity results in symplectic topology, namely pseudo-holomorphic curves and Donaldson techniques have already been succesfully applied in the setting of codimension-one strong symplectic foliations \cite{IboMar04a,IboMar04b,Mar09,Mar13,MTdPP18,AlbThesis,PreVen}. One of the most remarkable consequences is the existence of a transverse 3-dimensional manifold having the same transverse geometry (more precisely holonomy groupoids) as the ambient foliation.

\subsection{Lagrangian vanishing cycles}

The most pressing question is then if strong symplectic foliations, unlike taut foliations, do display the desired rigidity properties analogous to \Cref{thm:novikov}.
The most naive high-dimensional reformulation of \Cref{thm:novikov} does not generally hold as Venugopalan has found \cite{Venugopalan} examples of manifolds in any odd dimension at least $5$  which are equipped with a codimension~$1$ strong symplectic foliation admitting both a null-homotopic transverse loop and a non-trivial vanishing cycle (meant in the same sense as in the $3$-dimensional case described earlier, i.e.\ as an embedding of a $2$-disc with an induced foliation given by concentric circles such that the boundary is the only circle which is not contractible in the leaf of the ambient foliation in which it is contained).
Thus, even strong symplectic foliations do not satisfy the obvious generalizations of \Cref{thm:novikov} and \Cref{thm:Novikov_vanishingcycles}.

The rationale underlying the present work is that an appropriate notion of vanishing cycle for strong symplectic foliations should take into account the symplectic geometry of the leaves. We introduce the following notion of Lagrangian vanishing cycle. 
Since the full definition is somewhat technical, we defer to \Cref{def:lagr_van_cycle} for the precise formulation, and state the following preliminary version:

\begin{predef*}
    A \textbf{Lagrangian vanishing cycle} (modeled on a closed $(n-1)$-manifold $S$) in a symplectically foliated manifold $(M,\cF,\omega)$ is an embedding $e:\D^2 \times S \hookrightarrow M$ such that $e^*\cF$ is a (singular) $\omega$-Lagrangian foliation whose leaves are $\S^1_r \times S$, with $\S_r^1 \subset \D^2$ the circle of radius $r \in (0,1]$.

    \noindent

    A Lagrangian vanishing cycles is \textbf{trivial} if both the following hold: denoting by $F_1$ the leaf of $\cF$ containing $e(\S^1_1\times S)$, then
    \begin{itemize}
        \item $e(\S^1_1\times \{*\})$ bounds a disc with positive $\omega$-area in $F_1$;
        \item $L_1 = \partial \cC$ is null-homologous in $H_{n}(F_1;\Z_2)$.
        If $\cF$ is co-orientable and if $S$ is stably parallelizable, then it is also trivial in $H_{n}(F_1;\Z)$ and if  additionally $\rank(\cF)=4$ (in which case $S$ is necessarily a circle), $L_1\cong \T^2$ contracts to a loop inside $F_1$.
    \end{itemize}
\end{predef*}

In dimension~$3$, these conditions imply that $\dim S = 0$, so that a Lagrangian vanishing cycle with respect to any leafwise area form on $\cF$ is just the classical vanishing cycle found in a Reeb component.

Our main result is then the following: 

\begin{theorem}\label{thm:trivial_lagr_vanish_cycle}
    Every Lagrangian vanishing cycle of a strong symplectic foliation (of arbitrary codimension) on a closed manifold must be trivial.
\end{theorem}

In other words, with our notion of Lagrangian vanishing cycle, \Cref{thm:Novikov_vanishingcycles} generalizes to higher dimensions (and as a by-product we also obtain in ambient dimension~$3$ a new proof of \Cref{thm:Novikov_vanishingcycles} using symplectic techniques).

\subsection{Constructing examples}

It is natural to ask how strong symplectic foliations, and those without Lagrangian vanishing cycles are related, i.e.\ what is the high dimensional analogue of \Cref{fig:situation_dim3}?

One of the most prominent examples of symplectic foliations is the one constructed by Mitsumatsu \cite{Mit18} on $\S^5$. 
This foliation is obtained from an open book decomposition of $\S^5$ whose binding is a Boothby-Wang fibration over $\T^2$, and whose monodromy has order~$3$. 
From this it is not hard to see that Mutsumatsu's foliation contains a (non-trivial) Lagrangian vanishing cycle, and hence cannot be strong symplectic.

This example of a Lagrangian vanishing cycle is not very satisfactory from the perspective of \Cref{thm:trivial_lagr_vanish_cycle} since it already follows from the following elementary observation that the foliation cannot be strong symplectic:

\begin{observation}\label{obs:obstr_strongness}
    If one of the following conditions hold, $(\cF^{2n},\omega)$ is not strong:
    \vspace{-7pt}
    \begin{enumerate}
        \item $\cF$ is not taut (e.g. $\operatorname{codim} \cF = 1$ and $\cF$ has a closed separating leaf);
        \item $\cF$ has a closed leaf and the image of the map $\cdot^{\wedge n}\colon H^2(M)\to H^{2n}(M)$ is trivial (e.g.\ $H^2(M;\R)=\{0\}$).    
    \end{enumerate}
\end{observation}

In order to provide examples of symplectic foliations where strongness is obstructed by the presence of a non-trivial vanishing cycle, and not already from \Cref{obs:obstr_strongness} above, we present a new way to modify symplectic foliations, which is of independent interest.
The following construction is based on the strategy in \cite{Meigniez} and allows us to open up certain closed leaves of symplectic foliations:

\begin{proposition}\label{prop:eliminating_closed_leaves_intro}
    Let $(M,\cF,\omega)$ be a codimension~$1$ symplectic foliation, and let $F_0$ be a closed leaf.
    Suppose $F_0$ lies in the limit set of an open leaf on both sides.
    Then there exists an arbitrarily small neighborhood~$\cU$ of $F_0$ in $M$ and a symplectic foliation $(\cF',\omega')$ on $M$ which coincides with $(\cF,\omega)$ outside $\cU$ and for which every leaf of $\cF'$ intersecting $\cU$ is open.
\end{proposition}

The construction in the proof of \Cref{prop:eliminating_closed_leaves_intro} is relatively explicit, giving us considerable control over the topology of the resulting foliation.
In particular, we can perform the modification in such a way (see \Cref{lem:RemovingClosedLeaf}) that an existing Lagrangian vanishing cycle is preserved, which, according to the previous discussion of Mitsumatsu's example \cite{Mit18}, immediately implies the following:

\begin{corollary}\label{cor:mitsumatsu_sympl_fol_no_closed_leaves}
    The symplectic foliation on $\S^5$ described by Mitsumatsu \cite{Mit18} can be modified to obtain infinitely many non-isomorphic taut codimension~$1$ symplectic foliations containing a Lagrangian vanishing cycle and without closed leaves.
\end{corollary}

In fact, using \Cref{prop:eliminating_closed_leaves_intro} we can give further examples of symplectic foliations \emph{in every odd ambient dimension at least $5$}:

\begin{theorem}\label{thm:codim_1_examples}
    Let $n \geq 0$ and $\Sigma^{2n}$ be a closed, symplectically aspherical symplectic manifold, admitting a weakly exact Lagrangian~$L$.
    Then, $\S^3 \times \T^2 \times \Sigma$, and $\S^1 \times \S^2 \times \T^2 \times \Sigma$ admit a codimension~$1$ symplectic foliation which contains a non-trivial Lagrangian vanishing cycle, and does not satisfy any of the conditions in \Cref{obs:obstr_strongness} (and are, in particular, taut).
\end{theorem}

\begin{remark}
    Note that taking the product of the examples from \Cref{thm:codim_1_examples} with a closed  manifold of dimension $n$ directly gives closed manifolds $V^{2m+1}$ with a codimension-$(n+1)$ symplectic foliation that has a non-trivial Lagrangian vanishing cycle and don't satisfy any of the conditions in \Cref{obs:obstr_strongness}.
\end{remark}

The situation in higher dimensions can then be summarized as in \Cref{fig:situation_highdim}.
We have indicated in the figure which of the following foliations shows that a certain inclusion is strict:
\begin{enumerate}
    \item\label{exa:fol_mit} $\cF_{\operatorname{Mit}}$ is the codimension-one symplectic foliation on $\S^5$ constructed by Mitsumatsu \cite{Mit18}. It contains a Lagrangian vanishing cycle, and a Reeb component, and hence is not taut.
    \item\label{exa:fol_mit_modified} $\wtd{\cF}_{\operatorname{Mit}}$ is the modified Mitsumatsu foliation from \Cref{cor:mitsumatsu_sympl_fol_no_closed_leaves}, where we used \Cref{prop:eliminating_closed_leaves_intro} to open up the closed leaf while preserving the Lagrangian vanishing cycle. This results in a taut codimension-one foliation containing a Lagrangian vanishing cycle.
    \item The following construction is due to Venogupalan \cite{Venugopalan} and Seidel \cite{Seidel4dim}. 
    Let $\phi$ be a symplectomorphism of $(M^{2n},\omega)$ that is smoothly homotopic, but not symplectically homotopic to the identity.
    Then, the mapping torus $M(\phi)$ defines a symplectic fibration over the circle. Any choice of $\wtd{\omega} \in \Omega^2(M(\phi))$ that restricts to $\omega$ on the fibers defines a connection on this fibration. 
    Moreover, its monodromy is symplectic if and only if $\d\wtd{\omega} = 0$. See for example \cite[Lemma 6.3.5]{McDuffSalamonBookBasics}. 
    Now, since the chosen $\phi$ is smoothly isotopic to the identity, $M(\phi)$ is just smoothly of the form $\S^1\times X$ with foliation given by the second factor.
    However, since we assumed that $\phi$ is not symplectically homotopic to the identity, by \Cref{rmk:StrongAlternative} we deduce that the symplectic foliation defined by the fibers of $M(\phi)$ is not strong symplectic. 
    Moreover, again due to the fact that it is smoothly a product, it is not hard to see that it is taut and contains no non-trivial Lagrangian vanishing cycles $\cC$ just by topological reasons.
    Indeed, via the projection $\S^1\times M\to M$ one can naturally see every leaf $L_r$ of the induced foliation on $\cC$ in $M$, in which they would all be homotopic.
    \item\label{exa:fol_reebless_not_taut} Let $\cF$ be a codimension-one foliation on a $3$-manifold $(M,\cF_M)$, that is Reebless but not taut.
    Given a simply connected, closed, symplectic manifold $(W,\omega)$, the product foliation $\cF_{\operatorname{prod}}$ on $M \times W$ is a non-taut codimension-one symplectic foliation that doesn't contain any Lagrangian vanishing cycle. 
    Indeed, suppose by contradiction that $\cF_{\operatorname{prod}}$ does contain a Lagrangian vanishing cycle $\D^2 \times S$.
    For each $0<r \leq 1$ the leafwise curve $\S^1_r \times \{*\}$ induces an element
    \begin{equation*}
        [\S^1_r \times \{*\}] \in \pi_1(L_r \times W) = \pi_1(L_r) \;,    
    \end{equation*}
    for some leaf $L_r$ of $\cF_M$. By definition, this class is non-zero for $r = 0$ while it vanishes for $r>0$ sufficiently small. This implies that $(M,\cF_M)$ contains a vanishing cycle which contradicts the assumption that $\cF_M$ is Reebless
\end{enumerate}

\noindent
All the examples above are codimension-one foliations. By taking products we can obtain examples in any codimension (and in ambient dimension at least $5$ in the cases of \eqref{exa:fol_mit}, \eqref{exa:fol_mit_modified} and \eqref{exa:fol_reebless_not_taut}).

\begin{wrapfigure}{r}{0.55\textwidth}
\begin{center}
\begin{overpic}[width = 0.67\textwidth]{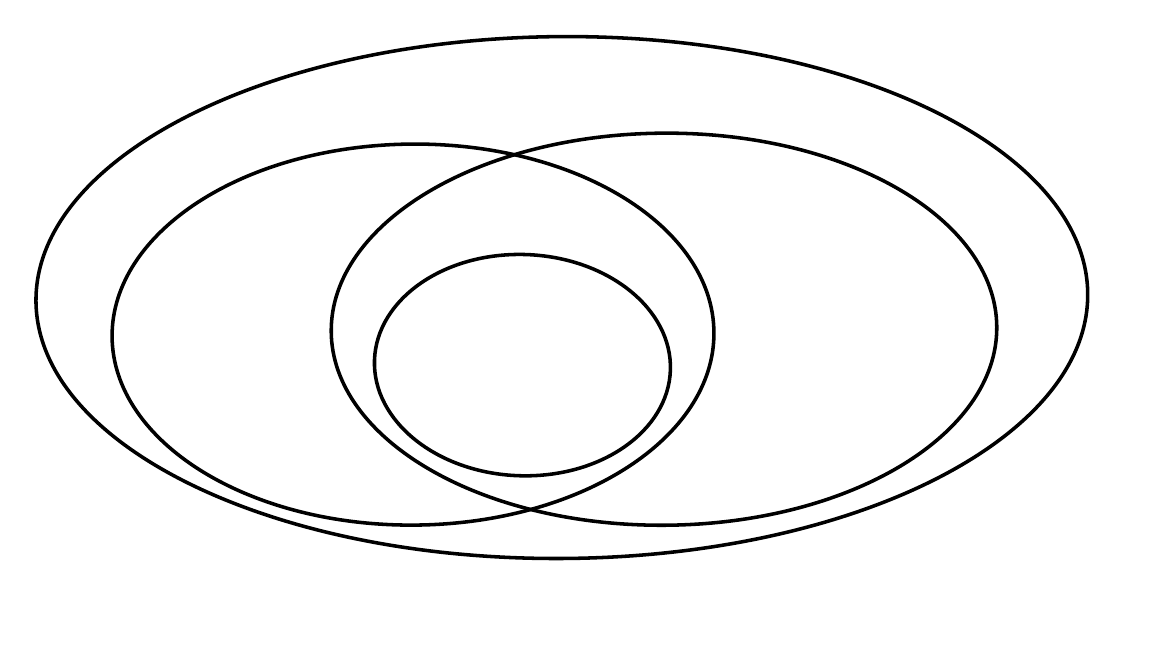}
\put(16,28){\parbox{2cm}{\scriptsize Taut}}
\put(35,25){\parbox{2cm}{ \scriptsize \centering Strong Symplectic}}
\put(61,32){\parbox{2.5cm}{\scriptsize No non-trivial \\ Lagrangian \\ Vanishing Cycles}}
\put(35,49){\parbox{3cm}{\scriptsize Symplectic foliations}}
\put(40,45){$\cF_{\operatorname{Mit}}$}
\put(17,21){$\widetilde{\cF_{\operatorname{Mit}}}$}
\put(42,37){$\cF_{\phi}$}
\put(68,21){$\cF_{\operatorname{prod}}$}
\end{overpic}
  \caption{The various types of symplectic foliations in higher dimensions; compare with \Cref{fig:situation_dim3}. 
  }
  \label{fig:situation_highdim}
\end{center}
\end{wrapfigure}

\subsection{Holomorphic techniques and the proof of the main theorem}

The proof of \Cref{thm:trivial_lagr_vanish_cycle} is inspired by the non-fillability proof for contact manifolds containing Plastikstufes \cite{Nie06} which itself is based on the classical idea of filling by pseudo-holomorphic discs due to \cite{Gro85}.
More precisely, after choosing a suitable leafwise almost complex structure~$J$, one finds a Bishop family of leafwise pseudo-holomorphic discs near the core $e(\{0\}\times S)$ of the Lagrangian vanishing cycle $\cC = e(\D^2\times S)$.
This family extends to a whole moduli space of discs of the same type, and the analysis of the possible degenerations leads to the desired conclusion.
The strongness assumption is crucial for obtaining global energy bounds for the moduli space.
Without them, the foliated version of Gromov's compactness theorem (see \cite{AlbNied}) fails and there is no control on the way discs can degenerate.

It should also be noted that in the symplectically foliated setup treated in this paper, there are several technical differences with respect to the the contact Plastikstufe in \cite{Nie06}. 
These are the reason behind the technical assumptions in \Cref{def:lagr_van_cycle}, that will be required to rule out disc and sphere bubbling, and the need to use domain-dependent perturbations of the almost complex structure to achieve generic regularity (c.f.\ \Cref{prop:mod_space_is_manifold_intro} and \Cref{rmk:domain_independent_J_for_closed_curves} below).

\medskip

Pseudo-holomorphic techniques have already been successfully applied to the study of foliations in previous works.
In \cite{dPiPre17}, the authors study a foliated version of the Weinstein conjecture generalizing ideas by Hofer.
Considering the symplectization of a given rank-$3$ contact foliation with an overtwisted leaf, they study the usual Bishop family of pseudo-holomorphic discs having boundary on the overtwisted disc.
The authors prove that there must be a closed contractible Reeb orbit for any defining foliated $1$-form in the closure of the overtwisted leaf.
Interestingly, the result is sharp in the sense that said closed orbit will in general not be on the overtwisted leaf itself.
Another instance where pseudo-holomorphic techniques have been used is \cite{PreVen}.
Here the authors study contact foliations that admit a filling by a strong symplectic foliation, and they prove uniqueness of such minimal foliated symplectic fillings for the unit cotangent bundle of the Reeb foliation on the $3$-sphere, which is naturally a contact codimension~$2$ foliation on a $5$-dimensional ambient manifold.
In his thesis \cite{AlbThesis}, Alboresi developed a large part of the analysis of closed pseudo-holomorphic curves in the foliated setting and applied it both to regular and to $\log$-symplectic foliations, that is, foliations having a mild type of singular leaves.
He then classifies all compact leaves of a codimension~$1$ foliation in dimension~$5$, containing a symplectic $2$-sphere with trivial normal bundle in one of its leaves.

At this point, it should be noted that a large part of the content of the present paper consists in carefully laying-out the foundations for the use of compact pseudo-holomorphic curves with boundary in a symplectic foliation.
More precisely, one first foundational result that we prove is the following (see \Cref{thm:leafwise_maps_trivial_hol_is_banach_submanifold} for a more precise statement):

\begin{theorem}
    \label{thm:leafwise_maps_trivial_hol_is_banach_submanifold_intro}
    Let $(M,\cF)$ be a foliated manifold and let $\Sigma$ be a compact connected $n$-manifold that has either empty or non-empty but connected boundary.
    In the latter case, fix also a foliated submanifold $P\subset (M,\cF)$.
    The space of leafwise $C^k$-maps $u\colon \Sigma \to M$ that have trivial holonomy and whose boundary (if non-empty) is mapped to $P$ is a smooth Banach manifold.
    The same also holds for Sobolev $W^{k,p}$-maps with $(k-1)\,p>n$.
\end{theorem}

We then explicitly develop all the Fredholm analysis necessary to talk about moduli spaces of pseudo-holomorphic maps for Riemann surfaces~$\Sigma$ with either connected or empty boundary.
More precisely, under this assumption we prove the following (see \Cref{prop:mod_space_is_manifold} for a more precise version, allowing to fix an arbitrary almost complex structure on some open set):

\begin{proposition}
    \label{prop:mod_space_is_manifold_intro}
    Let $(M,\cF,\omega)$ be a symplectically foliated manifold, and $\Sigma$ a Riemann surface with either connected or empty boundary.
    For generic (domain-dependent) almost complex structures $J$, the space of leafwise, trivial-holonomy $J$-holomorphic maps $u\colon \Sigma \to M$ is a smooth finite-dimensional manifold.
\end{proposition}

\begin{remark}
    \label{rmk:domain_independent_J_for_closed_curves}
    The reason we use domain-dependent almost complex structures in the proof of \Cref{prop:mod_space_is_manifold_intro} is that in the case of our main interest, i.e.\ leafwise pseudo-holomorphic discs with boundary on a Lagrangian vanishing cycle, the geometric setup does not guarantee that the curves are somewhere injective.
    This being said, in the somewhere-injective case, e.g.\ when dealing with simple leafwise  pseudo-holomorphic spheres, generic transversality can be achieved also with domain-independent almost complex structures with an analogous proof.
\end{remark}

\subsection*{Acknowledgements}

The authors thank Fran Presas for suggesting the idea that symplectic foliations with a closed extension of the leafwise symplectic form to the surrounding manifold may provide uniform energy bounds, thereby allowing us to propagate holomorphic curves from one leaf to neighboring ones.
The authors would also like to thank Petru Mironescu, for useful discussions concerning the Poincaré inequality for Sobolev-regular functions, Georgios Dimitroglou Rizell for discussing positivity of intersection arguments for holomorphic curves with boundary, as well as Claude Viterbo for his interest in this work and for suggesting the more appropriate name ``Lagrangian'' instead of ``symplectic'' vanishing cycles.

This work was funded by the ANR grant ``COSY -- New challenges in symplectic and contact topology'' (ANR-21-CE40-0002), by the region Pays de la Loire, via the project Étoile Montante 2023 SymFol; and by the French government's “Investissements d’Avenir” program integrated to France 2030 (ANR-11-LABX-0020-01).

The third author is funded by the NWO project ``proper Fredholm homotopy theory'' (project number OCENW.M20.195) of the research programme Open Competition ENW M20-3.


\section{The Lagrangian vanishing cycle}
\label{sec:lagr_van_cycles}

Before giving the definition of the Lagrangian vanishing cycle, we will briefly recall some basic concepts.

\begin{definition}\label{def: foliated submanifold}
  Let $M$ be a manifold carrying a regular\footnote{For the rest of the paper, foliations will always be assumed to be regular, unless explicitly called ``singular''.} foliation $\cF$, and let $P$ be a submanifold (possibly with boundary).
  We say that $P$ is a \defin{foliated submanifold} if around every point~$x\in P$ there exists a chart~$(U_x,\phi_x)$ of $M$ that is simultaneously a foliation chart for $\cF$, and a submanifold chart for $P$.
\end{definition}

Using this chart, it is clear that $\cF_P := \cF\cap TP$ is a foliation on $P$. 
The following criterion, whose proof we omit, simplifies the verification that a submanifold is foliated, see \cite{AlbNied}.

\begin{proposition}
  Let $M$ be a manifold carrying a foliation~$\cF$, and let $x$ be a point on a submanifold $P\subset M$.
  Then, the manifold~$P$ is foliated on some neighborhood of $x$, if and
  only if the dimension of $\cF\cap TP$ is constant close to $x$.
\end{proposition}

In what follows, we will be interested in submanifolds having an induced foliation which is compatible with the leafwise symplectic structure in the following sense:

\begin{definition}\label{def: Lagrangian foliation}
    Let $M$ be a smooth manifold of dimension $2n+q$ that carries a symplectic foliation $(\cF,\omega)$ of codimension~$q$.
    A foliated submanifold $P\subset M$ has a \defin{Lagrangian foliation} if the restriction of $\omega$ to every leaf of the induced foliation $\cF_P = TP \cap \cF$ vanishes.
\end{definition}

As we will have to deal with submanifolds whose induced foliation is in fact singular, we also describe explicitly the singularity that we allow:

\begin{definition}\label{def:simple_tangency}
    Let $M$ be a $(2n+q)$-dimensional manifold carrying a codimension~$q$ symplectic foliation~$(\cF,\omega)$, and let $P$ be an $(n+1)$-dimensional submanifold of $M$.
    
    We say that $P$ has a \defin{simple Lagrangian-type tangency with $\cF$} along the subset~$S$ if the following properties hold.
    \begin{enumerate}
        \item $S$ is a codimension~$2$ submanifold of $P$ that is contained in a leaf~$F_0\in \cF$, of which it is an isotropic submanifold with trivial symplectic normal bundle.
        \item  There is a neighborhood of $S$ in $M$ that contains a foliated submanifold diffeomorphic to $\D^2_\epsilon \times T_\epsilon^*S\times \R$ such that:
        \begin{enumerate}
        \item $S$ corresponds to $\{0\}\times S \times \{0\} \subset \D^2_\epsilon \times T_\epsilon^*S\times \R$;
        \item the leaves of the foliation of $\D^2_\epsilon \times T_\epsilon^*S\times \R$ are the  slices with constant $\R$-value;
        \item the leafwise symplectic structure on every leaf $\D^2_\epsilon \times T^*_\epsilon S\times\{t\}$ is
        \begin{equation*}
            i\, \d z \wedge \d\overline{z} + \d\lambda_{\mathrm{can}} \;,
        \end{equation*} where $\lambda_{\mathrm{can}}$ is the canonical $1$-form on $T^* S$;
        \item the intersection of $P$ with this neighborhood lies in the submanifold $\D^2_\epsilon \times T_\epsilon^*S\times \R$, and is given by the embedding
        \begin{equation*}
            \D^2_\epsilon\times S \hookrightarrow
            \D^2_\epsilon \times T_\epsilon^*S\times \R, \quad
            (z,x) \mapsto \bigl(z,x; |z|^2\bigr) \;.
        \end{equation*}
        \end{enumerate}
    \end{enumerate}
\end{definition}

\begin{remark}\label{rmk:codim_1_simple_tangency}
    In the case where $\cF$ is coorientable and of codimension~$1$, the submanifold $\D^2_\epsilon \times T_\epsilon^*S\times \R$ is an open neighborhood of $S$, and the condition that the leaves in this model are slices with constant $\R$-value is equivalent by Reeb stability (see \Cref{lemma:holonomy}) to the condition that the holonomy of $\cF$ along $S$ is trivial.
\end{remark}

\begin{example}\label{exa:lagrangian_type_tangency}
    If the codimension of $\cF$ is larger than $1$, the easiest explicit model satisfying the hypothesis in \Cref{def:simple_tangency} would correspond to a foliation with trivial holonomy of the form
    \begin{equation*}
        \D^2_\epsilon \times T_\epsilon^*S \times \R^q
    \end{equation*}
    with a symplectic foliation given by the leaves with constant value in $\R^q$ and the leafwise symplectic structure given by
    $\omega_0 = i\, \d z \wedge \d\overline{z} + \d\lambda_\mathrm{can}$.
    \\
    The embedding of $P$ in this model would be
    \begin{equation*}
        \D^2_\epsilon\times S \hookrightarrow
        \D^2_\epsilon \times T_\epsilon^*S\times \R^q, \quad
        (z,x) \mapsto \bigl(z,x; |z|^2,0,\dotsc,0\bigr) \;.
    \end{equation*}
    Morally, the meaning of \Cref{def:simple_tangency} is that we do not require that the holonomy of $\cF$ along $S$ is trivial along all transverse directions to $\cF$ but just along one.
\end{example}

We are now ready to give the main definition of the paper:

Let $M$ be a closed smooth manifold of dimension $2n+q$ that carries a symplectic foliation $(\cF,\omega)$ with $\rank \cF = 2n$.
We further denote by $\D^2$ the unit disc centered at the origin in $\R^2$, and by $\S^1_r\subset \R^2$ the circle of radius $r>0$ centered at the origin; with a little abuse of notation, we will sometimes denote by $\S^1_0$ the ``circle of radius $0$'', i.e.\ the origin.

\begin{definition}\label{def:lagr_van_cycle}
    A \defin{Lagrangian vanishing cycle (modeled on a closed manifold~$S$)} is an embedding $i\colon \D^2 \times S\hookrightarrow M$, or its image~$\cC$ with some abuse of notation, satisfying the following properties.
    \begin{enumerate}

        \item\label{item:lvc_lagrangian_fol} $\cC$ is a singularly foliated submanifold, with the foliation being regular on the complement $\cC^*$ of the \defin{core} $\cC_0=\{0\}\times S$.
        The leaves are more precisely given by $L_r\coloneqq i(\S^1_r\times S)$ for $0< r \leq 1$ and are Lagrangian with respect to $\omega$.
    
        \item\label{item:lvc_core} $\cC$ has a simple Lagrangian-type tangency with $\cF$ along $\cC_0$
    
        \item\label{item:lvc_sympl_aspherical_leaves} For every $0<r<1$, $\omega\vert_{\pi_2(F_{r})}=0$, where $F_{r}$ is the leaf of $\cF$ containing $L_{r}$.

        \item\label{item:lvc_pi1_condition} For every $0< r < 1$, $\ker(\pi_1(L_r) \to \pi_1(F_{r}))$ is either trivial or generated by $[\S^1_r \times \{p\}]$ with $p\in S$.

        \item\label{item:lvc_boundary_extension} If $q>1$ then we require that we can attach a collar along the boundary of $\cC$ such that the extended manifold is everywhere (except for $\cC_0$) foliated. 

    \end{enumerate}

    A Lagrangian vanishing cycle~$\cC$ is said to be \defin{trivial} if the following two conditions hold:
    \begin{enumerate}
        \setcounter{enumi}{5}
        \item\label{item:lvc_circle_bound_sympl_disc} for any $q\in S$, $i(\S^1_1\times \{q\})$ bounds an immersed disc with positive $\omega$-area in $F_{1}$;
        \item\label{item:lvc_boundary_null_homologous} $L_1 = \partial \cC$ is null-homologous in $H_{n}(F_1;\Z_2)$.
        If $\cF$ is co-orientable and if $S$ is stably parallelizable, then it is also trivial in $H_{n}(F_1;\Z)$ and if  additionally $\rank(\cF)=4$ (in which case $S$ is necessarily a circle), $L_1\cong \T^2$ contracts to a loop inside $F_1$.
    \end{enumerate}
\end{definition}

\begin{remark}
    \label{rmk:pi1_condition_def_lagr_vanish_cycle}
    Item~\eqref{item:lvc_pi1_condition} is trivially true if $S$ is simply connected.
    In this case, the normal form in \Cref{def:simple_tangency} also simplifies considerably, as the holonomy along $\cC_0=S$ is necessarily trivial and one can always find the foliated submanifold as in the third point of the definition; see \Cref{exa:lagrangian_type_tangency}.
    \\
    Moreover, \eqref{item:lvc_pi1_condition} extends to the case of $r=1$.
    Indeed, if there were contradiction a loop~$\gamma$ in $F_1$ whose homotopy class were not a multiple of $[\S^1_1\times \{p\}]$ and which bounded a disc~$D$ contained in $F_1$, then one could push the disc~$D$, and hence $\gamma$, by Reeb stability to a nearby leaf $F_{r_0}$ for some $r_0<1$; this would contradict the assumption~\eqref{item:lvc_pi1_condition} for $r_0$, leading to a contradiction. 
\end{remark}

\begin{remark}\label{rmk:local_leaf}
    Two leaves~$F_{r}$ and $F_{r'}$ may coincide even though $r\neq r'$. 
    Hence, $F_r$ might contain $L_r$ and $L_{r'}$ for some $r'\neq r$.
    This happens for instance in the case of the standard Reeb (codimension~$1$) foliation on $\S^1\times\S^2$, with a ``horizontal'' (i.e.\ contained in a sphere $\{pt\}\times\S^2$) vanishing cycle $\D^2\hookrightarrow (\S^1\times\S^2,\cF_{Reeb})$.
\end{remark}

\begin{remark}
\label{rmk:codim_1_boundary_extension}
    Condition~\eqref{item:lvc_boundary_extension} above holds trivially if $q=1$, but is in general not automatically satisfied in higher codimensions.
    The toy model to have in mind is $\R \times \R^q$ foliated by the affine lines $\R\times \{p\}$ for any $p\in \R^q$ fixed.
    Quotient $\R \times \R^q$ by the $\Z$-action $k\cdot (t,p) = (t+k, \phi^k(p))$, where $\phi$ is a diffeomorphism $\R^q \to \R^q$ that fixes the segment $I_+ := \{(x,0,\dotsc,0)\in \R^q| x\ge 0\}$ pointwise.
    \\
    It follows that we can embed $\S^1\times [0,\infty)$ by $(e^{2\pi t},x) \mapsto (t;x,0,\dotsc,0)$ to obtain a foliated submanifold with boundary.
    If $q=1$, then $\phi$ will also fix the segment $\{(x,0,\dotsc,0)\in \R^q| x< 0\}$ (but not necessarily pointwise); if $q>1$ then it is easy to think of a $\phi$ that does not fix any $1$-dimensional extension of $I_+$.
\end{remark}

Morally speaking, conditions \eqref{item:lvc_lagrangian_fol} and \eqref{item:lvc_core} are, in a certain sense, of a semi-local nature, and together with the opposite of either \eqref{item:lvc_circle_bound_sympl_disc} or \eqref{item:lvc_boundary_null_homologous}, which are of a less local nature, they are the main topological properties that one would expect from a high-dimensional symplectic generalization of classical (non-trivial) vanishing cycles for smooth foliations on $3$-manifolds.
On the other hand, conditions 
\eqref{item:lvc_sympl_aspherical_leaves} and \eqref{item:lvc_pi1_condition}
are of a more technical nature, and we ask them in order to ensure the absence of undesired bubbling in the compactification of moduli spaces of pseudo-holomorphic discs that we will consider in order to prove \Cref{thm:trivial_lagr_vanish_cycle}.

\begin{remark}
    \label{rmk:sympl_aspheric_not_avoidable}
    Condition \eqref{item:lvc_sympl_aspherical_leaves} has several purposes.
    First, it is used to guarantee uniform energy bounds for the just mentioned moduli spaces.
    Then, it is used to exclude bubbling of spheres as it is customary, but also in combination with \eqref{item:lvc_pi1_condition} in order to exclude disc bubbling.
\end{remark}

\medskip

\begin{remark}
\label{rmk:alternative_candidate_vanishing_cycle}
    During the preparation of this work, we have also considered the following alternative candidate for a higher dimensional vanishing cycle for symplectic foliations.
      
    Let $(M,\cF,\omega)$ be a manifold with a symplectic foliation, and let $h\colon M \to \R$ be a smooth function for which $V = f^{-1}(0)$ is a regular level set.
    Assume that there is locally on an open subset $U\subset M$ a vector field $X$ that is tangent to $\cF$ and leafwise Liouville, pointing positively transversely through $V$, and hence inducing on $V\cap U$ a contact structure.

    A vanishing cycle would now be a manifold $\cC \subset V\subset M$ with $\dim \cC = \frac{1}{2}\, \dim \cF + 1$ carrying the structure of a bordered open book $(\cC_0,\vartheta)$ where the binding $\cC_0\subset \cC$ is a codimension~$2$ submanifold that plays the role of the core of the vanishing cycle, and $\vartheta\colon \cC \setminus \cC_0 \to \S^1$ is a fibration that is transverse to the boundary of $\cC$ and that agrees with the angular coordinate on the normal bundle of $\cC_0$ in $\cC$.
    The core~$\cC_0$ and the boundary~$\partial \cC$ are supposed to lie each in a leaf of $\cF$, and $\cC^*:=\cC\setminus \cC_0$ to lie in $V\cap U$ where $X$ is well defined.
    We also assume that the kernel of the $1$-form~$\alpha = \iota_{X_L}\omega$ and of $\ker \vartheta$ in each regular leaf $\cC\cap \cF$ agree, so that $\vartheta$ defines actually a Legendrian foliation.
    The triviality of this vanishing cycle would be defined similarly to the one in \Cref{def:lagr_van_cycle}.

    It is not difficult to see that the vanishing cycle for a codimension~$1$ foliation in dimension~$3$ is not only a Lagrangian vanishing cycle in the sense of \Cref{def:lagr_van_cycle} but also a vanishing cycle according to the definition we have sketched here.
    
    Among the advantages of this definition is that there is no need to have any condition on the fundamental group as in \Cref{def:lagr_van_cycle}.\eqref{item:lvc_pi1_condition} to avoid bubbling, there is no need to choose domain dependent perturbations, we could possibly obtain a positivity of intersection arguments in case $\dim\cF = 4$, etc.
    
    However, we did not investigate this direction further, because we felt that the existence of a psh-function on $(M,\cF,\omega)$ would be too restrictive for a high-dimensional generalization of the vanishing cycle.
\end{remark}


\section{Analytical foundations}
\label{sec:analitical_foundations}

This section provides the analytical foundations of the moduli space of pseudoholomorphic discs needed to prove the main result.

We point out that this has already been done to some extent, at least for leafwise pseudo-holomorphic maps of spheres, by Alboresi in his thesis \cite{AlbThesis} using a local charts approach. 

The viewpoint we take here is a more abstract one.
First, we deduce that the subspace of leafwise maps with trivial holonomy is a submanifold of the ambient space of all maps, which is well-known to be a smooth Banach manifold.  We then study the underlying Fredholm analysis in the context of domain-dependent almost complex structures and Riemann surfaces with either empty or connected boundary.
We point out that domain-dependent perturbations are necessary for us to ensure generic regularity for leafwise pseudo-holomorphic discs with totally real boundary conditions;  for closed holomorphic curves the analysis also works with standard (i.e., domain-independent) almost complex structures.


\subsection{The space of leafwise maps with trivial holonomy}
\label{sec:leafwise_maps_trivial_holonomy}

Let $M$ be a smooth manifold of dimension~$m$ equipped with a codimension~$q$ foliation~$\cF$.  We assume that $M$ has no boundary.
We fix moreover a smooth foliated submanifold~$P\subset M$ of dimension $p = l+s$, where $l = \dim (P\cap \cF)$, which will serve as the boundary condition for the maps that we want to study.

Let now $\Sigma$ be a smooth connected $n$-manifold that is compact and that might have non-empty boundary. 
Although not always explicitly pointed out, we will use the following convention throughout the rest of this section:

\begin{convention}\label{con:boundary_curves}
    In the case where $\partial \Sigma \neq \emptyset$, we require $P\neq \emptyset$ and use it as boundary conditions for the maps $\Sigma\to M$ we will look at.
    If on the other hand, $\partial \Sigma = \emptyset$, then we also simply let $P$ be the empty manifold. 
    With this in mind, to unify notation we will look at maps of pairs $f\colon (\Sigma,\partial \Sigma)\to (M,P)$, whether $\partial \Sigma$ is empty or not.
\end{convention}

\begin{definition}\label{def:leafwise_map_trivial_holonomy}
    A continuous map $f\colon (\Sigma,\partial \Sigma)\to (M,P)$ is said to be \defin{leafwise} if its image is entirely contained in a leaf~$F$ of $\cF$.
    In this case, we also say that $f$ has \defin{trivial holonomy} if, for every $q\in \Sigma$,
    the image of the induced map $f_*\colon \pi_1(\Sigma,q)\to \pi_1(F,f(q))$ 
    lies in the kernel of the holonomy representation $\Hol_{f(q)}\colon \pi_1(F,f(q))\to \Diff(\R^q,0)$.
\end{definition}
For the definition of holonomy we invite the reader to consult \cite[Chapter~2]{CandelConlon}.

\medskip

We now look at maps $f\colon (\Sigma,\partial \Sigma) \to (M,P)$. 
The regularity of such maps throughout all this section is taken to be
\begin{itemize}
    \item either $C^k$ with $k\geq 2$;
    \item or $W^{k,p}$  for $k\geq 2$ and $(k-1)\,p > \dim N$.
\end{itemize}
With this assumption, all maps will be at least $C^1$; in the first case this is just by definition, while in the second case this is due to the Sobolev embedding theorem (see e.g. \cite[Theorem~B.1.11]{McDuffSalamonBook}).
To unify notation, we will call maps having the chosen regularity \defin{$\fS$-regular}.

We denote
\begin{itemize}
    \item the set of all $\fS$-regular maps $(\Sigma,\partial \Sigma)\to (M,P)$ by $\fS\bigl((\Sigma,\partial\Sigma),(M,P)\bigr)$;
    \item the subset of \emph{leafwise} maps by $\fS\bigl((\Sigma,\partial\Sigma),(\cF,P)\bigr)$;
    \item its subset made of those with trivial holonomy by $\fS_0\bigl((\Sigma,\partial\Sigma),(\cF,P)\bigr)$.
\end{itemize}
When there is no possible confusion, we will omit the domain~$\Sigma$ in the notation, and simply write $\fS(M,P)$, $\fS(\cF,P)$  and  $\fS_0(\cF,P)$ instead of $\fS\bigl((\Sigma,\partial\Sigma),(M,P)\bigr)$, $\fS\bigl((\Sigma,\partial\Sigma),(\cF,P)\bigr)$ and $\fS_0\bigl((\Sigma,\partial\Sigma),(\cF,P)\bigr)$ respectively.

The aim of this section is to prove the following:

\begin{theorem}\label{thm:leafwise_maps_trivial_hol_is_banach_submanifold}
    Assume that $\Sigma$ has connected (possibly empty) boundary or that the boundary condition~$P$ is a transverse submanifold.
    The space~$\fS_0(\cF,P)$ of leafwise maps with trivial holonomy is then a smooth Banach submanifold of $\fS(M,P)$.
\end{theorem}

\medskip

Let us now recall some known facts about $\fS(M,P)$ that we will need in the following.
First, a slight adaptation of the analogous result for $\fS(M)$ without boundary constraints in \cite[Section~5]{Eliasson} tells us that $\fS(M,P)$ is a smooth Banach manifold.
To prove this, it suffices to use an auxiliary Riemannian metric on $M$ for which $P$ is totally geodesic.
The total space of its tangent bundle $T\fS(M,P)$ is simply given by the Banach manifold $\fS(TM,TP)$ of $\fS$-regular maps $(\Sigma,\partial \Sigma) \to (TM,TP)$.

\smallskip

We wish to generalize the construction of the tangent bundle to other vector bundles.
Consider now instead of $(TM,TP)$ some other smooth real or complex vector bundle over $M$ of finite rank:
\begin{equation*}
    \pi\colon E\to M \, .
\end{equation*}
If $P\neq \emptyset$, we also suppose that there is a subbundle~$E_P\subset E\vert_P$ over $P$.

We can then consider the space $\fS(E,E_P)$ of $\fS$-regular maps $f\colon (\Sigma,\partial \Sigma) \to (E,E_P)$ (usually we write simply $\fS(E,P)$, when $E_P$ is the full bundle~$E\vert_P$), which is also a smooth Banach manifold.
In fact, $\pi$ induces a smooth map
\begin{equation*}
    \Pi\colon \fS(E,E_P) \to \fS(M,P), \, f\mapsto \pi\circ f   \, ,
\end{equation*}
and one can check using auxiliary connections on $E$ and $M$ as done in \cite[Section~5]{Eliasson} for the case of $TM\to M$,
that this makes $\fS(E,E_P)$ a Banach space bundle over $\fS(M,P)$.
More precisely, the fiber of $\fS(E,E_P)$ over an element $f\colon (\Sigma,\partial \Sigma)\to (M,P)$ of $\fS(M,P)$ is the space of all $\fS$-sections (i.e.\ sections of regularity $\fS$) of $f^*E$ over $\Sigma$ that take values in $f^*E_P$ over $\partial \Sigma$ (recall \Cref{con:boundary_curves}). 

\medskip

The following general fact will also be useful.
If $ \Pi\colon \cE\to \cB$ is a $C^s$-Banach vector bundle and $\cB_0$ is a Banach manifold, one can pullback $\Pi$ via any $C^s$-map $F\colon \cB_0 \to \cB$ and thus obtain $F^*\cE$ a $C^s$-Banach space vector bundle over $\cB_0$. 

In particular, since we will later apply differential operators to certain sections in a vector bundle, and since such operators obviously lower the regularity of the corresponding section, the pullback construction allows us to easily obtain suitable Banach vector bundles.
Denote $\fS' = C^{k'}$ in the case of $\fS = C^k$, and similarly $\fS' = W^{k',p}$ in the case $\fS=W^{k,p}$.
The inclusion  $\iota\colon \fS'(M,P) \hookrightarrow \fS(M,P)$ with $k' > k$ is a smooth map so that we can consider the Banach space bundle~$\iota^*\cE$ of $W^{k,p}$-sections over the Banach manifold of $W^{k',p}$-maps.

\bigskip

The strategy of the proof of \Cref{thm:leafwise_maps_trivial_hol_is_banach_submanifold} will consist in finding around any point $f\in \fS_0(\cF,P)$ a neighborhood $\cU_f \subset \fS(M,P)$ and a smooth map~$\Phi_f\colon \cU_f \to \cB$ into some Banach space~$\cB$ such that $0$ is a regular value of $\Phi_f$ and $\Phi_f^{-1}(0) = \fS_0(\cF,P)\cap\, \cU_f$.
The implicit function theorem provides then the desired result.

We will first develop a semi-local description of the foliation near the image of a leafwise map with trivial holonomy, obtained essentially thanks to Reeb stability.

\begin{lemma}\label{lemma:holonomy}
    Let $X$ be a manifold that carries a smooth codimension~$q$ foliation~$\cF$.
    Assume we are in the following situation:
    \begin{itemize}
        \item $Y \subset X$ is a connected compact submanifold that is contained in a leaf~$F_0$ of $\cF$, and the holonomy of $\cF$ along any loop in $Y$ is trivial.
        \item $Z\subset X$ is a foliated submanifold (that might be empty).
    \end{itemize}

    Then, we can choose an open neighborhood~$U_0$ of $Y$ in $F_0$, and an open neighborhood~$U$ of $Y$ in $X$ with $U\cap F_0 = U_0$ with the following properties.
    \begin{enumerate}
        \item\label{item:trivial_hol_lemma_holonomy}  For every base point~$x_0\in U_0$ there exists a diffeomorphism
        \begin{equation*}
            \psi \colon U_0 \times D^q_{x_0} \to U \;, 
        \end{equation*}
        where $D^q_{x_0} \subset T_{x_0} M$ is a very small disc centered at the origin that is transverse to $T_{x_0}F_0=T_{x_0}\cF$, such that:
        \begin{enumerate}
            \item $\psi^*\cF$ is the (smooth) foliation whose leaves are the slices $U_0\times\{v\}$ with  $v\in D^q_{x_0}$; 
            \smallskip
            \item the restriction of $\psi$ to $U_0\times\{0\}$ coincides with the inclusion of $U_0$ in $X$.
        \end{enumerate}
        
      \item\label{item:submanifold_lemma_holonomy} If $Z_0 := Z\cap U_0 \ne \emptyset$, let $s$ be the codimension of the induced foliation~$\cF_Z := \cF\cap TZ$ on $Z$.
      
      Then choosing the base point $x_0$ in $Z_0$ and denoting the connected component of $Z_0$ that contains $x_0$ by $Z_0(x_0)$, we can arrange $\psi$ in such a way that its restriction to $Z_0(x_0)\times D^s_{x_0}$ 
      is a diffeomorphism onto the connected component of $Z\cap U$ containing $Z_0(x_0)$.
      Here $D^s_{x_0} \subset D^q_{x_0} \cap T_{x_0}Z$
      is a small disc centered at the origin that is transverse to $T_{x_0}\cF_Z$.
      
      \item\label{item:change_of_base_point} If we choose in item~\eqref{item:trivial_hol_lemma_holonomy} two different base points~$x_0$ and $x_1 \in U_0$ to construct the maps 
      $\psi_0 \colon U_0 \times D^q_{x_0} \to U$ and $\psi_1 \colon U_0 \times D^q_{x_1} \to U$ 
      respectively, then there is a diffeomorphism 
      $\Phi\colon D^q_{x_1} \to D^q_{x_0}$ 
      keeping the origin fixed such that $\psi_1\bigl(x,v\bigr) = \psi_0(x,\Phi(v))$ for every $x\in U_0$ and every 
      $v\in D^q_{x_1}$.
    \end{enumerate}
\end{lemma}
\begin{proof}
    The first part of the statement is essentially a consequence of the Reeb stability theorem that implies that any foliation that has a leaf with trivial holonomy will be semi-locally trivial near that leaf.
    Here are some additional details.

    Let $U_0$ be a tubular neighborhood of $Y$ inside $F_0$ such that $U_0$ deformation retracts onto $Y$.
    We will possibly need to shrink the size of $U_0$ during the proof.

    Let us start by recalling more precisely how holonomy is defined, i.e.\ how parallel transport can be performed using the foliation.
    Choose a Riemannian metric~$h$ on $X$, and let $\nu\subset TX\vert_{U_0}$ be a distribution along the neighborhood~$U_0$ that is everywhere transverse to the leaf~$F_0$. (To deal with item~\eqref{item:submanifold_lemma_holonomy}, we will later make a more careful choice of $h$ and $\nu$; see below for the details.)
    After possibly shrinking $U_0$, we can find an $\epsilon>0$ such that the geodesic exponential map restricted to the subset $\nu_{<\epsilon}$ of $\nu$ made of vectors of norm at most $\epsilon$ is a diffeomorphism onto its image.
    (In the rest of the proof, $\epsilon>0$ is allowed to shrink as needed, namely all vectors of $\nu$ that we use from now on are assumed to be sufficiently small for the required properties to be satisfied.)
    
    Choose now a point~$x_0\in U_0$ as a reference point.
    Our aim is to parallel transport $\D^q_{<\epsilon}(x_0) \coloneqq (\nu_{<\epsilon})_{x_0} \subset \nu_{x_0}$ with respect to the foliation to any other point~$x$ in $U_0$.
    Choose a path $\delta\colon [0,1]\to U_0$ connecting $x_0$ to $x$.
    The pull-back of $\nu$ along $\delta$ is a (trivial) bundle $\delta^*\nu \cong [0,1]\times \R^q$.
    Using that the geodesic exponential map $(\tau,w)\in [0,1]\times \D^q_{\epsilon} \mapsto \exp_{\delta(\tau)} (w)$ is transverse to $\cF$ (at least after shrinking $\epsilon$), we can pull-back $\cF$ to $\delta^*\nu_{<\epsilon}$, thus obtaining a $1$-dimensional foliation~$\cF_\delta$ on $\delta^* \nu_{<\epsilon}$ such that the $0$-section corresponds to the path~$\delta$ in $U_0$.
    Consider now the leaf of $\cF_\delta$ passing through the point~$(0,v) \in (\delta^*\nu)_{0}= \nu_{x_0}$ with $\|v\|<\epsilon$, which
    will be approximately parallel to the $0$-section, and will thus end up at a certain point $(1,w) \in \delta^*\nu$.
    Depending on the context, the parallel transport along $\delta$ with respect to $\cF$ will mean either the transport of $v\in \nu_{x_0}$ yielding the vector~$w\in \nu_x$ or the transport of the point $\exp_{x_0}(v) \in X$ yielding the point $\exp_{x}(w) \in X$.
    After possibly reducing the size of $\epsilon>0$ and the one of $U_0$, there is for every point~$x$ in $U_0$ a path~$\delta$ connecting $x_0$ to $x$ such that the parallel transport of $\nu_{<\epsilon}$ along $\delta$ is defined.

    Our next aim will be to show that in our case the parallel transport only depends on the end point~$x$ of $\delta$ but not on the choice of $\delta$ itself.
    Recall that parallel transport along two homotopic paths provides identical results, as one can perform the construction above not only along arcs but also along homotopies.
    Denoting $\widetilde U_0$ the universal cover of $U_0$, which we see as the space of homotopy classes $[\delta]$ of arcs $\delta\subset U_0$ starting at $x_0$, there is a well-defined smooth map $\widetilde U_0\times \D^q_{\epsilon}(x_0) \to X$ that is locally a diffeomorphism associating to each $([\delta],v)$ the parallel transport of $v$ from $x_0 = \delta(0)$ to $x = \delta(1)$ which is then mapped via $\exp_x$ into $X$.
    Moreover, since $U_0$ deformation retracts onto $Y$ and since the inclusion $i\colon Y \hookrightarrow X$ has trivial holonomy, $U_0\hookrightarrow X$ does so too.
    The local diffeomorphism $\widetilde U_0\times \D^q_{<\epsilon}(x_0)\to X$ passes down to the quotient providing us with the smooth embedding $U_0\times \D^q_{<\epsilon}(x_0) \hookrightarrow \exp(\nu)$ that is the desired map~$\psi$.
    (In the statement, we simply choose $D^q_{x_0}\coloneqq\D^q_{<\epsilon}(x_0)$.)

    \medskip
    
    Item~\eqref{item:change_of_base_point}  follows directly from this construction.
    Namely, if $x_1\in U_0$ is a different base point, then choose for any $x\in U_0$ a path~$\delta$ connecting $x_1$ to $x$ by passing through $x_0$. 
    The parallel transport of $\cF$ along this path is the composition of the parallel transport along the path from $x_1$ to $x_0$ and of the parallel transport from $x_0$ to $x$.
    If we let $\Phi$ be the parallel transport from $x_1$ to $x_0$ then we obtain the desired claim.
    (Analogously to before, we choose $D^q_{x_1}\coloneqq \D^q_{<\epsilon}(x_1)$ in the statement.)

    \medskip

    Let us now discuss the remaining point~\eqref{item:submanifold_lemma_holonomy}.
    If $Z_0 = Z\cap U_0 \ne \emptyset$, we choose the Riemannian metric~$h$ on $X$ in such a way that $Z$ is totally geodesic.
    We also need to be more precise about the choice of the distribution~$\nu$.
    Let $\mu\subset TZ\vert_{Z_0}$ be a distribution in $Z$ along $Z_0$ that is everywhere transverse to $\cF_Z$.
    By our assumption, the rank of $\mu$ is $s$.
    To obtain a suitable~$\nu$, extend $\mu$ first to a distribution that is transverse to $F_0$ (and hence of rank~$q$) over each point $Z_0$;
    then extend this further to the desired distribution~$\nu$ over all of $U_0$ in such a way that $\nu$ is transverse to $F_0$.
    So, to summarize, we have a smooth $q$-dimensional distribution~$\nu$ along $U_0$, that is transverse to $\cF$, and an $s$-dimensional sub-distribution~$\mu \subset TZ$ along $Z_0$ that is transverse to $\cF_Z$ inside $Z$.
    
    Choose now the base point~$x_0$ in $Z_0$ and denote by $Z_0(x_0)$ the component of $Z_0$ containing $x_0$.
    We show now that the parallel transport of the foliation~$\cF_Z$ in $Z$ along $Z_0(x_0)$ is simply the restriction of the parallel transport of $\cF$.
    Let $\D_{<\epsilon}^s(x_0)$ be the intersection of $\D_{<\epsilon}^q(x_0)$ with $TZ$.
    As before we obtain the map~$\psi$ that trivializes the ambient foliation~$\cF$ along $U_0$, and it only remains to understand how $\psi$ behaves when restricted to $Z_0(x_0)\times \D_{<\epsilon}^s(x_0)$.
    Let $x$ be any point in $Z_0(x_0)$, and choose a path~$\delta$ in $Z_0(x_0)$ joining $x_0$ to $x$.
    The pull-back bundle $\delta^*\nu \cong [0,1]\times \R^q$ contains the subbundle~$\delta^*\mu$ that agrees at $0$ with $\mu_{x_0}$.
    The geodesic exponential map $(\tau,v) \in \delta^*\nu \mapsto \exp_{\delta(\tau)}(v)$ sends $\delta^*\mu$ into $Z$, because $Z$ is totally geodesic.
    In particular, we find that the induced foliation~$\cF_\mu$ on $\mu$ is $\exp^*\cF_Z = \exp^*(\cF\cap TZ) = (\exp^*\cF) \cap \mu = \cF_\delta \cap \mu$.
    Both $\cF_\mu$ and $\cF_\delta$ are $1$-dimensional, and thus the parallel transport of any vector $v\in \mu_{x_0}$ yields then at $\delta(1)$ the same vector~$w$ -- independently of whether the parallel transport is performed with respect to $\cF_\mu$ in $\delta^*\mu$ or with respect to $\cF_\delta$ inside $\delta^*\nu$.

    This proves that the restriction of $\psi$ to $Z_0(x_0)\times \D_{<\epsilon}^s(x_0)$ is a diffeomorphism onto the connected component of $Z\cap U$ containing $Z_0(x_0)$.
    Denoting $D^s_{x_0}\coloneqq \D_{<\epsilon}^s(x_0)$, this concludes the proof.
\end{proof}

By applying this lemma in our situation, we obtain a semi-local description of the foliation near a given trivial-holonomy map.

\begin{corollary}\label{corollary:trivial foliation close to graph}
    Let $g\in \fS_0(\cF,P)$ be a map.
    Denote the leaf containing $g(\Sigma)$ by $F_0$.
   
    Consider on $\Sigma\times M$ the induced foliation $T\Sigma\oplus \cF$, and let
    \begin{equation*}
        \Gamma_g\colon \Sigma \to \Sigma\times M, \; x\mapsto \bigl(x,g(x)\bigr)
    \end{equation*}
    be the graph of $g$, that is hence naturally a leafwise map.
    
    We can then find a neighborhood~$U_0$ of $\Gamma_g(\Sigma)$ in the leaf $\Sigma\times F_0$, a neighborhood~$U$ of $\Gamma_g(\Sigma)$ in $\Sigma\times M$, and a diffeomorphism $\Psi_g\colon U_0\times \R^q\to U$ such that the following properties hold (recall in particular \Cref{con:boundary_curves}).
    \begin{enumerate}
        \item $\Psi_g^*(T\Sigma\oplus \cF)$ is the (smooth) foliation whose leaves are the horizontal slices $U_0\times\{x\}$ with $x \in \R^q$.
        \item The restriction of $\Psi_g$ to $U_0\times\{0\}$ is the natural inclusion of $U_0$ in $U$.
        \item Let $(\partial \Sigma\times P)_0$ be one of the connected components of $(\partial \Sigma\times P) \cap U_0$, then the image of $(\partial \Sigma\times P)_0 \times \bigl(\R^s\oplus\{0\}\bigr)$ under $\Psi_g$ is the component of $(\partial \Sigma\times P)\cap U$ containing $(\partial \Sigma\times P)_0$.
        \item Let $(\partial \Sigma\times P)_0'$ be a different connected component of $(\partial \Sigma\times P) \cap U_0$.
        Then there is a diffeomorphism $\Phi\colon \R^q \to \R^q$ keeping the origin fixed such that the the diffeomorphisms~$\Psi_g$ and $\Psi_g'$ adapted to $(\partial \Sigma\times P)_0$ and $(\partial \Sigma\times P)_0'$ respectively are related by $\Psi_g'(x,v) = \Psi_g\bigl(x,\Phi(v)\bigr)$ for every $x\in U_0$ and every $v\in \R^q$.
    \end{enumerate}
\end{corollary}
\begin{proof}
    Up to reparametrizing the pair $(D^q_{x_0},D^s_{x_0})$ as $(\R^q,\R^s)$, this is a direct corollary of \Cref{lemma:holonomy}, with the following choices.
    The manifold~$X$ is $\Sigma\times M$, and the foliation is $T\Sigma\oplus \cF$.
    The compact submanifold~$Y$ is the image of the the graph~$\Gamma_g$ which by our choice of $\fS$ is at least a $C^1$-embedding. 
    For the submanifold~$Z$ we use $\partial \Sigma\times P$ and for $Z_0$ we use the connected component $(\partial \Sigma\times P)_0$.
\end{proof}

In the proof of \Cref{thm:leafwise_maps_trivial_hol_is_banach_submanifold}, we will use the following adapted version of the Poincaré inequality to obtain necessary $\fS$-norm bounds.
(Recall that throughout all this section we assumed $k\geq 2$ and $p>1$.)

\begin{lemma}[Poincaré inequality]\label{lemma:poincare_inequality}
    For every compact manifold~$\Sigma$ with (possibly empty) boundary, and for a regularity class~$\fS$ that is either $W^{k,p}$ or $C^k$, we find a constant~$C = C(\Sigma,\fS) >0$ such that every $\fS$-regular function $f\colon \Sigma \to \R$ vanishing somewhere on $\Sigma$ satisfies the inequality:
    \begin{equation}\label{eqn:gaffney}
        \|f\|_{\fS}\leq C \, \|\d f\|_{\fS^{-1}} \; .
    \end{equation}
\end{lemma}

Here, $\fS^{-1}$ is just a short-hand notation for $W^{k-1,p}$ in the case $\fS=W^{k,p}$ and for $C^{k-1}$ in the case $\fS=C^k$.
Note also that by our assumption of $k\geq 2$, $f$ is $C^1$ and $\d f$ is its honest differential.

\begin{proof}[Proof of \Cref{lemma:poincare_inequality}]
    Suppose by contradiction that there is a sequence $(f_j)_{j\geq 1}$ of $\fS$-functions such that each  $f_j$ vanishes at some point $p_j\in \Sigma$ but such that     
    \begin{equation}
    \label{eqn:first_bound_dfj}
        \|\d f_j\|_{\fS^{-1}} < \frac{1}{j} \| f_j\|_{\fS} \; .
    \end{equation}
    Passing to a subsequence we can assume that $(p_j)_j$ converges to a point $p_\infty \in \Sigma$.
    
    After rescaling by a constant, we can furthermore assume that $\|f_j\|_{\fS^{-k}} = 1$ for all $j$, where $\fS^{-k}$ denotes the $L^p$-norm if $\fS=W^{k,p}$, and the $C^0$-norm if $\fS = C^k$.
    Indeed, the inequality above is homogeneous, and thus preserved by rescaling, and clearly, $f_j(p_j) = 0$ also holds after rescaling the function by a constant.
    
    Then, using \eqref{eqn:first_bound_dfj}, the fact that (by definition) $\|f_j\|_{\fS} = \|f_j\|_{\fS^{-k}} + \|\d f_j\|_{\fS^{-1}}$ implies that 
    \begin{equation}
        \label{eqn:bound_norm_fj}
        \|f_j\|_{\fS} \leq \frac{j \| f_j\|_{\fS^{-k}}}{j- 1} \leq 2 \;,
    \end{equation}
    and hence that
    \begin{equation}
        \label{eqn:bound_norm_dfj}
        \|\d f_j \|_{\fS^{-1}} \leq \frac{ \| f_j\|_{\fS^{-k}}}{j-1}\leq \frac{2}{j}
        \quad\text{for $j>1$.}
    \end{equation}

    Now, recall that for maps on compact domains the embedding $\fS \to \fS^{-1}$ is compact:
    if $\fS = C^k$, this is just the Arzelà-Ascoli theorem; if $\fS = W^{k,p}$, then this is the content of the Kondrachov embedding theorem (note that $p^*\coloneqq np/(n-p)>p$).
    Hence using $\eqref{eqn:bound_norm_fj}$ and after passing to a subsequence, $f_j$ converges in $\fS^{-1}$-norm to a certain $f_\infty$.
    In particular, $\|f_\infty\|_{\fS^{-k}} = 1$.  
    Moreover, \Cref{eqn:bound_norm_dfj} implies that $\d f_\infty = 0$ so that $f_\infty$ is a constant function that is different from $0$, because $\|f_\infty\|_{\fS^{-k}} = 1$.

    Note however that the $\fS^{-1}$-convergence implies that $f_j$ converges uniformly to $f_\infty$, so that in particular $f_\infty$ will vanish at $p_\infty$.
    Since $f_\infty$ is a constant function, $f_\infty = 0$ everywhere, in contradiction to the previous claim.
    This concludes the proof.
\end{proof}

In order to apply the implicit function theorem for Banach spaces later on, we have to make sure that the target space of the considered map is indeed a Banach space.

\begin{corollary}\label{exact forms form Banach space}
    Let $\Sigma$ be a compact manifold that might have non-empty boundary.
    Let $\fS^{-1}\bigl(\Omega^1(\Sigma)\bigr)$ be the Banach space of $1$-forms with $\fS^{-1}$-regularity.
    The linear subspace of $\fS^{-1}\bigl(\Omega^1(\Sigma)\bigr)$ made of exact $1$-forms 
    \begin{equation*}
        \cB^{-1}\bigl(\Omega^1(\Sigma)\bigr) = \bigl\{\d g \in \fS^{-1}(\Omega^1(\Sigma))\bigm|\; g\colon \Sigma \to \R\text{ is $\fS$-regular}\bigr\}
    \end{equation*}
    and the one made of exact $1$-forms having a primitive that vanishes along the boundary
    \begin{equation*}
        \cB_0^{-1}\bigl(\Omega^1(\Sigma)\bigr) = \bigl\{\d g \in \fS^{-1}(\Omega^1(\Sigma))\bigm|\; g\colon \Sigma \to \R \text{ is $\fS$-regular with } g\vert_{\partial \Sigma} = 0\bigr\}
    \end{equation*}
    are both closed, and thus in particular they are Banach spaces.
\end{corollary}
\begin{proof}
    Clearly both these subsets are linear subspace of $\fS^{-1}\bigl(\Omega^1(\Sigma)\bigr)$.  We only need show that they are closed.

    Let $\alpha$ be an element in the closure of $\cB^{-1}\bigl(\Omega^1(\Sigma)\bigr)$, then there exists a sequence $(g_j)_j$ of real-valued functions of regularity $\fS$ on $\Sigma$ such that $\d g_j \to \alpha$ in $\fS^{-1}\bigl(\Omega^1(\Sigma)\bigr)$.
    Choose a point $p\in \Sigma$.
    Then, up to replacing each of the $g_j$ by $g_j - g_j(p)$, we can assume that all the $g_j$'s vanish at $p$;
    note that, since adding a constant to a function does not change its exterior derivative, we still have that $(\d g_j)_j$ converges to $\alpha$.

    We want to prove that $(g_j)_j$ converges in the space $\fS(\R)$ of $\fS$-regular functions $\Sigma\to\R$ to a function~$g_\infty$ such that $\alpha = \d g_\infty$, thus showing that $\alpha\in \cB^{-1}\bigl(\Omega^1(\Sigma)\bigr)$.
    \Cref{lemma:poincare_inequality} provides us with a constant~$C$ such that $\|g_j - g_{j'}\|_{\fS} \le C\,\|\d g_j - \d g_{j'}\|_{\fS^{-1}}$. 
    This implies that $(g_j)_j$ is a Cauchy and hence also a converging sequence, thus giving the claim.
    
    The proof that $\cB_0^{-1}\bigl(\Omega^1(\Sigma)\bigr)$ is closed is almost identical using \Cref{lemma:poincare_inequality} for functions
    vanishing along the boundary, and we omit it here.
    The only detail to add is that, since $\fS$-convergence implies uniform convergence, we find a primitive for $\alpha$ that vanishes along the boundary.
\end{proof}

We are now ready to give a proof of the main result of this section.
A note on convention: to simplify the exposition, when we need to talk about properties that certain maps defined on $\Sigma$ satisfy along $\partial \Sigma$, we will not treat explicitly the case of $\partial \Sigma = \emptyset$, in which case we implicitly intend the described conditions to be trivially satisfied.

\begin{proof}[Proof of \Cref{thm:leafwise_maps_trivial_hol_is_banach_submanifold}]
    Let $f_0\in \fS_0(\cF,P)$ be a map, and let $\Gamma_{f_0}$ be its graph in $\Sigma\times M$.
    Let $\Psi_{f_0}\colon U_0\times \R^q \to U\subset \Sigma\times M$ be the smooth map obtained in \Cref{corollary:trivial foliation close to graph}.
    Note that $f_0$ factors as follows:
    \begin{equation*}
        \begin{tikzcd}
            \Sigma \arrow[rr, "\Gamma_{f_0}'"] \arrow[rrrr, "f_0"', bend right] \arrow[rrr, "\Gamma_{f_0}", bend left] &  & U_0\times\R^q \arrow[r, "\Psi_{f_0}"'] & U\subset \Sigma\times M \arrow[r, "\pi_M"] & M
        \end{tikzcd}
    \end{equation*}
    where $\pi_M\colon \Sigma\times M\to M$ is the obvious projection and $\Gamma_{f_0}'$ is just the composition $\Psi_{f_0}^{-1}\circ\Gamma_{f_0}$.

    The graph~$\Gamma_f$ of every map $f\colon \Sigma \to M$ that is $C^0$-close to $f_0$ lies in $U$.
    Choose an open neighborhood $\cU_{f_0} \subset \fS(M,P)$ of $f_0$ such that $\Gamma_f(\Sigma) \subset U$ for every $f\in \cU_{f_0}$.
    We can thus associate to every $f\in \cU_{f_0}$ a map $\Gamma_f' := \Psi_{f_0}^{-1}\circ \Gamma_f\colon \Sigma \to U_0\times \R^q$ such that $f$ factors according to the diagram above, and such a $\Gamma_f'$ is of regularity~$\fS$ as composition of a smooth map with an $\fS$-regular map is $\fS$-regular, see for instance \cite[Proposition~B.1.20]{McDuffSalamonBook}.
    In fact, according e.g.\ to \cite[Lemma 2.98]{WenNotes}, post-composition with a smooth function is a smooth map between Banach spaces of $\fS$-regular maps, so that $\Gamma' \colon f\mapsto \Psi_{f_0}^{-1}\circ \Gamma_f$ is a diffeomorphism between $\cU_{f_0}$ and the space of $\fS$-regular maps $\Sigma \to U_0\times \R^q$ with a suitable boundary condition that we will describe now.  

    Since the boundary of $\Sigma$ is mapped by $f$ into $P$, $\partial \Sigma$ is mapped by $\Gamma_f$ into $\partial \Sigma\times P$.
    Consider the following two cases:

    \textbf{(i)} If $\partial\Sigma$ is connected (possibly empty), then let $(\partial \Sigma\times P)_0$ be the connected component of $\bigl(\partial \Sigma\times P\bigr)\cap U_0$ containing $\Gamma_{f_0}(\partial\Sigma)$.
    By the properties of $\Psi_{f_0}$ it then follows that $\Gamma_f'(\partial \Sigma)$ lies in $(\partial \Sigma\times P)_0\times \R^s$ with the notation of \Cref{corollary:trivial foliation close to graph}, where we consider $\R^s$ as $\R^s\times\{0\}\subset \R^q$.
    
    \textbf{(ii)} If $P$ is transverse to $\cF$, then the induced foliation~$\cF_P$ will be of codimension~$s=q$.
    It follows that $\bigl(\partial \Sigma\times P\bigr)\cap U$ agrees with the image of $\bigl(\bigl(\partial \Sigma\times P\bigr)\cap U_0\bigr)\times \R^s$  under $\Psi_{f_0}$, i.e.\ the desired form given in \Cref{corollary:trivial foliation close to graph}.(iii) holds for all components of $\bigl(\partial \Sigma\times P\bigr)\cap U_0$ simultaneously.

    \smallskip

    The map~$\Gamma_f'$ takes thus in both cases the form $x \mapsto \bigl(x,\hat f(x), v_1(x),\dotsc,v_q(x)\bigr)$, where the $v_{s+1}, \dotsc,v_q$ vanish along $\partial \Sigma$.
    It is clear that the image of $f$ lies in a leaf of the foliation if and only if all the $v_j$ for $j\in\{1,\dotsc,q\}$ are constant.
    Equivalently, we can define a smooth function of smooth Banach manifolds
    \begin{equation}
    \label{eqn:implicit_function}
        F\colon \cU_{f_0} \to \fS^{-1}\bigl(\Omega^1(\Sigma)\bigr) \times \dotsm \times \fS^{-1}\bigl(\Omega^1(\Sigma)\bigr), f\mapsto \bigl(\d v_1,\dotsc,\d v_q\bigr)
    \end{equation}
    where each of the factors $\fS^{-1}\bigl(\Omega^1(\Sigma)\bigr)$ denotes the space of $1$-forms on $\Sigma$ of class~$\fS^{-1}$, and
    \begin{equation*}
        \fS(\cF,P) \cap \, \cU_{f_0} = F^{-1}(0,\dotsc,0) \;.
    \end{equation*}
    
    \medskip

    It remains to apply the implicit function theorem.
    Note that the image of $F$ consists only of "exact" forms.  In particular, to have any hope of establishing the transversality of $F$ at $(0,\dotsc,0)$, it is necessary to reduce the target space of the map~$F$ to a suitable subspace.
    We will show that the correct target space to use is
    \begin{equation*}
        \underbrace{\cB^{-1}\bigl(\Omega^1(\Sigma)\bigr) \times \dotsm \times \cB^{-1}\bigl(\Omega^1(\Sigma)\bigr)}_{\text{$s$ times}}
        \times 
        \underbrace{\cB_0^{-1}\bigl(\Omega^1(\Sigma)\bigr) \times \dotsm \times \cB_0^{-1}\bigl(\Omega^1(\Sigma)\bigr)}_{\text{$q-s$ times}} \;,
    \end{equation*}
     where $\cB^{-1}\bigl(\Omega^1(\Sigma)\bigr)$ and $\cB_0^{-1}\bigl(\Omega^1(\Sigma)\bigr)$ are the Banach spaces introduced in \Cref{exact forms form Banach space} above.
    
    \medskip
    
    We now linearize the map~$F$ at the element~$f_0$.
    Consider a smooth path of maps $f_t \subset \fS(M,P)$ for $t\in (-\epsilon,\epsilon)$ passing through $f_0$, and denote its tangent vector at $t=0$ by $\dot f_0$.
    The linearization of $F$ at $f_0$ satisfies
    \begin{equation*}
        \frac{d}{dt}\Bigr\vert_{t=0}  F(f_t) = DF_{f_0}\cdot \dot f_0 \;.
    \end{equation*}
    Recall that $\Gamma'$ is a diffemorphism, 
    so writing in components $\Gamma'_{f_t} = \bigl(\Id_\Sigma,\hat f_t, v_{1,t}\, ,\dotsc, v_{q,t}\bigr)$
    we have (by definition of $F$) that
    $DF_{f_0}\cdot \dot f_0 = \bigl(\d \dot v_1,\dotsc, \d \dot v_q\bigr)$, where $\d \dot v_j = \frac{d}{dt}\bigr\vert_{t=0} \d v_{j,t}$.
    
    To see now that any element $\bigl(\alpha_1, \dotsc, \alpha_s, \alpha_{s+1}, \dotsc, \alpha_q\bigr)$, for arbitrary $\alpha_j \in \cB^{-1}\bigl(\Omega^1(\Sigma)\bigr)$ if $j\le s$ and $\alpha_j \in \cB_0^{-1}\bigl(\Omega^1(\Sigma)\bigr)$ if $j> s$, lies in the image of $DF_{t_0}$, choose functions $g_1,\dotsc, g_q\colon \Sigma \to \R$ of regularity~$\fS$ such that $g_j\vert_{\partial \Sigma} = 0$ if $j>s$, and such that $\d g_j = \alpha_j$.
    
    We can then define a path $\Gamma_{f_t}' = \bigl(\Id_\Sigma, f_0, t\,g_1,\dotsc, t\,g_q\bigr)$.
    This path corresponds to $f_t := \pi_M\circ\Psi_{f_0}\circ \Gamma_{f_t}'$ lying in $\cU_{f_0}\subset \fS(M,P)$, see also the diagram above.
    In particular, all $f_t$ respect the boundary condition $f_t(\partial \Sigma) \subset P$.
    
    It follows then directly that $DF_{f_0}\cdot \dot f_0 = (\d g_1,\dotsc,\d g_q) = (\alpha_1,\dotsc,\alpha_q)$ so that $DF_{f_0}$ is surjective.
    In particular, after possibly reducing the domain~$\cU_{f_0}$, $0$ is a regular value of $F$ and $F^{-1}(0) = \fS(N,\cF)\cap \cU_{f_0}$ is a smooth Banach manifold. 
    
    \medskip

    We still need to prove the every $f$ in $\fS(N,\cF)\cap \cU_{f_0}$ has trivial holonomy; together with the statement that we have just shown, it will then follow that $\fS_0(N,\cF)\cap \cU_{f_0} = \fS(N,\cF)\cap \cU_{f_0}$ is a Banach manifold, thus concluding the proof.
    Note that $\Gamma_f'$ is of the form $\bigl(\Id_\Sigma, f, c_1,\dotsc,c_s,0,\dotsc,0\bigr)$ inside $U_0\times \R^q$ where the foliation is given by leaves corresponding to the slices with constant $\R^q$-value.
    Clearly, $\Gamma_f'$ has thus trivial holonomy, and $\Psi_{f_0}$ being a diffeomorphism the same holds for $\Gamma_f$.

    In order to prove that a given $f$ in $\fS(N,\cF)\cap \cU_{f_0}$ has trivial holonomy, it is enough to interpret its holonomy as that of its graph, which we know has image in the open set $U$ from \Cref{corollary:trivial foliation close to graph} and hence has trivial holonomy; here are additional details.
    \\
    Let now $\gamma\subset \Sigma$ be a closed loop, and consider the parallel transport of $\cF$ along $f\circ\gamma$ (as described in the proof of \Cref{lemma:holonomy}) by choosing a normal bundle~$\nu$ that is transverse to the foliation~$\cF$ so that the total space of $\nu$ inherits on a neighborhood of the $0$-section an induced foliation that is transverse to the fibers.
    The pull-back bundle $(f\circ\gamma)^*\nu$ has then a $1$-dimensional foliation that is approximately parallel to the $0$-section, and which allows us to transport the fiber from $f(\gamma(0))$ to $f(\gamma(1))$ defining the parallel transport along $f\circ \gamma$.
    \\
    The bundle~$\nu$ can be lifted to $\Sigma\times M$ where it will be transverse to the product foliation~$\Sigma \oplus \cF$.
    Since $\Gamma_f$ lies in our model neighborhood $U$, the parallel transport along $\Gamma_f\circ \gamma$ will be trivial, but note that the construction of parallel transport agrees with the one performed for $f\circ \gamma$.
    This shows that $f$ has trivial holonomy, thus concluding.
\end{proof}

    \medskip
    
\begin{remark}\label{rmk:tangent_bdle_to_space_maps}
    Let $f_0$ be an element in $\fS_0(\cF,P)$.
    The implicit function theorem used in the proof of \Cref{thm:leafwise_maps_trivial_hol_is_banach_submanifold} allows us to identify the tangent space at any $f\in \fS_0(\cF,P)\cap \cU_{f_0}$ with the kernel of $DF_f$, the linearization of the map~$F$ in \eqref{eqn:implicit_function} at the point $f$.
    \\
    More precisely, recall that the tangent space $T_f\fS(M,P)$ is naturally identified with the space of $\fS$-sections of $f^*TM$ over $\Sigma$.
    Now, we also have an isomorphism of vector bundles $\bigl(D\Psi_{f_0}\bigr)^{-1}\colon TM\vert_{\Psi_{f_0}(U_0\times\R^q)}\xrightarrow{\cong} T(U_0\times\R^q)\cong TU_0\times T\R^q$ in the case where $\partial \Sigma = \emptyset$ (and $P=\emptyset$) and $\bigl(D\Psi_{f_0}\bigr)^{-1}\colon TM\vert_{\Psi_{f_0}(U_0\times\R^s)}\xrightarrow{\cong} T(U_0\times\R^s)\cong TU_0\times T\R^s$ in the case where $\partial \Sigma \neq \emptyset$ (and $P\neq \emptyset$ is a foliated submanifold with $\cF_P$ having rank $s$).
    Hence, we can identify $T_f\fS_0(\Sigma,\cF)$ with the set made of those $\fS$-sections $X$ of $f^*TM$ such that $\bigl(D\Psi_{f_0}\bigr)^{-1}(X)$ have constant components along the fiber direction of $T\R^q = \R^q\times\R^q_{\mathrm{fiber}}$ and, if $\partial N \neq \emptyset$, the $\R^q_{\mathrm{fiber}}$-component of $\bigl(D\Psi_{f_0}\bigr)^{-1}(X)$ has values in $\R^s\times\{0\}\subset \R^q_{\mathrm{fiber}}$.
    From this presentation, it is clear that this space can be naturally identified with the pairs~$(X',v)$, where $X'$ is an $\fS^{-1}$-section in $f^*\cF$ over $\Sigma$,
    and $v$ is an element of $\R^s$ or $\R^q$ depending whether $\Sigma$ has boundary or not.
    In other words, denoting by $L_f$ the leaf of $\cF$ containing the image of $f$, we then have a natural identification
    \begin{equation}
    \label{eqn:splitting_tangent_bundle_foliated_maps}
        T_f\fS_0(\Sigma,\cF) \simeq 
        \begin{cases}
            \fS(f^*TL_f) \oplus \R^q & \text{if } \partial \Sigma = \emptyset \\
            \fS(f^*TL_f) \oplus \R^s & \text{if } \partial \Sigma \neq  \emptyset
        \end{cases}
    \end{equation}
\end{remark}

\begin{remark}\label{rmk:foliation_space_foliated_maps}
    It is not hard to see that the foliation~$\cF$ on $M$ induces a foliation $\widetilde{\cF}$ on $\fS_0(\Sigma,\cF)$ with leaves of finite codimension, which is either equal to the codimension~$q$ of $\cF$ in $M$ if $\partial \Sigma = \emptyset$, or to the codimension~$s$ of the foliation~$\cF_P$ on the submanifold~$P$ serving as boundary condition in the case $\partial \Sigma \neq \emptyset$.    \\
    Indeed, in the proof of \Cref{thm:leafwise_maps_trivial_hol_is_banach_submanifold}, we found for every $f\in \fS_0(\Sigma,\cF)$ an open neighborhood~$\cU_f$ of $f$ in $\fS(M,P)$ and a submersion $\cU_f \cap \fS_0(\Sigma,\cF) \to \R^q$ or $\cU_f \cap \fS_0(\Sigma,\cF) \to \R^s$.
    The level sets of this map defines a foliation on $\cU_f \cap \fS_0(\Sigma,\cF)$, and each plaque in this chart consists of those maps that lie with respect to the chosen Reeb chart in \Cref{corollary:trivial foliation close to graph} in the same plaque of $\cF$.
    This implies that if $g\in \fS_0(\Sigma,\cF)$ lies close to $f$ such that $\cU_f \cap \fS_0(\Sigma,\cF)$ and $\cU_g \cap \fS_0(\Sigma,\cF)$ overlap, then the two local foliations will patch up well.
    We obtain thus a globally defined foliation on $\fS_0(\Sigma,\cF)$.
    
    Lastly, note that if $\cF$ on $M$ is cooriented, the Reeb chart above can be constructed in such a way that the orientation of the $\R^q$, resp.\ $\R^s$, factor coincides with the given coorientation of $\cF$.
    In this case, $\widetilde\cF$ on $\fS_0(\Sigma,\cF)$ will then also be naturally equipped with a coorientation.
    
\end{remark}


\subsection{Pseudo-holomorphic discs and linearized Cauchy-Riemann operator}
\label{sec:pseudo_hol_and_CR_operator}

Consider now a smooth manifold $M^{2n+q}$ equipped with a codimension~$q$ foliation $\cF$. 
Let moreover $\omega$ be a leafwise symplectic structure on $\cF$, and for the rest of this section restrict the attention to Sobolev maps of regularity $\fS = W^{k,p}$ with $k\geq 2$ and $(k-1)\,p > 2$.
Fix moreover a Riemann surface~$(\Sigma,j)$ and, in the case where the latter has non-empty boundary, also a foliated submanifold $P\subset M$.
From now on, we furthermore ask that the induced foliation on $P$ is Lagrangian, as in \Cref{def: Lagrangian foliation}.

The aim in this section is then to look at the subspace~$\cM$ of $W^{k,p}_0(\cF,P)$ made of pseudo-holomorphic maps $u\colon (\Sigma,\partial \Sigma) \to (M,P)$ that are foliated with trivial holonomy, and to prove that, for generic choices of auxiliary data, this is a finite dimensional smooth manifold. 
Let us explain the details.

\bigskip

Given a (domain-dependent) almost complex structure~$J$ on $\cF$, 
we are interested in \emph{$J$-holomorphic} maps, i.e.\ in maps $u\in W^{k,p}_0(\Sigma,\cF)$ verifying the \emph{Cauchy-Riemann equation}, i.e.\ 
\begin{equation}\label{eqn:pseudo_hol}
    \delbar^\cF_Ju=0,
    \quad
    \text{with}
    \quad
    (\delbar^\cF_Ju)(z)= \frac{1}{2}
    \left(
    \d_z u + J_{(z,u(z))}\circ \d_z u \circ j_z
    \right) \;,
\end{equation}
for every $z\in \Sigma$.
(Note that $\d u$ takes values in $\cF$, so the composition with $J$ indeed makes sense.)

As in the usual theory in the non-foliated setting, in order to prove that the space of $J$-holomorphic maps we are interested in is a smooth manifold, we now need to see the operator $\delbar^\cF$ as a section of a certain Banach vector bundle, and prove that its linearization is a surjective Fredholm operator at each $J$-holomorphic curve (for well chosen $J$).
In order to do so, having in mind the later application with the case of the \emph{universal moduli space}, we start by introducing the relevant Banach vector bundle whose fiber over a pair $(u,J)$ is made of all \emph{(domain-dependent) complex anti-linear forms}, or \emph{$(0,1)$-forms} in short, on $\Sigma$ and with values in $u^*\cF$.
This can be done as follows.

\medskip

Let $(E, \omega_E)$ be a symplectic vector bundle on a manifold $X$.
For fixed $p\in X$, an endomorphism~$J_p \in \End(E_p)$ is a complex structure on $E_p$ if $J_p^2=-\Id_{E_p}$; if it furthermore satisfies that
\begin{equation}
\label{eqn:def_compatible_J}
    \begin{split}
    (\omega_E)_p(J_p v,J_p w) &= \omega_E(v,w) \quad \text{ for every $v, w\in E_p$, and } \\
     (\omega_E)_p(v,J_pv) &\ge 0  \quad\text{ with equality if and only if $v = 0$,}    
    \end{split}
\end{equation}
then we say that it is \emph{compatible} with $(\omega_E)_p$.
We denote by $\Cpx(E,\omega_E)\subset \End(E)$ the space, for all points~$p\in M$, of all complex structures on $E_p$ that are \emph{compatible}\footnote{As in the non-foliated case, the whole theory can also be carried out analogously using leafwise almost complex structures which are just \emph{tamed}, i.e.\ that satisfy only the second condition specified in \eqref{eqn:def_compatible_J}.
For simplicity, we stick in this paper to the case of $\omega_E$-compatible almost complex structures.} 
with $(\omega_E)_p\in \Lambda^2(E_p^*)$.
By a slight adaptation of the classical proof due to Sevennec (see e.g.\ \cite[Chapter~II, Proposition~1.1.6]{AudLafBook}), one can see that $\Cpx(E,\omega_E)$ really is a smooth fiber bundle over $M$;
in fact, given a symplectic vector space~$(V,\Omega)$, and fixing any reference complex structure~$J_0$ that is compatible with $\Omega$, $\Cpx(V,\Omega)$ is diffeomorphic to the open unit ball (w.r.t.\  to an auxiliary metric) in the vector space of 
symmetric endomorphisms~$\psi$ of $V$ satisfying $J_0\circ \psi = \psi \circ J_0$.

We denote the Banach manifold of $C^l$-sections of the bundle $\Cpx(E,\omega_E) \to M$ by $\cJ^l\bigl(M,(E,\omega_E)\bigr)$, 
and we call almost complex structure on $X$ any section $J$ of $\Cpx(E,\omega_E)$.

\medskip

Let's now go back to our specific case of a symplectic foliation.
A compatible ($C^l$-regular) almost complex structure on a symplectic foliation~$(\cF,\omega_\cF)$ is an element of $\cJ^l\bigl(M,(\cF,\omega_\cF)\bigr)$.
Working with such almost complex structures would be sufficient for our purposes if we were only interested in closed holomorphic curves that are somewhere injective.
Unfortunately, using $\cJ^l\bigl(M,(\cF,\omega_\cF)\bigr)$ we cannot obtain the generic transversality property for curves with boundary\footnote{We will be interested mainly in moduli space of foliated pseudo-holomorphic discs with foliated Lagrangian boundary conditions.
As in the non-foliated case, the set of somewhere-injective points of a holomorphic disc does not need to be dense.
This is why we use domain-dependent perturbations of $J$ in this article in order to achieve transversality.} or for multiply covered curves, we will thus need to consider the space of \emph{domain-dependent} compatible almost complex structures.
For this, let $\pi_M\colon \Sigma\times M \to M$ be the obvious projection, and consider the pullback vector bundle $\pi_M^*(\cF,\omega_\cF)$ over $\Sigma \times M$.
We write $\Cpx\bigl(\Sigma\times M, (\cF,\omega_\cF)\bigr)$ instead of $\Cpx\bigl(\Sigma\times M, \pi_M^*(\cF,\omega_\cF)\bigr)$ to simplify notation. 
We then denote by $\cJ^l\bigl(\Sigma\times M,(\cF,\omega)\bigr)$ the space of $C^l$-regular sections of the fiber bundle $\Cpx\bigl(\Sigma\times M, (\cF,\omega_\cF)\bigr)\to \Sigma \times M$, and call its elements \emph{domain-dependent compatible almost complex structures} on $(\cF,\omega_\cF)$.

Now, over $\Cpx\bigl(\Sigma\times M, (\cF,\omega_\cF)\bigr)$ we can consider the (smooth, finite rank) vector bundle
\begin{equation*}
    \antiClinear\bigl(T\Sigma,\cF\bigr) \to \Cpx\bigl(\Sigma\times M, (\cF,\omega_\cF)\bigr)
\end{equation*}
that consists at a point $(z,p,J) \in \Cpx\bigl(\Sigma\times M, (\cF,\omega_\cF)\bigr)$ of all linear maps $\alpha\colon T_z\Sigma \to T_p\cF$ that anti-commute with $j$ and $J$, that is, $\alpha\circ j = - J\circ \alpha$.

By the general discussion about spaces of maps to the total space of a vector bundle in the beginning of \Cref{sec:leafwise_maps_trivial_holonomy} (and in particular about choosing sections of lower regularity than the maps to the base manifold),
\begin{equation*}
    \widetilde{\cE} \coloneqq  W^{k-1,p}\Bigl(\Sigma,\antiClinear\bigl(T\Sigma,\cF\bigr)\Bigr)
\end{equation*}
is a smooth Banach space bundle over 
\begin{equation*}
    \widetilde{\cB} \coloneqq W^{k,p}\bigl(\Sigma, \Cpx\bigl(\Sigma\times M, (\cF,\omega_\cF)\bigr)\bigr) \, .
\end{equation*}
Here, in the case where $\partial \Sigma \neq \emptyset$, $\widetilde{\cB}$ is more precisely to be intended as the space $W^{k,p}\bigl((\Sigma,\partial \Sigma), \bigl(\Cpx\bigl(\Sigma\times M, (\cF,\omega_\cF)\bigr), \Cpx\bigl(\Sigma\times M, (\cF,\omega_\cF)\bigr)\vert_{\partial\Sigma\times P}\bigr)\bigr)$;
for simplicity, we will stick to the previous shorter notation even when $\partial \Sigma \neq \emptyset$.

In \Cref{sec:leafwise_maps_trivial_holonomy} we have shown that, for $\fS = W^{k,p}$ with $k\geq 2$ and $(k-1)\,p > \dim \Sigma = 2$, the set $\fS_0\bigl(\cF,P\bigr)$ is a smooth Banach manifold if one of the following conditions is satisfied:
\begin{itemize}
    \item $\partial \Sigma$ is connected (possibly empty);
    \item  $P$ is transverse to $\cF$.
\end{itemize}
We will hence work under one of these two assumptions from now on.

It is also clear that the space of $C^l$-sections~$\cJ^l\bigl(\Sigma\times M,(\cF,\omega)\bigr)$ of $\Cpx\bigl(\Sigma\times M, (\cF,\omega_\cF)\bigr)$ with $l>k$ is a smooth Banach manifold.

To prepare for the later construction of the universal moduli space, we introduce now the product Banach manifold
\begin{equation*}
    \cB^\univ  \coloneqq \fS_0\bigl(\cF,P\bigr) \times \cJ^l\bigl(\Sigma\times M,(\cF,\omega)\bigr)
\end{equation*}
that allows us to vary the maps $u\colon \Sigma \to M$ and the almost complex structure~$J$ simultaneously.
For a fixed almost complex structure~$J \in \cJ^l\bigl(\Sigma\times M,(\cF,\omega)\bigr)$, we denote the slice $\fS_0\bigl(\cF,P\bigr) \times \{J\} \subset \cB^\univ$ by $\cB^J$.
(If directly working with $\cB^\univ$ causes too much psychological discomfort, the reader can restrict the following considerations to $\cB^J$, and only come back later in \Cref{sec:transversality_CR_operator} to the full product space.)

Consider the map
\begin{equation*}
    \Phi\colon \cB^\univ  \to \widetilde{\cB} ,
    \quad 
    (f,J) \mapsto J\circ \Gamma_f \;.
\end{equation*}
Note that even though $\Phi$ is a map between smooth Banach manifolds, it is a priori only of regularity $C^{l-k}$, see \cite[Lemma 2.98]{WenNotes}.
The pull-back of $\widetilde{\cE}$ to $\cB^\univ$ is then a $C^{l-k}$-Banach vector bundle that we denote by $\cE^\univ$; see the following diagram.

\begin{equation*}
    \begin{tikzcd}
            \cE^\univ \coloneqq \Phi^*\widetilde{\cE} \arrow[d, " "]   & \widetilde{\cE} = W^{k-1,p}\Bigl(\Sigma,\antiClinear\bigl(T\Sigma,\cF\bigr)\Bigr)  \arrow[d] \\
            \cB^\univ  = \fS_0\bigl(\cF,P\bigr) \times \cJ^l\bigl(\Sigma\times M,(\cF,\omega)\bigr) \arrow[r, "\Phi"] & \widetilde{\cB} = W^{k,p}\bigl(\Sigma, \Cpx\bigl(\Sigma\times M, (\cF,\omega_\cF)\bigr)\bigr)
    \end{tikzcd}
\end{equation*}

Note that, by definition of pull-back, the fiber of $\cE^\univ\to \cB^\univ$ over a point $(u,J)\in \cB^\univ = \fS_0(\cF,P) \times \cJ^l\bigl(\Sigma\times M,(\cF,\omega)\bigr)$ is the desired space $W^{k-1,p}\bigl(\Sigma,\antiClinear\bigl(T\Sigma,u^*\cF\bigr)\bigr)$, where $\Sigma$ is equipped with the complex structure~$j$, and $u^*\cF$ with $z\mapsto J\bigl(\Gamma_u(z)\bigr)$.
In particular, the restriction of $\cE^\univ$ to $\cB^J$ is then $\cE^J$.

\bigskip

Now that we have found a suitable Banach vector bundle $\cE^\univ\to \cB^\univ$, let's go back to the main aim, i.e.\ interpreting the Cauchy-Riemann operator in our setting as a section of $\cE^\univ$.
For the reader's sake, we start by explicitly arguing the following:

\begin{lemma}\label{lem:smooth_diff}
    Let $X$ and $Y$ be smooth manifolds and let $W^{k,p}(X,Y)$ be the smooth Banach manifold of Sobolev maps from $X$ to $Y$ with $kp > \dim X$ and $k>2$. 
    
    Then, the map associating to $f\in W^{k,p}(X,Y)$ its differential~$\d f$ gives a smooth section 
    \begin{equation*}
        \d \colon W^{k,p}(X,Y) \to W^{k-1,p}\bigl(X,\Hom(TX,TY)\bigr)
    \end{equation*}
\end{lemma}

(To make sense of the above statement, one would need to first argue that the space $W^{k-1,p}\bigl(X,\Hom(TX,TY)\bigr)$ is a Banach vector bundle over $W^{k,p}(X,Y)$.
This is analogous to the case of the Banach bundle $\widetilde\cE$ above, hence we will not give details here.)

\begin{proof}
    One way to see that the space of Sobolev maps $W^{k,p}(X,Y)$ is a Banach manifold consists in embedding $Y$ into $\R^n$ for $n$ sufficiently large and show via the implicit function theorem that $W^{k,p}(X,Y)$ is a Banach submanifold of the Banach \emph{space} $W^{k,p}(X,\R^n)$.
    
    Then, consider the differential on $W^{k,p}(X,\R^n)$.
    This is a bounded linear map that associates to every function $f=(f_1,\dotsc,f_n)\colon X \to \R^n$ an element $\d f = (\d f_1,\dotsc,\d f_n)$ each of whose components is an element of the space of Sobolev $1$-forms $W^{k-1,p}\bigl(\Omega^1(X)\bigr)$.
    For simplicity of notation denote the target space by $W^{k-1,p}\bigl(\Omega^1(X, \R^n)\bigr)$.

    To reinterpret this operator in terms of Banach bundles, replace the Banach space $W^{k-1,p}\bigl(\Omega^1(X, \R^n)\bigr)$ by the trivial Banach bundle $W^{k,p}(X,\R^n)\times W^{k-1,p}\bigl(\Omega^1(X;\R^n)\bigr)$.
    Clearly then we obtain from $\d$ the section
    \begin{equation*}
        \widetilde{\d}\colon W^{k,p}(X,\R^n) \to W^{k,p}(X,\R^n)\times W^{k-1,p}\bigl(\Omega^1(X;\R^n)\bigr),\;
        f \mapsto (f,\d f)
    \end{equation*}
    which is still bounded and linear, and thus in particular smooth.
    
    The desired bundle $W^{k-1,p}\bigl(X,\Hom(TX,TY)\bigr)$ is then given by the restriction of the previous bundle over $W^{k,p}(X,Y)$, and then further restricting the fibers to those elements of $W^{k,p}(X,Y)\times W^{k-1,p}\bigl(\Omega^1(X;\R^n)\bigr)$ whose second component is a differential taking values in $TY\subset T\R^n$.
    (Here, we use the hypothesis $k>2$ guaranteeing that these differentials are at least continuous.) 
    Then, the desired operator~$\d$ is just given by the restriction of $\widetilde{\d}$.
\end{proof}

By the definition of $\delbar_J^\cF$ in \eqref{eqn:pseudo_hol}, \Cref{lem:smooth_diff} and \cite[Lemma 2.98]{WenNotes} then tell us that, for all $l> k$, we have the following $C^{l-k}$-section of $\cE$ over $\cB = \fS_0\bigl(\cF,P\bigr) \times \cJ^l\bigl(\Sigma\times M,(\cF,\omega)\bigr)$:
\begin{equation*}
    \cS^\univ\colon \cB^\univ\to \cE^\univ\, , 
    \quad 
    \cS^\univ(u,J) = \delbar^\cF_{J\circ \Gamma_u} u \; .
\end{equation*}

Choose now a fixed almost complex structure~$J\in \cJ^l\bigl(\Sigma\times M,(\cF,\omega)\bigr)$ and restrict $\cS^\univ$ to the slice $\cB^J=\fS_0\bigl(\cF,P\bigr)\times\{J\}\subset \cB^\univ$ thus obtaining
\begin{equation*}
    \cS^J\colon \fS_0\bigl(\cF,P\bigr) \to \cE^J\, , 
    \quad 
    \cS^J(u) = \delbar^\cF_{J\circ \Gamma_u} u \; .
\end{equation*}
Using the canonical splitting $T\cE^J\vert_{O_{\cE^J}} = TO_{\cE^J}\times \cE^J\vert_{O_{\cE^J}}$ over the zero section $O_{\cE^J}\simeq \fS_0\bigl(\cF,P\bigr)$ of $\cE^J$, the vertical differential of $\cS^J$, denoted $D\cS^J$, at a point $u\in \fS_0\bigl(\cF,P\bigr)$ is a map 
\begin{equation*}
    D_u\, \cS^J \colon T_u \fS_0(\cF,P) \to \cE_u^J = W^{k-1,p}\bigl(\Sigma,\antiClinear\bigl(T\Sigma,u^*\cF\bigr)\bigr) \; .
\end{equation*}
In fact, using \Cref{rmk:tangent_bdle_to_space_maps} one can identify $T_u \fS_0(\cF,P)$ with $W^{k,p}(u^*\cF) \times \R^q$ if $\partial \Sigma = \emptyset$ and with $W^{k,p}(u^*\cF,u^*\cF_P) \times \R^s$ if $\partial \Sigma \neq \emptyset$ and if $s$ is the codimension of $\cF_P = \cF \cap TP$ in $P$.

This way, if $\partial\Sigma = \emptyset$ we can split $D_u \cS^J$ using the Reeb chart around $u$ as 
\begin{align}
        D_u\cS^J \colon & W^{k,p}(u^*\cF) \times \R^q \to W^{k-1,p}\bigl(\Sigma,\antiClinear\bigl(T\Sigma,u^*\cF\bigr)\bigr)   \label{eqn:linearized_CR_operator_without_bdry}  \, ;
        \intertext{similarly, if $\partial \Sigma\neq \emptyset$ (in which case recall we assume that either $\partial\Sigma$ is connected or $P$ is transverse to $\cF$) $D_u\cS^J$ splits as}
        D_u\cS^J \colon & W^{k,p}(u^*\cF,u^*\cF_P) \times \R^s \to 
        W^{k-1,p}\bigl(\Sigma,\antiClinear\bigl(T\Sigma,(u^*\cF,u^*\cF_P)\bigr)\bigr)
        \label{eqn:linearized_CR_operator_with_bdry}
\end{align}
where $W^{k-1,p}\bigl(\Sigma,\antiClinear\bigl(T\Sigma,(u^*\cF,u^*\cF_P)\bigr)\bigr)$ denotes the closure of the linear subspace of smooth sections 
$\eta\colon\Sigma\to \antiClinear\bigl(T\Sigma,u^*\cF\bigr)$ satisfying $\eta(T\partial \Sigma)\subset \cF_P\subset P$ in the Sobolev space 
$W^{k-1,p}\bigl(\Sigma,\antiClinear\bigl(T\Sigma,u^*\cF\bigr)\bigr)$.
More explicitly, the maps in (\ref{eqn:linearized_CR_operator_without_bdry},\ref{eqn:linearized_CR_operator_with_bdry}) are in both cases simply $(\xi_\cF,\xi_{\mathrm{transv}})  \mapsto D_u \cS^J (\xi_\cF,0) + D_u\cS^J(0,\xi_{\mathrm{transv}})$.
Note that $D_u \cS^J (\xi_\cF,0)$ is nothing else than the linearisation of the usual (non-foliated) Cauchy-Riemann operator (with totally real boundary if $\partial \Sigma\neq\emptyset$) associated to $u$, seen as a map into the symplectic manifold given by the leaf of the symplectic foliation that contains its image.

\begin{corollary}\label{lemma:lin_CR_op_is_fredholm}
    The linearised operators $D_u\cS^J$ in \Cref{eqn:linearized_CR_operator_without_bdry,eqn:linearized_CR_operator_with_bdry} are Fredholm and their respective indices are given by: 
    \begin{itemize}
        \item $\ind(D_u\cS^J) = n  \chi(\Sigma) + 2 \,\langle\, c_1(u^*\cF), [\Sigma]\, \rangle + q $ in the case $\partial \Sigma = \emptyset$.
        Here, $\chi(\Sigma)$ denotes the Euler characteristic of $\Sigma$, $2n$ is the dimension of the leaves of $\cF$, $c_1(\cF)$ is the first Chern class of $\cF$ with the complex structure given by $J\circ \Gamma_u$, and $q$ is the codimension of $\cF$ in $M$.
        \item $\ind(D_u\cS^J)=n  \chi(\Sigma) + \mu(u^*\cF,u^*\cF_P)+ s$ in the case $\partial \Sigma \neq \emptyset$.
        Here, $s$ is the codimension of $\cF_P = \cF \cap TP$ in $P$, and $\mu(u^*\cF,u^*\cF_P)$ denotes the \emph{relative Maslov index} of the complex vector bundle $u^*\cF$ over $\Sigma$ with respect to the totally real sub-bundle $u^*\cF_P$ defined over the boundary $\partial \Sigma$.
    \end{itemize}
\end{corollary}
\begin{proof}
    As already pointed out above, first note that the restriction of $D_u\cS^J$ to $W^{k,p}(u^*\cF)$ or to $W^{k,p}(u^*\cF,u^*\cF_P)$ in \Cref{eqn:linearized_CR_operator_without_bdry,eqn:linearized_CR_operator_with_bdry} are precisely the linearisations of the Cauchy-Riemann operators that one would consider if we were studying holomorphic curves in one fixed leaf of $\cF$.
    According to the established theory in the non-foliated setting, it is well known that these operators are Fredholm  \cite[Page~39 and Theorem~C.1.10]{McDuffSalamonBook} and that their indices agree with the ones stated above without the addition of the $q$- or $s$-terms.

    Since the complements to $W^{k,p}(u^*\cF)$ in $\fS_0\bigl(\cF\bigr)$ and to $W^{k,p}(u^*\cF,u^*\cF_P)$ in $\fS_0\bigl(\cF,P\bigr)$ are isomorphic to $\R^q$ and $\R^s$ respectively, and because the restriction of the continuous operator~$D_u\cS^J$ to a finite dimensional space is a compact operator, $D_u\cS^J$ is in both cases a Fredholm operator, and
    its index is simply the one used in the non-foliated setting, increased by the dimension of the finite dimensional complement, i.e.\  by $q$ or by $s$ respectively.
\end{proof}

\medskip

Analogously to what done in the non-foliated setup \cite{McDuffSalamonBook}, in order to prove in the next section that the space of holomorphic curves is for a generic choice of the domain-dependent almost complex structure in $\cJ^l\bigl(\Sigma\times M,(\cF,\omega)\bigr)$ a (finite dimensional) manifold, we will in fact first prove that the universal moduli space, i.e.\ the zero set of the section $\cS^\univ\colon \cB^\univ\to \cE^\univ$, $\cS^\univ(u,J) = \delbar^\cF_{J\circ \Gamma_u} u$, is a Banach manifold.
To this end, we first explicitly describe the tangent space to the second factor $\cJ^l\bigl(\Sigma\times M,(\cF,\omega)\bigr)$ of $\cB^\univ$ and then the linearisation of $\cS^\univ$ in $J$-direction.

Recall (e.g.\ from \cite[Page~47]{McDuffSalamonBook}) that, for a given symplectic vector bundle~$(E, \omega_E)$ over $M$, the space of $C^l$-sections~$\cJ^l\bigl(M,(E,\omega_E)\bigr)$ of $\Cpx\bigl(M, (E,\omega_E)\bigr)$ is a smooth Banach manifold.
Its tangent space $T_J\cJ^l\bigl(M,(E,\omega_E)\bigr)$ at a point~$J$ consists of $C^l$-sections of the bundle $\overline{\End}(E,\omega_E,J)\subset \End(E)$ over $M$ whose fiber over $p\in M$ is given by the space of linear maps $Y_p\colon E_p\to E_p$ such that
\begin{equation}
    \label{eqn:tangent_space_alm_compl_str}
    Y_pJ_p+J_pY_p=0 \, ,
    \quad
    \omega_E(Y_p\cdot,\cdot) + \omega_E(\cdot,Y_p\cdot)=0 \, .    
\end{equation}
In our case, where $(E,\omega_E)$ is the pull-back bundle of $(\cF,\omega)$ over $\Sigma\times M$, the tangent space $T_J\cJ^l\bigl(\Sigma\times M,(\cF,\omega)\bigr)$ at an element~$J$ consists of $C^l$-sections of the vector bundle $\overline{\End}\bigl(\Sigma\times M, (\cF,\omega,J)\bigr)$ over $\Sigma\times M$, where the fiber over $(z,p)\in \Sigma\times M$ is the set of all $Y(z,p)\colon \cF_p\to \cF_p$ that satisfy the conditions in \eqref{eqn:tangent_space_alm_compl_str} with $\omega_E$ replaced by $\omega$, and $J_p$ by $J(z,p)$.

Now, using as before the canonical splitting $T\cE^\univ\vert_{O_{\cE^\univ}} = TO_{\cE^\univ}\times \cE^\univ\vert_{O_{\cE^\univ}}$ over the zero section $O_{\cE^\univ} \simeq \fS_0\bigl(\cF,P\bigr)\times \cJ^l\bigl(\Sigma\times M,(\cF,\omega)\bigr)$ of $\cE$, the vertical differential~$D\cS^\univ$ of $\cS^\univ$ at a pair~$(u,J)$ is given by
\begin{equation}
    \label{eqn:linearized_CR_operator_univ}
    \begin{split}
        D_{(u,J)}\, \cS^\univ \colon&  T_u\fS_0\bigl(\cF,P\bigr)\times T_J\cJ^l\bigl(\Sigma\times M,(\cF,\omega)\bigr) \\
        &\qquad\qquad\qquad \longrightarrow \quad \cE^\univ_{(u,J)} = W^{k-1,p}\bigl(\Sigma,\antiClinear\bigl(T\Sigma,u^*\cF\bigr)\bigr) \\
        &(\xi_u,Y_J)  \longmapsto D_u \cS^J (\xi_u) + \frac{1}{2}Y_J(\cdot , u(\cdot)) \circ \d u  \circ j
    \end{split} 
    \, .
\end{equation}


\subsection{Transversality of the linearized Cauchy-Riemann operator}
\label{sec:transversality_CR_operator}

Given an almost complex structure $J \in \cJ^l\bigl(\Sigma\times M,(\cF,\omega)\bigr)$, let $\cMtilde^J$ be the space of $J$-holomorphic maps $u\in \fS_0\bigl(\cF,P\bigr)$ (note that if $\Sigma = \D$ or $\Sigma = \S^2$, we could replace $\fS_0\bigl(\cF,P\bigr)$ by $\fS\bigl(\cF,P\bigr)$ as $u$ will automatically have trivial holonomy), i.e.\ $\cMtilde^J$ can be interpreted as the intersection locus between the section $\cS^J\colon \fS_0\bigl(\cF,P\bigr) \to \cE^J$ and the zero section of $\cE^J$. 

A leafwise $J$-holomorphic curve $u\colon \Sigma\to M$ is said to be \emph{regular}, if $D_u\cS^J$ is surjective.
By the implicit function theorem, $\cMtilde^J$ is a manifold around any such $u$.
The goal of this section is to prove that $J$ can be chosen in such a way that every $u\in \cMtilde^J$ will be regular.
In fact, for later purposes we must require some additional control on the almost complex structures.

More precisely, for a given open subset $U\subset \Sigma \times M$, and a given domain-dependent almost complex structure~$J_0$ defined on its complement~$U^c$,
we denote by $\cJ^l\bigl(\Sigma\times M,(\cF,\omega) \,\text{rel.\ } J_0\vert_{U^c}\bigr)$ the subspace of $\cJ^l\bigl(\Sigma\times M,(\cF,\omega)\bigr)$ made of those domain-dependent leafwise almost complex structures which coincide with $J_0$ on the complement~$U^c$.

With this additional notation set up, we can now state the main result of this section:

\begin{proposition}\label{prop:mod_space_is_manifold}
    Let $U$ be an open subset of $\Sigma \times M$ (possibly the whole of $\Sigma \times M$ itself), and $J_0$ be a given domain-dependent almost complex structure defined on its complement~$U^c$.
    In the case of $\partial \Sigma \neq \emptyset$, fix also a closed leaf $F_0$ of $\cF_P$ on $P$.
    Then, for generic choice of $J$ in $\cJ^l\bigl(\Sigma\times M,(\cF,\omega) \,\text{rel.\ } J_0\vert_{U^c}\bigr)$, we have the following properties: 
    \begin{enumerate}
        \item every \emph{non-constant} leafwise $J$-holomorphic map $u \in \cMtilde^J$ whose graph in $\Sigma\times M$ intersects $U$ is regular, and in particular $\cMtilde^J$ is a $C^{l-k}$-manifold of dimension $\ind(D_u\cS^J)$ (given in \Cref{lemma:lin_CR_op_is_fredholm}) near $u$;
        
        \item if $\cF$ is coorientable and $\partial \Sigma = \emptyset$, $\cMtilde^J$ carries a natural orientation, and the same holds under the assumptions that $\cF$ is coorientable and $P$ is parallelizable in the case $\partial \Sigma\neq \emptyset$;
        \item if $\partial \Sigma\neq \emptyset$, the boundary evaluation map $\ev_\partial \colon \cMtilde^J \times \partial \D^2 \to P $ is transverse to $\cF_P$ along $F_0$.
    \end{enumerate}
\end{proposition}

We will deduce \Cref{prop:mod_space_is_manifold} from the following lemma:

\begin{lemma}\label{lemma:univ_mod_space_is_manifold}
    Consider integers $l\geq 2$ and $2\leq k \leq l$, and a real $p>2$.
    Let also $U$ be an open subset of $\Sigma \times M$ (possibly the whole of $\Sigma \times M$ itself), and $J_0$ be a given domain-dependent almost complex structure defined on its complement~$U^c$. 
    The \emph{universal moduli space}
    \begin{equation*}
        \cMtilde^\univ = \bigl\{\,
        (u,J)\in \fS_0\bigl(\cF,P\bigr) \times \cJ^l\bigl(\Sigma\times M,(\cF,\omega) \,\text{rel.\ } J_0\vert_{U^c}\bigr)
        \; \bigm\vert \; \cS^\univ (u,J) = 0  \,\bigr\} \; ,
    \end{equation*}
    is a $C^{l-k}$-Banach submanifold of $\fS_0\bigl(\cF,P\bigr) \times \cJ^l\bigl(\Sigma\times M,(\cF,\omega)\bigr)$, and the restriction of the projection map
    \begin{align*}
        \Pi_2\colon \fS_0\bigl(\cF,P\bigr) \times \cJ^l\bigl(\Sigma\times M,(\cF,\omega) \,\text{rel.\ } J_0\vert_{U^c}\bigr)
        & \to \cJ^l\bigl(\Sigma\times M,(\cF,\omega) \,\text{rel.\ } J_0\vert_{U^c}\bigr) \\
        (u,J) & \mapsto J
    \end{align*}
    to $\cMtilde^\univ$ is a map whose linearization at every point is a Fredholm operator of index $\ind \bigl(\d\Pi_2\vert_{T\cMtilde^\univ}\bigr) = \ind D_u\cS^J$ given in \Cref{lemma:lin_CR_op_is_fredholm}. 
    
    Moreover, in the case where $\partial\Sigma \neq \emptyset$, the universal boundary evaluation map 
    \begin{equation*}
        \ev^\univ_\partial \colon \cMtilde^\univ \times \partial \Sigma \to P,
        \quad (u,p) \mapsto u(p)
    \end{equation*}
    is transverse to $\cF_P$.
\end{lemma}

A clarification of our choices of $k$ and $p$ is in order.
Recall that these parameters determine the Sobolev regularity~$W^{k,p}$ of the maps contained in $\fS_0\bigl(\cF,P\bigr)$.
The assumptions that $k\geq 2$ and $p>2$ imply in particular that $(k-1)\,p > 2$, that is needed in \Cref{thm:leafwise_maps_trivial_hol_is_banach_submanifold} in order to ensure that $\fS_0(\cF,P)$ is a smooth Banach manifold.
Moreover, the assumption $p > 2$ comes into play in the proof below, where a continuous extension of the linearized Cauchy--Riemann operator to a space of $W^{1,p}$ maps needs to be considered.

\begin{remark}\label{rmk:generic_J_relative_subset}
    In what follows, we will need the following (trivial) consequence to the second part of the statement: if one can prove regularity by hand for those curves in $\cMtilde^J$ whose graph is completely contained in $U^c$, transversality holds at every point of $\cMtilde^J$, which is hence a manifold around any of its points. 
    \\
    This observation could sound like a qualitative improvement on the well-known analogous statement in the case of domain-independent almost complex structure, because one only needs to prove regularity by hand for those curves whose \emph{graph} is included in the closed set $U^c$ (which a priori need not be of the form $\Sigma\times V \subset \Sigma\times M$ for $V\subset M$ a closed subset).
    This said, in practice there is no sensible general way to control which points in the domain of the curves in the moduli space maps in $U$, hence one is in fact forced to check regularity by hand for all the curves whose \emph{image} is contained in the \emph{closed subset of $M$ given by the projection of $U^c$ in $M$}.
\end{remark}

\begin{proof}[Sketch of proof of \Cref{prop:mod_space_is_manifold}, using \Cref{lemma:univ_mod_space_is_manifold}]
The first conclusion in \Cref{prop:mod_space_is_manifold} follows directly from the \Cref{lemma:univ_mod_space_is_manifold}, using the same argument as in the non-foliated case explained in detail for instance in \cite[Proof of 3.1.6~(ii) at page~54]{McDuffSalamonBook}, but using \Cref{lemma:lin_CR_op_is_fredholm} in our setting.
More precisely, with \Cref{lemma:univ_mod_space_is_manifold} it is enough to choose a $J\in \cJ^l\bigl(\Sigma\times M,(\cF,\omega)\bigr)$ that is a regular value of the projection map~$\Pi_2$.

\medskip

The smooth moduli space~$\cMtilde^J$ is orientable if and only if $\Lambda^{\max} \bigl(T\cMtilde^J\bigr)$ is trivial.
The idea to show that the latter is trivial is to deform the family  of linearized Cauchy-Riemann operators over $\cMtilde^J$ through Fredholm operators to operators with more pleasant properties.
We work with a $K$-theoretic generalization of $\Lambda^{\max} \bigl(T\cMtilde^J\bigr)$ called the determinant bundle~$\det(X)$, which does not require us to assume that the considered operators are transverse.
This is explained in \cite[Appendix A.2]{McDuffSalamonBook} for the non-foliated case;
we limit ourselves here to describe the needed adaptations in our foliated setup.
The determinant bundle~$\det(X)$ is representable by a genuine real line bundle, which in our case will be a line bundle over the space~$X := \cMtilde^J\times [0,2]$ satisfying $\det(X)\rvert_{\cMtilde^J\times\{0\}} = \Lambda^{\max} \bigl(T\cMtilde^J\bigr)$.
The fiber of $\det(X)$ over $(u,\tau)\in X$ is given by
\begin{equation*}
    \det\nolimits_{(u,\tau)}(X) = \Lambda^{\max}(\ker D_{(u,\tau)})\otimes \Lambda^{\max}(\ker D_{(u,\tau)}^*) \; ,
\end{equation*}
where $D_{(u,\tau)}$ is the real linear Cauchy-Riemann operator over the point~$(u,\tau)$, that we will obtain as a deformation (parametrized by $\tau \in [0,2]$) of $D_u\cS^J$.

Let us first discuss the case of $\partial \Sigma = \emptyset$.
To deform $D_u\cS^J$, recall that this operator acts on $T_u\cMtilde^J$ which consists of vector fields along $u$.
Using a Riemannian metric, we can split $TM$ into $T\cF$ and the orthogonal complement to the foliation.
Denote the projection $TM \to T\cF$ by $\pi$ and consider the family of operators $D_{(u,\tau)}:=(1-\tau)\, D_u\cS^J + \tau\,D_u\cS^J\circ\pi$ for $(u,\tau)\in \cMtilde^J\times[0,1]$.
The linearized Cauchy-Riemann operator~$D_u\cS^J$ which we are studying in the foliated case splits with respect to a Reeb chart according to \eqref{eqn:linearized_CR_operator_without_bdry} as an operator~$D_u \oplus A$ defined on $W^{k,p}(u^*\cF) \otimes \R^q$, where $A$ is obviously a compact operator.
Now, $D_u\cS^J\circ\pi$ agrees with $D_u\cS^J$ along $W^{k,p}(u^*\cF)$ and only differs along a $q$-dimensional subspace, showing that the $D_{(u,\tau)}$ are Fredholm operators and proving that $\ker D_{(u,1)} = \ker D_u \oplus \R^q$, and the $\R^q$-factors can be identified with the normal bundle to the foliation~$\widetilde\cF$ on the space of foliatied maps, see \Cref{rmk:tangent_bdle_to_space_maps}.

This gives a family over the subset $\cMtilde^J\times[0,1]\subset X$ which restricts over $\{\tau=1\}$ as
\begin{equation}\label{eqn:fiber_det_bundle}
\begin{split}
    \det\nolimits_{(u,1)}(X) & \,= \Lambda^{\max}(\ker D_u \oplus \R^q) \otimes  \Lambda^{\max}(\ker D_u^*) 
    \\ &\, = 
    \Lambda^{\max}(\ker D_u) \otimes  \Lambda^{\max}(\ker D_u^*) \otimes \Lambda^q(\R^q) \; .
\end{split}
\end{equation}
The first part $\Lambda^{\max}(\ker D_u) \otimes  \Lambda^{\max}(\ker D_u^*)$ defines the determinant bundle at $u$ for the moduli space of curves that are contained in a single leaf, which agrees with the standard situation studied in the non-foliated theory.
In particular, we can deform $D_u$ by a second linear homotopy to a complex linear Cauchy-Riemann operator $C_u$ (as explained e.g.\ \cite[Appendix A.2]{McDuffSalamonBook}).
We use this homotopy to extend $\det(X)$ over the remaining piece $\cMtilde^J\times[1,2]$ of $X$.

It is now not hard to prove triviality of $\det(X)$ over all of $X$.
Indeed, recall that both the kernel and the cokernel of $C_u$ are complex vector spaces which are naturally oriented.
The last factor of the tensor product in \eqref{eqn:fiber_det_bundle} also has a natural trivialization, since $\R^q$ can naturally be identified with the normal bundle to the (cooriented) foliation~$\widetilde\cF$ on the space of foliated maps (at points of $\cMtilde^J$).
The natural trivialization along $\cMtilde^J\times\{2\}$ shows then that $\det(X)$ and in particular also $\det(X)\rvert_{\cMtilde^J} = \Lambda^{\max} (T\cMtilde^J)$ are globally trivial.

In the case where $\partial \Sigma \neq \emptyset$, under the additional assumption that $P$ is stably parallelizable, orientability of $\det(D_u)$ where $D_u$ is as in \eqref{eqn:linearized_CR_operator_with_bdry} can be proven similarly to the closed case above, using the additional arguments in \cite[Appendix]{GNW16} which deals with the orientability of a moduli space of pseudo-holomorphic discs having boundary condition on a (family of) contact Legendrian open books.
We refer the reader to \cite[Appendix]{GNW16} for the details.

\medskip

The proof of the third claim follows instead closely the argument in \cite[Pages~151 and 152]{WenNotes}.
The only adaptation needed is the fact that one needs to look at $(\ev_\partial^\univ)^{-1}(F_0)$ at the beginning of the proof in our case, which is a submanifold as $\ev_\partial^\univ$ is transverse to $\cF_P$ according to \Cref{lemma:univ_mod_space_is_manifold}. 
The rest of the argument is then formally the same, and one can reach the desired conclusion.
\end{proof}

\begin{proof}[Proof of \Cref{lemma:univ_mod_space_is_manifold}]
    The proof follows essentially the line of \cite[Proof of Proposition~3.2.1]{McDuffSalamonBook}.
    However, some adjustments are necessary for two reasons.
    The first is that we are in a leafwise setting; hence we need to use the leafwise linearized Cauchy-Riemann operator introduced earlier.
    Secondly, because we ultimately want to apply our results to discs, one cannot guarantee that injective points are dense in the domain; in our situation, we need thus to use domain-dependent perturbations of the (leafwise) almost complex structure to achieve the desired transversality of the linearized (leafwise) Cauchy-Riemann operator.
    Contrary to what one might expect, this actually makes the proof considerably easier.
    Here are the details.

    \medskip

    Let $(u,J)$ be any element of $\cMtilde^\univ$.
    According to \Cref{eqn:linearized_CR_operator_univ}, $D_{(u,J)}\cS^\univ$ is 
    \begin{equation*}
        \begin{split}
            D_{(u,J)}\, \cS^\univ \colon&  T_u\fS_0\bigl(\cF,P\bigr)\times T_J\cJ^l\bigl(\Sigma\times M,(\cF,\omega)\bigr)
            \to \cE^\univ_{(u,J)} = W^{k-1,p}\bigl(\Sigma,\antiClinear\bigl(T\Sigma,u^*\cF\bigr)\bigr) \\
            & (\xi_u,Y_J)  \longmapsto D_u \cS^J (\xi_u) + \frac{1}{2}\,Y_J(\cdot , u(\cdot)) \circ \d u  \circ j
        \end{split}
    \end{equation*}
    Recall also that $T_u\fS_0\bigl(\cF,P\bigr)$ is $W^{1,p}(u^*\cF)\times \R^q\times T_J\cJ^l\bigl(\Sigma\times M,(\cF,\omega)\bigr)$ in the case $\partial \Sigma = \emptyset$ or $W^{1,p}(u^*\cF,u^*\cF_P)\times \R^s\times T_J\cJ^l\bigl(\Sigma\times M,(\cF,\omega)\bigr)$ if $\partial\Sigma\neq \emptyset$.
    Denote $V_u\subset T_u\fS_0\bigl(\cF,P\bigr)$ the Banach subspace given by vectors with trivial $\R^q$- or $\R^s$-coordinate respectively, and $W_{(u,J)}:=V_u\times T_J\cJ^l\bigl(\Sigma\times M,(\cF,\omega)\bigr)\subset T_u\fS_0\bigl(\cF,P\bigr)\times T_J\cJ^l\bigl(\Sigma\times M,(\cF,\omega)\bigr)$.
    We are now going to prove that $D_{(u,J)}\, \cS^\univ\vert_{W_{(u,J)}}$ is surjective (which in particular implies surjectivity of $D_{(u,J)}\, \cS^\univ$).
    
    First, the image of $D_{(u,J)}\, \cS^\univ\vert_{W_{(u,J)}}$ \Fabio{modifying here} is closed for the following reason.
    The restriction of $D_{(u,J)}\, \cS^\univ\vert_{W_{(u,J)}}$ to $V_u\subset W_{(u,J)}$ is just the non-foliated linearized CR operator $D_u\, \cS^J_\cF$ obtained by looking at $u$ as simply a curve in the almost complex leaf $L$ of $\cF$ in which it is contained.
    This is a Fredholm operator according to the Riemann-Roch theorem, see \cite[Theorem C.1.10]{McDuffSalamonBook}.
    Hence, the image of $D_u\, \cS^J_\cF$ is a closed subspace of $\cE^\univ_{(u,J)}$ of finite codimension.
    This allows us to write $\cE^\univ_{(u,J)}$ as $\Image \bigl(D_u\, \cS^J_\cF\bigr) \oplus \R^a$ with $a\in \N$.
    The image of the initial operator~$D_{(u,J)}\, \cS^\univ\vert_{W_{(u,J)}}$ is then of the form $\Image \bigl(D_u\, \cS^J_\cF\bigr) \oplus V$ with $V$ a linear subspace of $\R^a$, and this is clearly closed.
    
    \smallskip
    
    For technical reasons, we would like to prove the surjectivity of $D_{(u,J)}\, \cS^\univ\vert_{W_{(u,J)}}$ in the case $k=1$ first so that $W^{k-1,p}\bigl(\Sigma,\antiClinear\bigl(T\Sigma,u^*\cF\bigr)\bigr)$ is simply $L^p\bigl(\Sigma,\antiClinear\bigl(T\Sigma,u^*\cF\bigr)\bigr)$, and then increase $k$ successively by a bootstrapping argument until arriving at the desired value chosen above.
    
    Now, $k=1$ however is not allowed under our hypothesis as we require the condition $k \geq 2$ to apply the \Cref{thm:leafwise_maps_trivial_hol_is_banach_submanifold} which tells us that the space $\fS_0(\cF,P)$ is a Banach manifold.
    
    Nevertheless, recall that $u$ is at least of class~$C^{l}$ (for $k\geq 2$ as in the hypothesis) by elliptic regularity as $k=\min(l,k)$ \cite[Theorem~B.4.1]{McDuffSalamonBook}. 
    Then, if we consider $(u,J) \in \fS_0(\cF,P)\times \cJ^l\bigl(\Sigma\times M,(\cF,\omega)\bigr)$ fixed, then $D_{(u,J)}\cS^\univ\vert_{W_{(u,J)}}$ above can be seen as naturally defined on sections of regularity $C^k$ with values in sections also of regularity $C^{k-1}$, for any $2\leq k\leq l$.
    Moreover, as it is simply a first order linear differential operator defined on these spaces of $C^k$-sections, it can be naturally extended to a continuous operator defined on the space $W^{1,p}(u^*\cF)\times T_J\cJ^l\bigl(\Sigma\times M,(\cF,\omega)\bigr)$, or the space $W^{1,p}(u^*\cF,u^*\cF_P)\times T_J\cJ^l\bigl(\Sigma\times M,(\cF,\omega)\bigr)$ respectively, and taking values in $L^{p}\bigl(\Sigma,\antiClinear\bigl(T\Sigma,u^*\cF\bigr)\bigr)$.
    For simplicity of the notation, we will also write $D_{(u,J)}\cS^\univ\vert_{W_{(u,J)}}$ for this extension to $k = 1$ in this part of the proof (until we specify otherwise).
    
    \smallskip

    Note also that the argument given above to show that $D_{(u,J)}\cS^\univ\vert_{W_{(u,J)}}$ has closed image carries over word for word to our situation where $k=1$, since the restriction of the operator to the first factor is still Fredholm \cite[Theorem C.2.3]{McDuffSalamonBook}.
    
    Because of its closedness, to prove that $D_{(u,J)}\, \cS^\univ\vert_{W_{(u,J)}}$ is surjective, it is thus enough to prove that its image is dense in $L^p\bigl(\Sigma,\antiClinear\bigl(T\Sigma,u^*\cF\bigr)\bigr)$.
    If there were an element in $L^{p}\bigl(\Sigma,\antiClinear\bigl(T\Sigma,u^*\cF\bigr)\bigr)$ not contained in $\Image D_{(u,J)}\cS^\univ\vert_{W_{(u,J)}}$, then we would find by the Hahn-Banach theorem a non-zero continuous operator $\eta'\in (L^p\bigl(\Sigma,\antiClinear\bigl(T\Sigma,u^*\cF\bigr)\bigr))^*$ in the dual space of $L^p\bigl(\Sigma,\antiClinear\bigl(T\Sigma,u^*\cF\bigr)\bigr)$ vanishing on the image of $D_{(u,J)}\cS^\univ\vert_{W_{(u,J)}}$.
    Choosing scalar products on the vector bundles and using the Hölder inequality, we can identify $(L^p\bigl(\Sigma,\antiClinear\bigl(T\Sigma,u^*\cF\bigr)\bigr))^*$ with $L^q\bigl(\Sigma,\antiClinear\bigl(T\Sigma,u^*\cF\bigr)\bigr)$, for $1/p + 1/q = 1$, via the pairing
    \begin{equation*}
        \alpha'(\beta) := \int_\Sigma \langle \alpha,\beta\rangle \Vol_\Sigma
    \end{equation*}
    for an $L^p$-section~$\beta$ and an $L^q$-section~$\alpha$.
    To prove the surjectivity of $D_{(u,J)}\, \cS^\univ$ it is thus enough to show that the only $\eta \in L^q\bigl(\Sigma,\antiClinear\bigl(T\Sigma,u^*\cF\bigr)\bigr)$ such that
    \begin{equation*}
        \int_\Sigma \langle \eta,\beta\rangle \Vol_\Sigma = 0
    \end{equation*}
    for every $\beta \in \Image D_{(u,J)}\, \cS^\univ\vert_{W_{(u,J)}}$ is the trivial one.
    That is, if $\eta$ is an element in $L^q\bigl(\Sigma,\antiClinear\bigl(T\Sigma,u^*\cF\bigr)\bigr)$ such that
    \begin{equation*}
        \int_\Sigma \langle \eta, D_{(u,J)}\, \cS^\univ(\xi,Y) \rangle = 0 \quad
        \text{for every  } (\xi, Y) \, ,
    \end{equation*}
    where $Y\in T_J\cJ^l\bigl(\Sigma\times M,(\cF,\omega)\bigr)$ and  $\xi\in W^{1,p}(u^*\cF)$ if $\partial \Sigma =\emptyset$, respectively $\xi\in W^{1,p}(u^*\cF,u^*\cF_P)$ if $\partial \Sigma\neq \emptyset$,
    then necessarily $\eta = 0$.
    This is equivalent to 
    \begin{equation}\label{eqn:eta}
        \int_\Sigma \langle \eta, D_u \cS^J (\xi) \rangle = 0
        \;\text{ for all $\xi$, and } \;
        \int_\Sigma \langle \eta,  Y(\cdot , u(\cdot)) \circ \d u  \circ j \rangle = 0 
        \text{ for all $Y$.}
    \end{equation}
    As $\xi$ is tangent to the foliation, $D_u\cS^J\vert_{W_{(u,J)}}(\xi)= D_u\cS^J_\cF(\xi)$, and hence
    it follows from the first of the two conditions in \eqref{eqn:eta} (together with \cite[Theorem~C.2.3]{McDuffSalamonBook}, similarly to the explanations in \cite[Proposition~3.1.11]{McDuffSalamonBook} for the non-foliated case) that $\eta$ is in fact a section in  $W^{1,q}\bigl(\Sigma,\antiClinear\bigl(T\Sigma,u^*\cF\bigr)\bigr)$, and that $(D_u \cS^J_\cF)^*\eta = 0$, where $(D_u \cS^J_\cF)^*$ is the formal adjoint of $D_u \cS^J_\cF\colon W^{k,p}(u^*\cF,u^*\cF_P) \to W^{1,p}\bigl(\Sigma,\antiClinear\bigl(T\Sigma,u^*\cF\bigr)\bigr)$.
    In particular, it follows that $\eta$ is continuous.
    
    Now, recall again that $u$ is at least of class $W^{1,p}$, and in fact even $C^k$ by elliptic regularity \cite[Theorem~B.4.1]{McDuffSalamonBook} as $k = \min(l,k)$ due to our hypothesis.
    In particular, by the assumption that $u$ is not constant, we know that the subset of points $z_0\in \Sigma$ such that $\d_{z_0} u \neq 0$ is open and dense:
    this is a consequence of the Unique Continuation theorem \cite[Theorem 2.3.2]{McDuffSalamonBook}.
    We will show that $\eta(z_0) = 0$ at any such point.
    Because if this were false, then we could find a $Y_0 \in \overline{\End}\bigl(\cF_{u(z_0)},\omega_{u(z_0)},J_{z_0,u(z_0)}\bigr)$ such that
    \begin{equation*}
        \langle \eta(z_0), \, Y_0 \circ d_{z_0}u\circ j_{z_0} \rangle >0 \, .
    \end{equation*}
    Extend $Y_0$ to a section $Y \in T_J\cJ^l\bigl(\Sigma\times M,(\cF,\omega)\bigr)$ such that $Y\bigl(z_0,u(z_0)\bigr) = Y_0$ (recall that we are using \emph{domain-dependent} $J$'s, so the $Y$ is a section of the bundle $\overline{\End}\bigl(\cF,\omega,J\bigr)\to \Sigma\times M$).
    By continuity of $\eta$, $\d u$, and $Y$, there is then an open neighborhood~$U_{z_0}$ of $z_0$ in $\Sigma$ such that
    \begin{equation*}
        \langle \eta(z), \, Y(z,u(z)) \circ d_{z}u\circ j_{z} \rangle >0
    \end{equation*}
    for every $z\in U_{z_0}$.
    Fix moreover a small open neighborhood $V_{u(z_0)}$ of $u_{z_0}$ in $M$.
    We can then choose a smooth cutoff function $\beta\colon \Sigma\times M \to [0,1]$ with support contained in a sufficiently small neighborhood of $(z_0,u(z_0))$ entirely contained in $U_{z_0}\times V_{u(z_0)}$ such that $\beta(z_0,u(z_0)) = 1$.
    Replacing $Y$ by $\beta \cdot Y$ yields a new element that still satisfies \Cref{eqn:tangent_space_alm_compl_str} and thus lies in $T_J\cJ^l\bigl(\Sigma\times M,(\cF,\omega)\bigr)$.
    With this new $Y$, the term on the right-hand side of \eqref{eqn:eta} is strictly positive, because the only points that contribute to the integral lie very near $z_0$ where the argument is positive.
    This contradiction together with the continuity of $\eta$ proves that $\eta$ must in fact be trivial, i.e.\ $D_{(u,J)}\, \cS^\univ\vert_{W_{(u,J)}}$ is in fact surjective as desired.
    
    \smallskip
    
    To prove surjectivity in the case where we look at $J$'s that are equal to a prescribed $J_0$ on the complement of a given open subset $U\subset \Sigma\times M$, and at the moduli space~$\cMtilde_U^\univ$ of those curves $u$ whose graph $\Gamma_u\colon \Sigma\to \Sigma\times M$ intersects $U$, it is easy to convince oneself that the argument above shows that $\eta$ vanishes on all of the (non-empty) open set $\Gamma_u^{-1}(U)$.
    Since $(D_u \cS^J_\cF)^*\eta = 0$, we obtain by Aronszajn's theorem  \cite[Theorem~2.3.4]{McDuffSalamonBook} that $\eta$ vanishes everywhere if it vanishes on any non-trivial open set.
    This shows then that $\eta$ also needs to vanish on $U^c$, so that $D_{(u,J)}\, \cS^\univ\vert_{W_{(u,J)}}$ is indeed surjective for $J$'s with prescribed on $U^c$. 
    
    \smallskip
    
    It is then left to prove surjectivity for the case $k>1$.
    Given an element $\mu \in W^{k-1,p}\bigl(\Sigma,\antiClinear\bigl(T\Sigma,u^*\cF\bigr)\bigr)$, we can use that $\mu$ lies in $L^p\bigl(\Sigma,\antiClinear\bigl(T\Sigma,u^*\cF\bigr)\bigr)$ and the surjectivity shown above to find $(\xi,Y)\in W^{1,p}(u^*\cF) \times T_J\cJ^l\bigl(\Sigma\times M,(\cF,\omega)\bigr)$ such that $D_{(u,J)}\, \cS^\univ(\xi,Y) = \mu$, or, equivalently,
    $D_u\cS^J(\xi) = \mu - \frac{1}{2}Y(\cdot , u(\cdot)) \circ \d u \circ j$.
    As $\xi$ is tangent to a leaf of $\cF$ and as $l\geq k$, by elliptic regularity in the non-foliated setting (see e.g.\ \cite[Theorem~C.2.3]{McDuffSalamonBook}) this implies that in fact $\xi\in W^{k,p}\Gamma(u^*\cF)$ as desired, showing that $D_{(u,J)}\cS^\univ\vert_{W_{(u,J)}}$ is also surjective for $k>1$.

    Note that the step to go from surjectivity for $k=1$ to surjectivity for $k>1$ does not require any particular adaptation in the case of the space of $J$ with prescribed values over $U^c$.

    \medskip

    Now that the surjectivity of $D_{(u,J)}\, \cS^\univ\vert_{W_{(u,J)}}$, and hence of $D_{(u,J)}\, \cS^\univ$ as a direct consequence, has been established, we apply the implicit function theorem to conclude. 
    Here are the details.

    Splitting $D_{(u,J)}\, \cS^\univ$ as in \Cref{eqn:linearized_CR_operator_univ} allows us to directly apply \cite[Lemma~A.3.6]{McDuffSalamonBook}, because 
    $D_u\, \cS^J$ is by \Cref{lemma:lin_CR_op_is_fredholm} a Fredholm operator.
    It follows that $D_{(u,J)}\, \cS^\univ$ moreover has a right-inverse operator and the claim about the projection operator~$\Pi_2$.
    The implicit function theorem \cite[Theorem~A.3.3]{McDuffSalamonBook} then applies to $D_{(u,J)}\, \cS^\univ$ thus giving that the universal moduli space~$\cMtilde^\univ$ is a $C^{l-k}$-manifold at any of its points, as $(u,J)$ was chosen arbitrarily.
    
    \medskip

    The only thing left to prove is the claim about $\ev_\partial^\univ$;
    for this, we follow closely the proof in \cite[Proposition~4.55]{WenNotes}.
    
    Let $p_0\in \partial \D^2$ and $(u_0,J_0)\in \cMtilde^\univ$.
    Recall that in the case where $\Sigma$ has non-empty boundary we have an identification $T_{u_0}\fS_0\bigl(\cF,P\bigr)=W^{k,p}((u_0)^*\cF,u^*\cF_P) \oplus \R^s$.
    Then, using the fact that the domain of $\ev^\univ_\partial\colon \partial \D^2\times  \cMline^\univ\to \cF_P\subset \cC$ naturally lives in $\partial \D^2 \times \cB^\univ  = \partial \D^2\times \fS_0\bigl(\cF,P\bigr) \times \cJ^l\bigl(\Sigma\times M,(\cF,\omega)\bigr)$, we can naturally write tangent vectors to $\partial \D^2\times  \cMline^\univ$ at a point $(p_0,u_0,J_0)$ as 
    \[
        (v,\xi_{u_0}=(\xi_{\cF}, \xi_{\transv}), Y)\in T_{p_0}(\partial \D^2)\times  W^{k,p}(u_0^*\cF,{u_0}^*\cF_P) \times \R^s \times T_{J_0}\cJ^l\bigl(\Sigma\times M,(\cF,\omega)\bigr) \; .
    \]
    Using this notation, we simply have that
    \begin{align*}
        \d_{(p_0,u_0,J_0)}\ev^\univ_\partial \colon & T_{(p_0,u_0,J_0)}(\partial \D^2 \times \cMtilde^\univ) \to TP 
        \\
        \quad 
        & (v,\xi_{\cF}, \xi_{\transv}, Y) \mapsto \d_{p}u(v) + \xi_{\cF}(p_0) + \xi_{\transv}(p_0) 
    \end{align*}
    As $u_0$ has boundary on $\cF_P$ and $\xi_{\cF}$ is tangent to $\cF_P$ by definition, it then follows that, in order to show that the image of $\d_{(p_0,u_0,J_0)}\ev^\univ_\partial$ to be
    transverse to $(\cF_{P})_{p_0}$ inside $T_{p_0}P$, it is enough to prove that, for any given $w_0\in T_{p_0}P$ transverse to $\cF_P$ there is
    $(v^0,\xi_{\cF}^0,\xi_{\transv}^0, Y^0)\in T_{(p_0,u_0,J_0)}(\partial \D^2 \times \cB^\univ)$ such that $\xi_{\transv}^0(p_0)=w_0$.

    Consider any $\eta^0_{\transv} \in \{0\}\times \R^s \subset  W^{k,p}(u_0^*\cF,{u_0}^*\cF_P) \times \R^s \simeq T_{u_0}\cB^\univ$.
    Recall now that we have already proved above that $D_{(u,J)}\cS^\univ$ is surjective.
    In fact, surjectivity has been achieved without using the direction transverse to the foliation: 
    namely, only vector fields with trivial component on the factor $\R^s$ have been used.
    In other words, we can find $(\xi_{\cF}^1,\xi_{\transv}^1=0, Y^1)\in T_{(u_0,J_0)}\cB^\univ$ such that 
    \[
        D_{(u_0,J_0)} \cS^\univ(\xi_{\cF}^1,0,Y^1) = - D_{(u_0,J_0)} \cS^\univ (\eta^0_{\transv}) \, .
    \]
    Then, the vector $(v^0,\xi_{\cF}^0,\xi_{\transv}^0, Y^0) \coloneqq (0,\xi_{\cF}^1,\eta^0_{\transv},Y^1)$ is by construction in the tangent space to $\cMtilde^\univ$ and satisfies at the same time that $\xi_{\transv}^1(p_0)=w_0$, as desired.
\end{proof}


\section{Moduli space of discs: topological properties}
\label{sec:mod_space_topological_properties}

Let $(M^{2n+q},\cF^{2n},\omega)$ be a symplectically foliated manifold, and $\cC$ a Lagrangian vanishing cycle embedded in $(M,\cF,\omega)$ and modeled on a closed $(n-1)$-dimensional manifold~$S$.

\Cref{sec:bishop_family} details how, for a particular almost complex structure $J$ on a neighborhood of the core, this normal form guarantees the existence of a parametric Bishop family of leafwise $J$-holomorphic discs.
We then explain in \Cref{sec:mod_space_disc_properties} that, via a semi-local uniqueness lemma for such discs, the parametric Bishop family gives rise to a finite dimensional moduli space of leafwise $J$-holomorphic discs with boundary in the Lagrangian foliation $P\cap \cF$.
We then derive in \Cref{sec:energy} the energy bounds which are necessary to apply Gromov compactness in the proof of \Cref{thm:trivial_lagr_vanish_cycle}, that is the content of \Cref{sec:proof_obstr}.


\subsection{Bishop family}
\label{sec:bishop_family}

According to \Cref{item:lvc_core} in \Cref{def:lagr_van_cycle}, the Lagrangian vanishing cycle has a simple Lagrangian-type tangency with the foliation along the core~$\cC_0$.
By \Cref{def:simple_tangency} there is then a model close to core~$\cC_0$, which we will denote $\Umod$ in the rest of this section, and which is given by $(D^*_\epsilon \times T^*_\epsilon S \times \R, i \d z \wedge \d \overline{z} + \d \lambda_0)$.
(Recall that $\Umod$ is not a neighborhood of the core inside $M$, but rather a local $1$-dimensional extension of a neighborhood of the core inside the leaf in which it is contained.
As will become clear later, this is sufficient for our purposes, since we are only considering pseudo-holomorphic curves with boundary on the Lagrangian vanishing cycle).
With a slight abuse of notation, we also denote $\Umod$ its image in $(M,\cF,\omega)$.

We equip the leaves of $\Umod = \D^2_\epsilon \times T^*_\epsilon S \times \R$ with the split leafwise almost complex
structure $\Jmod = i \times J_0$, where $i$ is the standard complex structure on the $\D^2_\epsilon$-factor and $J_0$ is an almost complex structure on $T_\epsilon^*S$ that is compatible with $d\lambda$.
The structure~$i\oplus J_0$ is then compatible with $\omega$, and although $\Jmod$ is not domain-dependent, it can of course be considered as such by simply assuming that it is defined on $\D^2 \times \Umod$ but that it happens to be constant in the first component.
Later we will extend this $\Jmod$ to a domain-dependent almost complex structure outside of $\Umod$, which is not constant in the $\D^2$ factor.
This will allow us to use the results obtained in \Cref{sec:transversality_CR_operator}.

Now, for any constant value of $\bfq_0 \in S$ and every $y_0\in \R$, the discs
\begin{equation*}
   \D^2_\epsilon\times \{(\bfq_0, \bfp=0)\} \times \{y_0\} \subset U =
    \D^2_\epsilon \times T^*_\epsilon S \times \R 
\end{equation*}
are clearly contained in the leaves of the foliation~$\cF$ and are $\Jmod$-holomorphic. 
Moreover, these naturally contain sub-discs that have boundary in $\cC^* = \cC\setminus \cC_0$ and that can be parametrized as maps
\begin{equation}
\label{eqn:def_bishop_disc}
  u_{r,\bfq}\colon (\D^2,\partial \D^2) \to (\Umod = \D^2_\epsilon \times T^*_\epsilon S \times \R, \cC^*) \, , 
  \quad 
  z\mapsto
  \bigl(\sqrt{r} z;  \bfq, 0; r\bigr) \, .
\end{equation}
for fixed choice of $r\in (0,\sqrt{\epsilon}/2)$ and $\bfq\in S$.
The image of each $u_{r,\bfq}$ is contained in the leaf $F_{r}$, and its boundary lies on the Lagrangian $L_{r}\subset \cC$.
We will refer to the family $\{u_{r,\bfq}\}_{(r,\bfq)\in [0,\sqrt{\epsilon}/2]\times S}$ as \emph{Bishop family stemming from $\cC$}, and call each $u_{r,\bfq}$ a \emph{Bishop disc}.

The following property follows directly from the fact that $\mu(T\C,T\S^1) = 2$ and from the properties of the Maslov index \cite[Appendix~C.3]{McDuffSalamonBook}:

\begin{lemma}
    \label{lemma:maslov_index}
    For every Bishop disc $u_{r,\bfq}$ as above, we have
    \begin{equation*}
        \mu\bigl(u_{r,\bfq}^*\cF, u_{r,\bfq}^*(T\cC^*\cap\cF)\bigr) = 2 \, .
    \end{equation*}
\end{lemma}

Because the induced foliation on $\cC^*$ is of codimension~$1$, according to \Cref{lemma:lin_CR_op_is_fredholm} and \Cref{lemma:maslov_index}, the index of the linearized Cauchy-Riemann operator at $u_{r,\bfq}$ then becomes
\begin{equation}
    \label{eqn:index_bishop_discs}
    \ind  D_{u_{r,\bfq}}\cS = n + \mu\bigl(u_{r,\bfq}^*\cF, u_{r,\bfq}^*(T\cC^*\cap\cF)\bigr) + 1 = n+3 \, .
\end{equation}
(For simplicity, we omit here, and later, explicit mention of $J$ from the notation for $\cS$, as contrary to the previous section we will only work with a fixed $J$ from now on.)

\begin{remark}
    The expected dimension of the space of Bishop discs given by this index is $n+3$.
    The parameters $(r,\bfq)$ however form only an $n$-dimensional space; 
    the remaining three dimensions can be recovered by acting on each $u_{r,\bfq}$ with the Möbius group.
    \\
    The usual way to reduce from the space of $J$-holomorphic maps to the space of geometric holomorphic curves is by passing to the moduli space identifying any two holomorphic maps that are equal to one another after reparametrization by the automorphism group of the domain.
    In the setup we have chosen, the almost complex structure is domain-dependent (besides on the neighborhood of the Lagrangian vanishing cycle), so that it is not invariant by the action of the Möbius group; we will instead reduce the dimension of the moduli space by asking that the maps evaluate to certain fixed submanifolds of the Lagrangian vanishing cycle at fixed points in the boundary of the domain.
    We will discuss these important points in more detail in \Cref{sec:mod_space_disc_properties}.
\end{remark}

To guarantee that the space of holomorphic discs is a smooth finite dimensional manifold, one usually appeals to the choice of a generic almost complex structure, since we have chosen $\Jmod$ however to be of a specific form on $\Umod$, we need to verify the transversality for the maps~$u_{r,\bfq}$ by hand, see also \Cref{rmk:generic_J_relative_subset}.

\begin{lemma}\label{lemma:regularity_by_hand}
  The discs $u_{r,\bfq}$ are regular.
\end{lemma}
The proof of this lemma follows closely the one given in \cite[Section 4.1.3]{NieNotes} for the case of contact bordered Legendrian open books.
\begin{proof}
    According to the index computation in \eqref{eqn:index_bishop_discs}, the linearized Cauchy--Riemann operator $D_{u_{r,\bfq}}\cS$ has a kernel $\cK$ of dimension at least $n+3$, and, more precisely, it is surjective if and only if $\cK$ has dimension \emph{at most} (and hence exactly) $n+3$.
    Now, recalling that 
    \[
    D_{u_{r,\bfq}}\cS \colon  W^{k,p}(u^*\cF,u^*\cF_P) \times \R^s \to 
        W^{k-1,p}\bigl(\Sigma,\antiClinear\bigl(T\Sigma,(u^*\cF,u^*\cF_P)\bigr)\bigr) \, ,
    \]
    we see that it is hence enough to prove that the restriction $D_{u_{r,\bfq}}\cS_\cF$ of $D_{u_{r,\bfq}}\cS$ to the first factor $W^{k,p}(u^*\cF,u^*\cF_P)$ has dimension \emph{at most} $n+2$ (in which case it will automatically be $n+2$).
    Note that $D_{u_{r,\bfq}}\cS_\cF$ is nothing else than the usual non-foliated linearized Cauchy--Riemann operator.
    
    Now, in the explicit neighborhood  $\Umod = \D^2_\epsilon \times T^*_\epsilon S \times \R$ where $u_{r,\bfq}$ lives, we have $\cF= \bigcup_{s} \D^2_\epsilon \times T^*_\epsilon S\times\{s\}$, with the standard leafwise symplectic structure, and $u_{r,\bfq}$ lives on the local leaf $F_r^{\mathrm{loc}}=\D^2_\epsilon \times T^*_\epsilon S\times\{r\}$, which we naturally identify with $\D^2_\epsilon \times T^*_\epsilon S$.
    One has $P_r\coloneqq P\cap F_r^{\mathrm{loc}} = \{\vert z \vert =\sqrt{r}\}\times S \subset \D^2_\epsilon \times T^*_\epsilon S$.
    Moreover, as $u_{r,\bfq}$ has constant value $\bfq\in S$ in the factor $T^*_\epsilon S$, we can find a local chart $T^*\R^{n-1}=\R^{2n-2}$ in $T^*_\epsilon S$ centered at $\bfq$, such that $\D^2_\epsilon \times \R^{2n-2}$ contains the image of any disc near to $u_{r,\bfq}$ in the space of trivial-holonomy leafwise maps $W^{k,p}_0(\cF,P)$.
    Recall that the almost complex structure we use is $i\oplus J_0$ on $\D^2_\epsilon \times T^*_\epsilon S$, as chosen at the beginning of \Cref{sec:bishop_family}, and so it is also of this split form in this open subset $\D^2_\epsilon \times \R^{2n-2}$.
    Up to a change of coordinates on the $\R^{2n-2}$ factor, we can also assume that $J_0$ on $\R^{2n-2}\subset T^*S$ coincides with the standard complex structure on $\R^{2n-2}=\C^{n-1}$ at the origin.
    
    Given a family $v_s$ in $W^{k,p}_0(\cF,P)$ starting at $v_0=u_{r,\bfq}$, we now want to understand when 
    \[
        D_{u_{r,\bfq}}\cS_\cF (\dot v_0) = \left.\frac{\d}{\d s}\right\vert_{s=0}\delbar_J v_s
    \]
    vanishes.
    In the chosen chart, we can write in coordinates $v_s=(a_s,b_s;\bfx_s,\bfy_s)\colon \D^2 \to \D^2_\epsilon \times \R^{2n-2}$.
    As the almost complex structure splits as $i\oplus J_0$ with $i$ translation invariant, $D_{u_{r,\bfq}}\cS \dot \dot v_0=0$ is equivalent to  
    \begin{align*}
    & \d \dot a_0 - \d \dot b_0 \circ i = 0 \, ,
    \quad 
    \d \dot b_0 + \d \dot a_0 \circ i = 0 \, , 
    \\
    & (\d \dot \bfx_0, \d\dot\bfy_0) - \left.\frac{\d}{\d s}\right\vert_{s=0} 
    (J_{(\bfx_s,\bfy_s)}\circ (\d \bfx_0, \d\bfy_0) \circ i) = 0 \, .
    \end{align*}
    Now, the first two identities mean that $(\dot a_0,\dot b_0)$ is a holomorphic map $\D^2\to \D^2_\epsilon$, while
    the second identity above can be simplified by using Leibniz' rule and the fact that $J_0(0,0)=i$ and $(\bfx_0,\bfy_0)=(0,0)$ to get 
    \begin{equation}
    \label{eqn:xy_coord_hol}
    \d \dot\bfx_0-\d\dot\bfy_0\circ i = 0 \, ,
    \quad 
    \d \dot\bfy_0+\d\dot\bfx_0\circ i = 0 \, ,
    \end{equation}
    which then implies that $(\dot\bfx_0,\dot\bfy_0)$ is in fact holomorphic $\D^2\to \R^{2n-2}=\C^{n-1}$ with respect to the standard complex structure on the target.
    Now, by harmonicity of $\dot\bfy_0$, we deduce that it must attain maximum and minimum at the boundary; however, by the required boundary conditions we know that $\dot\bfy_0=0$ on $\partial \D^2$, hence we get $\bfy_0$ vanishes identically.
    By \eqref{eqn:xy_coord_hol}, this implies that $\dot\bfx_0$ must be constant.
    This accounts for $n-1$ directions in the kernel of $D_{u_{r,\bfq}}\cS_\cF$; we hence need to prove that there are at most another $3$.
    
    Now, we know that $\dot w_0=\dot a_0+ i\dot b_0$ is a holomorphic map $\D^2\to \D^2_\epsilon$, and, by the explicit form of $P_r$ described above, its boundary condition is given by the equation $\frac{\d}{\d s}\vert_{s=0}(w_0\overline{w}_0-r)=0$,
    i.e.\ 
    \begin{equation}
        \label{eqn:boundary_cond_w_0}
        \dot w_0 \overline{w}_0 + \dot{\overline{w}}_0 w_0 = 0 \, .    
    \end{equation}
    Now, we can write in power series expansion
    \[
        \dot w_0(z) = \sum_{i=0}^\infty \alpha_k z^k \, .
    \]
    Evaluating this at $z=e^{i\varphi}\in \partial \D^2$, and recalling that $w_0(e^{i\varphi})=\sqrt{r}e^{i\varphi}$ by the boundary conditions for $u_{r,\bfq}$, we can hence substitute everything in \eqref{eqn:boundary_cond_w_0} and we get, by simply equating the coefficients in the expansion, that
    \[
    \alpha_1 + \overline{\alpha}_1=0 \, ,
    \quad 
    \alpha_0 + \overline{\alpha}_2 = 0 \, ,
    \quad
    \alpha_k = 0 \text{ for } k\geq 3 \, ,
    \]
    which shows that the choice of $\dot w_0$, i.e.\ of $(\dot a_0,\dot b_0)$, has $3$ free parameters, namely the whole complex parameter $\alpha_0$ and the imaginary part of $\alpha_1$.
    This concludes.
\end{proof}


\subsection{A priori energy bounds}
\label{sec:energy}

\begin{lemma}\label{lemma:primitive_on_cC}
    Let $\Omega\in\Omega^2(M)$ be an $\cF$-closed extension of the leafwise symplectic structure~$\omega$ on $\cF$.
    Then, the restriction of $\Omega$ to $\cC$ has a primitive $\lambda\in\Omega^1(\cC)$ such that $\lambda\vert_{L_{r}}$ is closed on $L_{r}$ for every $0\leq r\leq 1$, and $\lambda\vert_{L_0}=0$.
\end{lemma}
\begin{proof}
    We first claim that the restriction of $\Omega$ to $\cC^* = \cC \setminus \cC_0$ is closed.
    Indeed, $\cF$ induces a codimension~$1$ foliation on $\cC^*$.
    Thus, in particular we can choose at every point~$p\in \cC^*$ a basis of $T_p\,\cC^*$ of the form $(v_1,\dotsc,v_n,w)$ with all $v_j$ lying in the foliation. 
    Clearly, any combination of three vectors from this base to be plugged into $\d\Omega$ will contain at least two $v_j$'s so that $(\d\Omega)\vert_{T\cC^*}$ vanishes.
    
    Now, since $\d\Omega$ restricts to a smooth $3$-form on $\cC$ and $\d\Omega$ vanishes on the dense subset~$\cC^*$, it follows that $(\d\Omega)\vert_{T\cC}=0$, i.e.\ $\Omega\vert_{T\cC}$ is in fact also closed.
    
    The core~$\cC_0$ of the Lagrangian vanishing cycle is $\omega$-isotropic, and hence $\Omega\vert_{T\cC_0}=0$ as well.
    In particular, since $\cC$ deformation retracts onto $\cC_0$, by homotopy invariance of de Rham cohomology there is a $\lambda \in \Omega^1(\cC)$ such that $\Omega\rvert_{T\cC} =  \d\lambda$.
    
    Now, it is easy to see that $\lambda$ is closed on $L_r$ for every $r\in [0,1]$, because
    \begin{equation*}
        \d(\lambda\vert_{TL_r}) = (\d\lambda)\vert_{TL_r} = \Omega\vert_{TL_r} = \omega\vert_{TL_r} = 0 \; ,
    \end{equation*}
    since $L_r$ is isotropic in the symplectic leaves.
    
    Lastly, as $\d\lambda\vert_{L_0}=0$ on $L_0 = \cC_0$ and as $\cC$ has a natural projection map $\pi$ onto $\cC_0$, up to replacing $\lambda$ with $\lambda-\pi^*\bigl(\lambda\vert_{L_0}\bigr)$ (which is another primitive for $\Omega$ by closedness of $\lambda\vert_{T L_0}$), we can also arrange that $\lambda\vert_{L_0} = 0$ as desired.
\end{proof}

\begin{lemma}\label{lem:action_functional_bounded}
    Let now $[\gamma_0]$ be the homotopy class of a free loop~$\gamma_0$ in $\cC^*$, and denote by $C^1_{[\gamma_0]}(\S^1, \cF_{\cC^*})$ the space of all leafwise $C^1$-loops $\gamma\colon \S^1\to \cF_{\cC^*}$ representing $[\gamma_0]$, where $\cF_{\cC^*} = \cF\cap T\cC^*$. 
    The functional 
    \begin{equation*}
        A\colon C^1_{[\gamma_0]}(\S^1, \cF_{\cC^*})\to \R\, ,
        \quad 
        \gamma \mapsto \int_\gamma \lambda
    \end{equation*}
    descends to a continuous function $A_{[\gamma_0]}\colon  (0,1]\to \R$ such that the evaluation of $A$ on any loop $\gamma\in C^1_{[\gamma_0]}(\S^1, \cF_{\cC^*})$ whose image lies in $L_r$ is simply $A_{[\gamma_0]}(r)$.
    
    Furthermore, $A_{[\gamma_0]}$ can be extended continuously to $r=0$ by setting $A_{[\gamma_0]}(0) = 0$.    
    In particular, there is a uniform bound $C_{[\gamma_0]}$ for the functional $A$ on $C^1_{[\gamma_0]}(\S^1, \cF_{\cC^*})$.
\end{lemma}
\begin{proof}
    If $\gamma$ and $\gamma'$ are two loops in $C^1_{[\gamma_0]}(\S^1, \cF_{\cC^*})$ that lie both in the same $L_r$, then they are homotopic to each other in $L_r$, as they are homotopic in $\cC^*$, and $\cC^*$ retracts to $L_r$.
    Then, as the restriction of $\lambda$ to $L_r$ is closed, Stokes' theorem gives us that $\int_{\gamma}\lambda = \int_{\gamma'}\lambda$.
    In particular, the value of $A(\gamma)$ really only depends on the leaf~$L_r$ containing $\gamma$, so that the functional $A$ descends indeed to a well-defined function $A_{[\gamma_0]}\colon (0,1]\to \R$ with $A_{[\gamma_0]}(r) = A(\gamma)$ for any loop $\gamma \in C^1_{[\gamma_0]}(\S^1, \cF_{\cC^*})$ such that $\gamma \subset L_r$.
    
    Note also that $A_{[\gamma_0]}(r)$ is a smooth function of $r$. 
    This can be seen for instance by choosing a smooth family of loops in $C^1_{[\gamma]}(\S^1, \cF_\cC)$ that crosses all $L_r$ with $r\in (0,1]$.
    
    Now, recalling that the restriction of $\lambda$ to $L_0 = \cC_0$ is zero, by the very definition of $A$ one can see that $A_{[\gamma_0]}(r)\to 0$ as $r\to 0$.
    In particular, $A_{[\gamma_0]}$ can be extended to a continuous function $[0,1]\to \R$, as desired.
\end{proof}

\begin{proposition}\label{prop:bounded_energy}
    Consider the space $\cB = \cB\bigl((\D^2,\S^1), (\cF, \cF_{\cC^*})\bigr)$ of all smooth leafwise maps $f\colon (\D^2,\S^1) \to (\cF, \cF_{\cC^*})$, and let $\cB_0\subset \cB$ be the connected component of a given $f_0\in \cB$.
    Then, there exists a continuous function $E_{[f_0]}\colon [0,1] \to \R$ such that the symplectic energy
    \begin{equation*}
        E(f) := \int_{\D^2} f^*\omega
    \end{equation*}
    of an $f\in \cB_0$ is given by $E(f) = E_{[f_0]}(r)$, where $r\in (0,1]$ is such that $f(\S^1) \subset L_r$.
    In particular, the symplectic energy is uniformly bounded on $\cB_0$ by some constant~$C_{f_0}>0$ that depends only on $f_0$.
\end{proposition}
\begin{proof}
    Let $\lambda$ be as in \Cref{lemma:primitive_on_cC}, and  $C_{[\gamma_0]}$ with $\gamma_0 =\partial f_0$ be as in \Cref{lem:action_functional_bounded}.
    
    For every $f\in \cB_0$, we can find a smooth $[0,1]$-family of discs~$f_t$ with $f_1=f$ and $f_0$ being the reference disc such that each disc~$f_t$ is foliated.
    Define now the smooth map $F\colon [0,1]\times \D^2 \to M$, $F(t,p)=f_t(p)$.
    
    Using that $\Omega$ is $\cF$-closed, it follows that $F^*\d\Omega = 0$, and Stokes' theorem gives
    \begin{equation}\label{eqn:energy_1}
        0= \int_{[0,1]\times \D^2} F^* \d\Omega = 
        \int_{\partial([0,1]\times \D^2)} F^*\Omega = 
        \int_{[0,1]\times \overline{\S^1}} F^*\Omega + \int_{\D^2}f_1^*\Omega - \int_{\D^2}f_0^*\Omega.
    \end{equation}
    Now, as that the restriction of $\Omega$ to $\cC$ is $\d\lambda$ and that $F([0,1]\times \S^1)$ lies in $\cC$, we have
    \begin{equation*}
        - \int_{[0,1]\times \overline{\S^1}} F^*\d\lambda =
        + \int_{[0,1]\times \S^1} F^*\d\lambda =
        \int_{\{0,1\}\times \S^1} F^*\lambda =
        \int_{\S^1}f^*\lambda - \int_{\gamma_0} \lambda =
        A(\partial f) - A(\gamma_0)\; ,
    \end{equation*}
    and hence we can rewrite \eqref{eqn:energy_1} as
    \begin{equation*}
        E(f) = A(\partial f) - A(\gamma_0) + E(f_0) \; .
    \end{equation*}
    The only term on the righthand side depending on $f$ is $A(\partial f)$ which by \Cref{lem:action_functional_bounded}, can be replaced by
    $A_{[\gamma_0]}(r)$ for $r \in (0,1]$ being the index of the Lagrangian~$L_r$ that contains the boundary of $f$.
    
    The first two terms are bounded by $2 C_{[\partial f_0]}$, hence $C_0 \coloneqq E(f_0) + 2 C_{[\partial f_0]}$ is a suitable bound.
\end{proof}

\begin{remark}\label{rmk:closed_curves_energy_bound_zero_variation}
    When studying moduli spaces of \emph{closed} pseudo-holomorphic curves in a symplectic foliation admitting an $\cF$-closed global $2$-form, the same exact computation using Stokes' theorem in \Cref{eqn:energy_1} ensures that every curve in the same path-connected component of the moduli space has equal $\omega$-energy, thus giving directly a global $\omega$-energy bound (for that path-component).
\end{remark}


\subsection{The moduli space of discs and its properties}
\label{sec:mod_space_disc_properties}

We now look at the space of domain-dependent leafwise almost complex structures that extend the $\Jmod$ chosen in the local model~$\Umod$ as in \Cref{sec:bishop_family}, and fix any such $J$.

\medskip

For technical reasons that will be clarified later in the section, we make a preliminary modification to the Lagrangian vanishing cycle $\cC$.
According to \Cref{item:lvc_boundary_extension} of \Cref{def:lagr_van_cycle} and \Cref{rmk:codim_1_boundary_extension}, we know that $\cC$ can be smoothly extended to a slightly bigger foliated submanifold.
Note that the induced foliation in the extended region will a priori not be Lagrangian, but up to taking a smaller extension one can guarantee that the leaves are totally real with respect to the just chosen almost complex structure $J$, as being totally real is an open condition and it holds up to the boundary of $\cC$.
(In fact, considering a leafwise totally real condition would be enough for all the Fredholm analysis in \Cref{sec:analitical_foundations} to work --- just as in the usual non-foliated setup.)
In the following, we will denote by $\cCext$ this slight extension $\D^2_{1+\epsilon}\times S$ of the original vanishing cycle $\cC$; similarly, $\cCext^*$ will denote $\cCext\setminus \cC_0$.

\medskip

Let now $\cMtilde^J$ be the connected component of the space of $J$-holomorphic (parametrized) leafwise discs $u\colon (\D^2,\partial \D^2)\to (\cF, \cCext^*)$ that contain the Bishop discs~$u_{r,\bfq}$.
We claim that $\cMtilde^J$ satisfies the following properties:
\begin{enumerate}
    \item every disc in $\cMtilde^J$ has boundary with linking number $1$ w.r.t.\ $\cC_0$ inside $\cC$;
    \item there is a leafwise almost complex structure $J$ such that 
    $\cMtilde^J$ is a smooth manifold of dimension $n+3$; 
    \item every $u\in \cMtilde^J$ whose boundary intersects a sufficiently small neighborhood of the core~$\cC_0$ of $\cCext$ is a reparametrization of a Bishop disc.
\end{enumerate}

The first claim simply follows from the fact that, by choice of the path-connected component, every disc in $\cMtilde^J$ is homotopic  to a Bishop disc through discs having boundary on $\cCext^*$, and that Bishop discs satisfies this property.

The second property follows directly from \Cref{prop:mod_space_is_manifold} for all discs that are not contained in $\Umod$ by choosing a generic $J$ that agrees with $\Jmod$ on $\Umod$.
By \Cref{rmk:generic_J_relative_subset} this together with the explicit verification in \Cref{lemma:regularity_by_hand} is enough to prove the desired claim for all of $\cMtilde^J$.

Let us now discuss the third property. 
For this, we start by proving the following essential uniqueness near the core:

\begin{lemma}\label{lemma:semi_uniqueness_general_u}
    Let $\Umod$ be the model neighborhood introduced in \Cref{sec:bishop_family}.
    The only \emph{closed} $J$-holomorphic curves lying in $\Umod$ are the constant ones.
    
    Furthermore, if $\partial\Sigma\neq \emptyset$ the following holds: for every sufficiently small~$r_0 > 0$, there exists a minimal energy~$\hbar = \hbar(r_0,J) > 0$ such that the image of every $J$-holomorphic map $u\colon (\Sigma, \partial \Sigma) \to (M,L_r)$ is entirely contained in $\Umod$, whenever $u(\partial\Sigma)$ lies in a Lagrangian~$L_r$ with $r < r_0$ and the energy of $u$ is less than $\hbar$.
    In this case, there is also a holomorphic map $\varphi\colon (\Sigma, \partial \Sigma) \to (\D^2,\partial \D^2)$ such that $u = u_{r,\bfq}\circ \varphi$ where $u_{r,\bfq}$ is one of the Bishop discs.
\end{lemma}
\begin{proof}
    The restriction of the symplectic structure~$\omega$ to $\Umod$ is exact.
    Any closed holomorphic curve lying in $\Umod$ will have zero energy, and thus be constant.

    \smallskip

    Let us now assume that $u\colon (\Sigma, \partial \Sigma) \to (M,L_r)$ is a $J$-holomorphic map defined on a compact Riemann surface~$\Sigma$ with non-empty boundary.
    Since $\Umod$ is relatively compact, there is a uniform lower bound on the constants in the monotonicity lemma.
    This implies that if we choose an $r_0 > 0$ such that all $L_r$ with $r\leq r_0$ are contained in $\Umod$, there is an $\hbar > 0$ such that any holomorphic curve~$u$ with boundary in $L_r$ for $r \leq r_0$ and such that $u(\Sigma) \not\subset \Umod$ needs to have at least the minimal amount of energy~$\hbar$.
    
    \smallskip
    
    Assume now that the image of $u$ does lie inside $\Umod$.
    Note then that the composition of $u$ with the natural projection $\pi_S\colon \Umod = \D^2_\epsilon \times T^*_\epsilon S  \times \R \to T_\epsilon^*S$ is a $J_0$-holomorphic curve (recall $\Jmod = (i,J_0)$ on every leaf of $\Umod$ by the choice in \Cref{sec:bishop_family}).  
    Moreover, the symplectic energy of this projected holomorphic curve $u_S = \pi_S\circ u$ is just given by
    \begin{equation*}
        \int_{\Sigma} u^*\d\lambda_0 = \int_{\partial \Sigma}  u^*\lambda_0 = 0 \, ,
    \end{equation*}
    because the boundary of $\Sigma$ is being mapped by $u$ into $\cC$ so that its projection lies in the zero section of $T^*S$ where the standard Liouville form~$\lambda_0$ vanishes.
    It follows then that the projected curve is constant; in other words, $u$ is of the form $u(z) = (f(z); \bfq_0, 0;r) \in \Umod = \D^2_\epsilon\times T^*_\epsilon S \times \R$, for some holomorphic function $f\colon \Sigma \to \D^2_\epsilon$, fixed $\bfq_0 \in S$, and ---since the image of $u$ lies in the leaf of the foliation on $\Umod$--- for a constant $r \in [0,\sqrt{\epsilon}/2)$.
    
    In fact, $f$ maps the boundary of $\Sigma$ into a circle of radius~$\sqrt{r}$ in $\D^2_\epsilon$, and by the maximum principle, it follows that the image of $f$ has to be completely contained in the disc of radius~$\sqrt{r}$, that is, the image of $u$ lies in the image of the Bishop disc~$u_{r,\bfq_0}$.

    The Bishop discs are embedded, hence we can define a holomorphic map $\varphi\colon (\Sigma, \partial \Sigma) \to (\D^2,\partial \D^2)$ by $\varphi := u_{r,\bfq_0}^{-1}\circ u$ so that $u = u_{r,\bfq_0} \circ \varphi$, concluding the proof.
\end{proof}

As promised, we now upgrade \Cref{lemma:semi_uniqueness_general_u} to all the discs in $\cMtilde^J$ whose boundary lie sufficiently near the core:

\begin{proposition}\label{prop:uniqueness_Mtilde_J}
    Every $u\in \cMtilde^J$ whose boundary intersects a sufficiently small neighborhood of the core~$\cC_0$ of $\cCext$ is a reparametrization of a Bishop disc.
\end{proposition}
\begin{proof}
    Let $u$ be a disc in $\cMtilde^J$.  Choose an $r_0$ and its corresponding~$\hbar$ as in \Cref{lemma:semi_uniqueness_general_u}.
    We know from \Cref{prop:bounded_energy} that the energy of a disc in $\cMtilde^J$ only depends on the Lagrangian~$L_r$ on which the disc has boundary.
    Furthermore, by considering the energy of the Bishop discs, we can find an $r_0' < r_0$ such that the energy of every disc in $\cMtilde^J$ with boundary on an $L_r$ with $r<r_0'$ will be smaller than $\hbar$.
    We obtain then from \Cref{lemma:semi_uniqueness_general_u} that every such disc~$u$ is entirely contained in $\Umod$ and that it can be written as $u = u_{r,\bfq_0} \circ \varphi$ for a holomorphic map $\varphi\colon (\D^2, \partial \D^2) \to (\D^2,\partial \D^2)$.
    
    Since the boundaries of $u\in \cM^J$ and of $u_{r,\bfq_0}$ wind once around the core, the winding number of $\varphi\rvert_{\partial \D^2}$ around the origin needs also to be $1$.
    Every such $\varphi$ is surjective onto $\D^2$ as it would otherwise provide a retraction of $\D^2$ onto $\S^1$.

    We now aim to prove that $\varphi$ is a Möbius transformation, which will conclude the proof that $u$ is a reparametrization of a Bishop disc.
    For this, up to first composing with a suitable Möbius transformation, we can assume that $\varphi$ fixes the origin.
    As in the proof of the Schwarz lemma, we can then define an auxiliary function $g(z) = \varphi(z)/z$, which holomorphically extends to the origin by setting $g(0) := \varphi'(0)$.
    The holomorphic function~$g$, sends $\partial \D^2$ to $\partial \D^2$ and hence, by the maximum principle, sends the interior of $\D^2$ into the interior of $\D^2$.
    The winding number of $g\rvert_{\S^1}$ around the origin is the difference of the winding number of $z\mapsto \varphi(z)$ and of $z\mapsto z$, that is, $g\rvert_{\S^1}$ has zero winding number.
    
    Write $g$ in polar coordinates as $(R,\vartheta) \mapsto \bigl(g_R(R,\vartheta),g_\vartheta(R,\vartheta)\bigr)$.
    Since $g(\S^1) \subset \S^1$, $g_R(1,\vartheta) = 1$ does not depend on $\vartheta$.
    Using that the winding number of $g\vert_{\S^1}$ is trivial, there is a $\vartheta_0$ such that $\frac{\partial g_\vartheta}{\partial \vartheta}(e^{i\vartheta_0}) = 0$, and combined with the previous statement, $\frac{\partial g}{\partial \vartheta} (e^{i\vartheta_0}) = 0$.
    By the Cauchy-Riemann equation $Dg$ can only have either rank two or zero at any point, and in particular we obtain then that $\frac{\partial g}{\partial r}(e^{i\theta_0})$ also needs to vanish. 
    The Hopf boundary point lemma \cite[Lemma~3.4]{GilTruBook} shows that $g$ is constant, and thus necessarily of the form $g(z) = z_0$ with $|z_0| = 1$.
    The map $\varphi$ is thus a Möbius transformation, and $u$ is as desired a reparametrization of the Bishop disc~$u_{(r,\bfq_0)}$.
\end{proof}

\medskip

From now on, with a little abuse of notation, we replace $\Umod$ by the neighborhood (inside the original $\Umod$) in \Cref{prop:uniqueness_Mtilde_J}.
In other words, we shrink $\Umod$ so that every holomorphic disc in $\cMtilde^J$ intersecting $\Umod$ is a model Bishop disc up to parametrization. 

\bigskip

As mentioned above, the usual way to get the moduli space~$\cM^J$ is to identify all maps in $\cMtilde^J$ that match after reparametrizing their domains.
There are several advantages in using $\cM^J$ instead of working with the original $J$-holomorphic maps, $\cMtilde^J$.
Firstly, from an intuitive point of view, all maps that are equal up to reparametrization have the same image and represent thus a single geometric holomorphic curves.
Secondly, the Möbius group is not compact.
In $\cMtilde^J$, any holomorphic disc~$u$ can degenerate simply by reparametrizing it with a suitable family of elements in the Möbius group.
By working in $\cM^J$, these trivial degenerations are avoided.

Unfortunately, however, we are forced to use domain-dependent almost complex structures in this article, and in this case, it is no longer true that the Möbius transformations preserve the space of $J$-holomorphic discs.

Instead, to enforce in our case a reduction that will have a similar effect as the one sought by taking the quotient by the Möbius group, we will proceed as follows.
Select in the extended Lagrangian vanishing cycle $\cCext = j(\D^2_{1+\epsilon}\times S)$ the three "radial hypersurfaces" $N_1$, $N_i$, and $N_{-1}$ that are the image under $j$ of the three subsets $\bigl\{(r,\bfq)\in \D^2_{1+\epsilon}\times S\bigm| r\in [0,1+\epsilon]\bigr\}$,  $\bigl\{(ir,\bfq)\in \D^2_{1+\epsilon}\times S\bigm| r\in [0,1+\epsilon]\bigr\}$,  and $\bigl\{(-r,\bfq)\in \D^2_{1+\epsilon}\times S\bigm| r\in [0,1+\epsilon]\bigr\}$ respectively.
These will later be slightly perturbed to obtain transversality.
We define then
\begin{equation}
\label{eqn:constrained_mod_space}
    \cM^J = \bigl\{u \in \cMtilde^J \bigm|\;  u(1) \in N_1, \, u(i) \in N_i,\, u(-1) \in N_{-1}\bigr\} \;.
\end{equation}
Since we are imposing three codimension~$1$ conditions, we might hope by the standard formula for the expected dimension that $\dim \cM^J = \dim \cMtilde^J - 3 = n + 3 -3=n$.

It is easy to verify that the Bishop discs~$u_{r,\bfq}$ all lie in $\cM^J$.
Furthermore, given any Bishop disc~$u_{r,\bfq}$, the only Möbius transformation $\varphi\in \Aut(\D^2)$ such that $u_{r,\bfq}\circ \varphi$ still lies in $\cM^J$ is $\varphi = \Id_{\D^2}$, that is, in $\Umod$ our space $\cM^J$ is naturally diffeomorphic to the usual moduli space obtained by taking the quotient by the Möbius group.

Recalling the notation $\ev_\partial$ for the boundary evaluation map as in \Cref{prop:mod_space_is_manifold}, 
we can now give the main result in this section.

\begin{proposition}
\label{prop:moduli_space_plus_constant_discs}
    For generic choice of $J$, $\cMtilde^J$ is a smooth finite dimensional manifold, and the boundary evaluation map $\ev_\partial$ is transverse to $\partial\cC$.
    It is moreover naturally oriented if $S$ is stably parallelizable.
    Furthermore, up to slight perturbation of (the product of) the $N_1$, $N_i$ and $N_{-1}$, 
    the space~$\cM^J_\cC\coloneqq \ev_\partial^{-1}(\cC)\cap \cM^J$ is a smooth manifold with boundary of dimension $\dim \cM^J_\cC = n$.

    We can add the set of constant Bishop discs~$u_{(0,\bfq_0)}$ for $\bfq \in S$ to $\cM^J_\cC$, and a natural structure of smooth manifold with boundary to the resulting $\cM^J$ by gluing the constant Bishop discs to the other points via the collar model $[0,\epsilon) \times S$ glued in via  $(r,\bfq) \mapsto u_{r,\bfq}$.

    Furthermore, the only discs in $\cM^J$ that can approach the core~$\cC_0$ are Bishop discs. 
\end{proposition}

\begin{proof}
    According to \Cref{prop:mod_space_is_manifold}, if we choose a generic domain-dependent compatible almost complex structure~$J$ that agrees on $\Umod$ with $\Jmod$, then $\cMtilde^J$ will be a smooth manifold of dimension $n+3$, which is naturally oriented if $S$ is stably parallelizable, and $\ev_\partial$ will be transverse to $\partial \cC$.
    (Recall that we shrank $\Umod$ right after \Cref{prop:uniqueness_Mtilde_J}, so that the uniqueness property stated in that proposition holds on the whole of $\Umod$.)
    In particular, $\cMtilde^J_\cC\coloneqq \ev_\partial^{-1}(\cC)$ is a smooth manifold with boundary, of the same dimension of $\cMtilde^J$.

    Now we can define a smooth map $EV\colon\cMtilde^J_{\cC} \to \cC^*\times \cC^* \times \cC^*$,  $u\mapsto \bigl(u(1), u(i), u(-1)\bigr)$ and define $\cMtilde^J_{\cC} = EV^{-1}\bigl(N_1\times N_i \times N_{-1}\bigr)$.
    Slightly perturbing the product of the three radial manifolds away from $\Umod\times \Umod\times\Umod$,  we can assume that $EV \pitchfork \bigl(N_1\times N_i \times N_{-1}\bigr)$ so that $\cM^J_{\cC}$ will be a smooth manifold with boundary, of dimension $n+3 + 3\,(n-1) - 3n = n$. 
    Note that, a priori, as $\cC^*\times \cC^* \times \cC^*$ is a manifold with boundary with corners, the transversality theorem only guarantees that $\cMtilde^J_\cC$ is a manifold with boundary also having corners; however, the special situation at hand guarantees that if any of the three components of $EV$ take values in $\partial \cC^*$ then so do the other two, which guarantees that $\cMtilde^J_\cC$ is in fact simply a manifold with boundary (without corners).
    We also point out that, using Möbius transformations, it is easy to show that the desired transversality already holds close to $\Umod$ where we did not use any perturbation (recall the model almost complex structure is \emph{not} domain dependent).
    
    What's more, by \Cref{prop:uniqueness_Mtilde_J} any disc in $\cMtilde^J_{\cC}$ lying close to $\cC_0$ must be a reparametrization of a Bishop disc.
    However, the remark above implies that none of the reparametrizations of a Bishop disc (other than the Bishop disc itself) lies in $ \cM^J_{\cC}$, so that indeed the only discs in $ \cM^J_{\cC}$ getting close to $\cC_0$ are Bishop discs.
    
    Finally, it is easy to see that the constant Bishop discs can be glued onto $ \cM^J_{\cC}$ to give a boundary component of the resulting manifold.
\end{proof}

It remains to study the Gromov compactness of $\cM^J_\cC$ from \Cref{prop:moduli_space_plus_constant_discs}:

\begin{proposition}
\label{prop:compactness}
    The manifold with boundary~$\cM^J_{\cC}$ is compact. 
\end{proposition}
\begin{proof}
    We use a standard trick that allows us to rewrite curves that are holomorphic with respect to a domain-dependent $J$ as curves that are holomorphic with respect to a classical almost complex structure, namely, if $u\colon (\D^2, \partial \D^2) \to (M,\cC^*)$ is a foliated $J$-holomorphic disc for the chosen domain-dependent leafwise almost complex structures~$J$, then its graph $\Gamma_u\colon (\D^2, \partial \D^2) \to \bigl(\D^2\times M, \partial \D^2\times \cC^*\bigr)$, $z\mapsto \bigl(z,u(z)\bigr)$ is a foliated holomorphic disc with respect to the foliation~$\D^2\times \cF$ and the leafwise almost complex structure~$i\oplus J$ on $\D^2\times \cF$.
    Note that in contrast to $J$ on $\cF$, $i\oplus J$ on $\D^2\times \cF$ is \emph{not} anymore domain-dependent.
    It is easy to see that $i\oplus J$ is tamed by $dx\wedge dy \oplus \omega$ so that the energy of $\Gamma_u$ is just $\pi + E(u)$.
    Finally note that the boundary of $\Gamma_u$ lies in the totally real submanifold $\partial \D^2\times L_r$.
    
    If $(u_n)_n$ is now a sequence of holomorphic discs lying in $\cM^J_\cC$, so that their energy is bounded by a constant $C>0$, then the associated graphs~$\Gamma_{u_n}$ are holomorphic discs whose energy is bounded by $C + \pi$.
    By \Cref{prop:moduli_space_plus_constant_discs} it follows that any disc $u_n$ sufficiently close to the core~$\cC_0$ will be a Bishop disc, we assume thus that there is a $r_0>0$ such that none of the $u_n$ has boundary in $L_r$ with $r<r_0$.
    
    Except for the fact that the considered holomorphic curves live in a foliated symplectic manifold, and not just in a symplectic manifold, we could now apply standard Gromov compactness to the sequence of $\Gamma_{u_n}$.
    The additional complications coming from the foliated setup are fortunately minor because the isoperimetric inequality and the monotonicity lemma also hold in this situation.
    In particular the bubble tree is connected and all components of a nodal curve appearing as a Gromov limit lie in a single leaf of $\D^2\times\cF$.
    (A detailed proof of the leafwise version of Gromov compactness will appear in \cite{AlbNied}.)

    We will now show that the bubble tree consists only of a single component that lies in $\cM^J_\cC$.
    In fact, if the gradient of $\Gamma_{u_n}$ is uniformly bounded, then we can find a subsequence that converges uniformly to a disc $\Gamma_\infty\colon \D^2 \to \D^2\times M$ that is $i\oplus J$-holomorphic and whose first factor is the identity (as this is true for every $\Gamma_{u_n}$), so that $\Gamma_\infty$ is the graph of an element $u_\infty \in \cMtilde^J$ and hence the corresponding subsequence of $(u_n)_n$ converges to $u_\infty$.
    Clearly, $1$, $i$, and $-1 \in \partial \D^2$ are mapped to the corresponding submanifolds $N_1$, $N_i$, and $N_{-1}$ respectively, so that $u_\infty \in \cM^J_\cC$ as desired.
    
    \smallskip
    
    Suppose now instead that the gradient of the $\Gamma_{u_n}$ does blow-up.
    Since the first factor of $\Gamma_{u_n}$ is just the identity, gradient blow-up can only occur on the second factor of $\Gamma_{u_n}$.
    More precisely, if we find a point~$z_0$ in the domain of $\Gamma_{u_n}$ at which the gradient explodes,
    we need to rescale $\Gamma_{u_n}$ around this point, so that in the limit we will obtain a curve that is equal to the constant map~$z_0$ in the first factor, and whose second factor is a non-constant holomorphic curve (it is holomorphic with respect to the non-domain dependent almost complex structure~$J_{(z_0,\cdot)}$).
    
    The assumption \Cref{def:lagr_van_cycle}.\eqref{item:lvc_sympl_aspherical_leaves} excludes the existence of non-trivial holomorphic spheres.
    Rescaling can thus only lead in the limit to disc bubbling.
    
    As a result we obtain that $\bigl(\Gamma_{u_n}\bigr)_n$ has a subsequence and a finite collection of points $z_1,\dotsc,z_N \in \D^2$ such that $\bigl(\Gamma_{u_n}\bigr)_n$ converges on $\D^2\setminus \{z_1,\dotsc,z_N\}$ locally uniformly to a disc in $\D^2\times M$.
    This disc is the identity in the first component, while the second factor is a $J$-holomorphic disc that we denote by $u_\infty$.
    
    For each of the points $z_1,\dotsc,z_N$, we recover at least one  non-constant holomorphic disc obtained by rescaling at $z_j$ and possibly an additional finite collection of further holomorphic discs. 
    
    If $N=0$, then we are done as already explained above.
 
    If $N=1$, then there need to be at least two non-constant holomorphic components in the nodal curve in $M$:  discs obtained from rescaling are always non-constant. 
    The reason why the ``main'' component~$u_\infty$ is non-trivial is that $\D^2\setminus \{z_1\}$ contains at least two of the points $\{1,i,-1\}$.  As these two points are mapped by $u_\infty$ to different submanifolds $N_1$, $N_i$, or $N_{-1}$, the two points need to have positive distance, and $u_\infty$ also cannot be constant.
    
    If $N\ge 2$, then there are at least two non-constant holomorphic discs in the nodal curve in $M$ obtained from rescaling.
    
    Denote the projection to $M$ of all the components of the bubble tree by $v_1,\dotsc,v_{N'}$, with $N' \ge N + 1$.
    All of these discs are holomorphic maps $(\D^2,\partial \D^2) \to (F_r,L_r)$ for the same $r\in [r_0,1]$.
    (Note however that they are holomorphic with respect to different almost complex structures.)
    By the argument given above, at least two of the discs $v_1,\dotsc,v_{N'}$ are non-constant.
    
    The boundary of each disc~$v_j$ represents a loop in $L_r$ that is trivial in $\pi_1(F_r)$, so that by condition
    \Cref{def:lagr_van_cycle}.\eqref{item:lvc_pi1_condition} we find integers $a_1,\dotsc,a_{N'}$ such that $v_j \sim a_j\cdot\gamma_0$ in $\cC^*$, where $\gamma_0$ is the boundary of a Bishop disc.  
    Similarly, denote $a_\infty$ the coefficient so that $u_\infty\sim a_\infty\cdot\gamma_0$ in $\cC^*$.
    
    Using that the bubble tree is the limit of a sequence of Bishop discs, 
    we obtain the identity
    \begin{equation}
        \label{eqn:sum_homology_class}
        [u_\infty\vert_{\partial \D^2}] + [v_1\vert_{\partial \D^2}] + \dotsm + [v_{N'}\vert_{\partial \D^2}]  = [\gamma_0] \in H_1(\cC^*, \Z)
    \end{equation}
    leading to the condition
    \begin{equation}
        \label{eqn:sum_homology_class_in_terms_ai}
        a_\infty+
        a_1 + \dotsm + a_{N'} = 1 \;.
    \end{equation}

    If all the $a_i$'s are non-negative, they must in fact all be $0$ besides one, that we know must be $a_\infty$ and  which would have to be equal to $1$.
    Any other disc $v_j$ can then be capped-off with a disc lying inside $L_r$, because its boundary is null-homotopic in $\cC^*$ and thus also in $L_r$.
    We obtain a sphere in $F_r$ that needs to has trivial $\omega$-energy by \eqref{item:lvc_sympl_aspherical_leaves} of \Cref{def:lagr_van_cycle}.
    But since the cap-off disc lies in $L_r$, we deduce that $E(v_j) = 0$ for any $j\ge 2$.
    This implies that all discs~$v_j$ except for $v_1$ are trivial, which is a contradiction to the fact that there need to be at least \emph{two} discs that are not.
    
    It remains to exclude that some of the $a_j$'s are negative.
    By \Cref{eqn:sum_homology_class_in_terms_ai}, there needs to be a at least one disc~$v_{j'}$ such that $a_{j'}>0$; up to exchanging indices, we will suppose that $a_1>0$ and $a_2<0$.
    
    Define for every~$k\in \Z\setminus\{0\}$, a map $\phi_k\colon \D^2\to \D^2$ by $z\mapsto z^k$ if $k>0$, and $z\mapsto \bar z^{-k}$ if $k<0$.
    Then, it follows that $v := v_1\circ \phi_{a_2}$ and $v' := v_2\circ \phi_{a_1}$ have homotopic boundaries in the Lagrangian~$L_r$ so that $E(v) = E(v')$ by \Cref{prop:bounded_energy}. 
    This leads to a contradiction, because $E(v) = a_2\cdot E(v_1)$ and $E(v') = a_1\cdot E(v_2)$, but $E(v_1)$ and $E(v_2)$ are strictly positive because $v_1$ and $v_2$ are non-trivial holomorphic discs.
    
    This shows that there cannot be any bubbling.
\end{proof}


\subsection{Proof of \Cref{thm:trivial_lagr_vanish_cycle}}
\label{sec:proof_obstr}

Consider $(M,\cF,\omega)$ a strong symplectic foliation admitting a Lagrangian vanishing cycle~$\cC$.
As explained in the previous sections, for a conveniently chosen $J$, we obtain a compact $n$-dimensional moduli space $\cM^J_\cC$ of leafwise $J$-holomorphic discs with boundary on the Lagrangian foliation on $\cC$.
With a slight abuse of notation, we drop the explicit mention of $J$ and just denote such moduli space by $\cM^\cC$.
We start by proving that \eqref{item:lvc_circle_bound_sympl_disc} in \Cref{def:lagr_van_cycle} must hold.

\bigskip

Consider the moduli space $\cM_*^\cC = \cM^\cC\times \partial \D^2$, namely the space of pairs $(u,p)$ with $u\in\cM^\cC$ and $p\in \partial\D^2$. 
This comes equipped with a natural evaluation map
\begin{equation}
\label{eqn:evaluation}
    \ev_\partial\colon \cM_*^\cC\to M \, ,
    \quad 
    (u,p) \mapsto u(p) \; ,
\end{equation}
which is just the restriction of the one considered for $\cM^J_* = \cM^\cC\times \D^2$ in \Cref{prop:mod_space_is_manifold} to $\cM^\cC_*$.

Consider now an arc $\gamma\colon [0,1]\to \cC^*$ given by $\gamma(r) = i(r,\bfq_0)\in \cC = i\bigl(\D^2\times S\bigr)$, where $i\colon \cC\hookrightarrow M$ is the embedding of \Cref{def:lagr_van_cycle}.
By a small smooth perturbation of $\gamma$ away from $\Umod$ and $r=1$, we can assume that $\ev \pitchfork \gamma$; in particular, $\ev_\partial^{-1}(\gamma)$ is a $1$-dimensional (possibly disconnected and not necessarily closed) submanifold of $\cM_*^\cC$. 

According to \Cref{prop:moduli_space_plus_constant_discs}, there is exactly one connected component of $\ev^{-1}_\partial(\gamma)\subset \cM_*^\cC$ that comes arbitrarily close to the core~$\cC^0$.
This component, which we will call $\cM_\gamma$, is then diffeomorphic to an interval, closed at least at one of its ends, consisting of a constant Bishop disc at $i(0,\bfq_0)$.
We denote this endpoint by $e_0$.

By Gromov compactness  \Cref{prop:compactness}, $\cM_\gamma$ is compact and so will be homeomorphic to a \emph{closed} interval.
Let's then look at its other endpoint, which we will denote by $e_1$.

Now, $e_1$ cannot be a constant Bishop disc, because the only point where $\gamma$ intersects the core is $\gamma(0)$, which would mean that $e_0=e_1$ in $\cM^J_*$ and hence, again by \Cref{prop:moduli_space_plus_constant_discs}, points in $\cM_\gamma$ near $e_1$ would also coincide with points near $e_0$ contradicting the fact that $\cM_\gamma$ is not a circle but an interval.

According to \Cref{prop:compactness}, $e_1$ is then simply a disc in the leaf~$F_1$.
After a small perturbation, we can assume that it is immersed (and if $\dim \cF \ge 6$ also embedded) proving that \eqref{item:lvc_circle_bound_sympl_disc} of \Cref{def:lagr_van_cycle} holds as desired.

\bigskip

We now prove that \eqref{item:lvc_boundary_null_homologous} of \Cref{def:lagr_van_cycle} also holds.
Let us first deal with the homological statements, that hold for any even rank of $\cF$; the homotopy statement in the case $\rank\cF=4$ will be dealt with afterwards.
For this, we first note that according to \Cref{prop:moduli_space_plus_constant_discs}, $\cM^\cC$ (or rather its compactification) has two boundary components: one made of the discs sitting on the boundary of the vanishing cycle, and one made of constant Bishop discs~$u_{0,\bfq}$ with $\bfq \in S$.
Recall from \Cref{prop:moduli_space_plus_constant_discs} that $\cM^\cC$ is naturally oriented   in the case where $\partial \Sigma = \emptyset$, and also if $\partial \Sigma \neq \emptyset$ if $S$ is stably parallelizable.

Adding a marked point to each disc, we obtain the manifold $\cM_{\D^2}^\cC = \cM^\cC\times \D^2$ with boundary and corners that comes with a differentiable evaluation map $\ev\colon \cM_{\D^2}^\cC \to M, (u,z) \mapsto u(z)$.
The boundary of $\cM_{\D^2}^\cC$ consists of $\bigl(\partial \cM^\cC\bigr)\times \D^2 \cup \cM^\cC\times \partial\D^2$.
One can naturally obtain a new moduli space $\hat\cM_{\D^2}^\cC$ by blowing down each constant Bishop disc~$u_{0,\bfq}\times \D^2$ to a single point that we can identify with $\bfq\in S$.
Clearly, the evaluation map on $\cM_{\D^2}^\cC$ descends to a continuous evaluation map on $\hat\cM_{\D^2}^\cC$, and using the parametrization~\eqref{eqn:def_bishop_disc} in \Cref{sec:bishop_family} provides us with an explicit model that shows that the blown down space can even be given a $C^{l-k}$-structure such that $\ev$ is also $C^{l-k}$.

The boundary of this new space~$\hat \cM_{\D^2}^\cC$ is the union of $\hat \cM_{\partial \D^2}^\cC$ which is $\cM^\cC\times \S^1$ blown-down at the constant Bishop discs and $\cM^{\partial\cC}_{\D^2} = \cM^{\partial\cC} \times \D^2$ with $\cM^{\partial\cC}$ consisting of the discs with boundary in $\partial \cC$.  
The two parts intersect along their common boundary~$\cM^{\partial\cC}\times \S^1$, and they are both naturally oriented if $\cM^\cC_{\D^2}$ is.

Note then that $H_n(\hat \cM_{\partial \D^2}^\cC; \Z_2)$ is trivial because $\hat \cM_{\partial \D^2}^\cC$ has non-empty boundary, while $H_n(\hat \cM_{\partial \D^2}^\cC,  \partial \hat \cM_{\partial \D^2}^\cC; \Z_2) \cong \Z_2$.
By \Cref{prop:moduli_space_plus_constant_discs}, there is a neighborhood~$V$ of the core~$\cC_0$ in $\cC$ such that the evaluation map~$\ev\colon \hat \cM_{\partial \D^2}^\cC \to \cC$ restricts to a diffeomorphism $\ev^{-1}(V) \to V$.
This implies that $\ev\colon \bigl(\hat\cM_{\partial \D^2}^\cC, \partial\hat \cM_{\partial \D^2}^\cC\bigr)\to (\cC,\partial \cC)$ is a degree~$1$ map.

From the pair of long exact sequences
\begin{equation*}
    \begin{tikzcd}
        \ldots \arrow[r] & H_n(\hat \cM_{\partial \D^2}^\cC; \Z_2) \arrow[r] \arrow[d, "\ev_*"] &
        {H_n(\hat \cM_{\partial \D^2}^\cC, \partial\hat \cM_{\partial \D^2}^\cC; \Z_2)} \arrow[r, "\delta"] \arrow[d, "\ev_*"] 
        & H_{n-1}\bigl(\partial\hat \cM_{\partial \D^2}^\cC; \Z_2\bigr) \arrow[r] \arrow[d, "\ev_*"] 
        & \ldots \\
        \ldots \arrow[r]  & H_n(\cC; \Z_2) \arrow[r]                                           & {H_n(\cC, \partial\cC; \Z_2)} \arrow[r, "\delta"]                                                                    & H_{n-1}(\partial\cC; \Z_2) \arrow[r]                                                     & \ldots
    \end{tikzcd}
\end{equation*}
we see that $\delta\circ \ev_*$ and $\ev_*\circ \, \delta$ induce the same isomorphism $H_n(\hat \cM_{\partial \D^2}^\cC, \partial\hat \cM_{\partial \D^2}^\cC; \Z_2) \to H_{n-1}\bigl(\partial\cC; \Z_2\bigr)$.
Since $\delta$ is also an isomorphism,  $\ev_*\bigl([\partial\hat \cM_{\partial \D^2}^\cC]\bigr) = [\partial\cC] \in  H_{n-1}\bigl(\partial\cC; \Z_2\bigr)$.

If the involved manifolds are all oriented, we deduce using the exactly same arguments with homology with $\Z$-coefficients that $\ev_*\bigl([\partial\hat \cM_{\partial \D^2}^\cC]\bigr) = \pm [\partial\cC] \in  H_{n-1}\bigl(\partial\cC; \Z\bigr)$.

\smallskip

Let us now consider $\cM^{\partial\cC}_{\D^2} = \cM^{\partial\cC} \times \D^2$, the second part of the boundary of $\hat \cM_{\D^2}^\cC$.
Clearly, $[\partial \cM^{\partial\cC}_{\D^2}] = 0$ in $H_{n-1}\bigl(\cM^{\partial\cC}_{\D^2}; \Z_2\bigr)$, and since $\cM^{\partial\cC}_{\D^2}$ is mapped by the evaluation map into the leaf~$F_1$ of $\cF$, it follows that $\ev_*\bigl([\partial \cM^{\partial\cC}_{\D^2}]\bigr)$ is trivial in $H_{n-1}\bigl(F_1; \Z_2\bigr)$.
On the other hand, $\partial \cM^{\partial\cC}_{\D^2}$ agrees with $\partial\hat \cM_{\partial \D^2}^\cC$ (in fact, equipped with the opposite orientation) and we have shown that $\ev_*\bigl([\partial\hat \cM_{\partial \D^2}^\cC]\bigr) = [\partial \cC]$, we thus obtain that $\partial \cC$ is null-homologous with $\Z_2$-coefficients in the corresponding leaf, that is, property \Cref{def:lagr_van_cycle}.~\eqref{item:lvc_boundary_null_homologous}.
The result for $\Z$-homology is the same, if all moduli spaces are oriented.

\bigskip

The only thing left to prove is that if $\dim\cF = 4$, $\partial\cC$ is homotopic to a loop inside the leaf~$F_1$ it is contained in.
Note that in this case the connected closed moduli space $\cM^{\partial\cC}$ made of leafwise pseudo-holomorphic discs with boundary on $\partial \cC$ is just a circle so that $\partial \cM^{\partial\cC}_{\D^2}$ is diffeomorphic to the $2$-torus.
According to what we have said in the previous paragraphs, $\ev\colon \partial \cM^{\partial\cC}_{\D^2}  \to \partial \cC$ is a degree~$\pm 1$, and by a result due to Nielsen \cite{Nielsen29} (c.f.\ \cite{Edmonds79}), a degree~$1$ map from the torus~$\T^2$ to itself is homotopic to a homeomorphism (hence the same holds for degree $\pm1$ maps as well).

Since $\partial \cC$ is also a torus, we can thus homotope the evaluation map $\ev\colon \cM_{\D^2}^{\partial\cC} \to F_1$ in such a way that it restricts along the boundary $\partial \cM^{\partial\cC}_{\D^2}$ to a homeomorphism onto $\partial\cC$.
For simplicity, we denote this deformed evaluation map also by $\ev$.

We obtain now the desired homotopy $f_t\colon  \partial \cM^{\partial\cC}_{\D^2} \to F_1$ for $t\in [0,1]$ by  $f_t\bigl(u,z\bigr) = \ev\bigl(u,(1-t)\,z\bigr)$.
By our previous argument, $f_0$ is a homeomorphism onto $\partial \cC$, and $f_1$ does not depend on the $z$-coordinate, and represents a loop in $F_1$.
This concludes the proof of \Cref{thm:trivial_lagr_vanish_cycle}.
\hfill
\qedsymbol
\qedhere


\section{Examples of non-strong symplectic foliations without closed leaves}
\label{sec:examples}

There are plenty of closed manifold admitting a symplectic foliation with a Lagrangian vanishing cycle. For example, consider any closed $3$-manifold $M$ with a foliation $\cF$ containing at least one Reeb component, and $\omega$ a leafwise area form. If $(X,\Omega)$ is a symplectic manifold containing a weakly exact Lagrangian $L$, then the symplectic product foliation
\[(\cF \times X, \omega + \Omega) \;,\]
on $M \times X$ contains a Lagrangian vanishing cycle, given by the product of a disc slice of a Reeb component in $\cF$ with $L$.

In these examples \Cref{thm:trivial_lagr_vanish_cycle} is in fact not needed. 
Indeed, all these symplectic foliations contain a closed separating leaf, given by the the product of boundary leaf in a Reeb component with $X$, and hence are not strong by \Cref{obs:obstr_strongness} (and in fact not even taut). 
To obtain more interesting examples we show in this section that (under certain conditions) the underlying smooth foliation can be made taut while preserving the Lagrangian vanishing cycle. 
The construction is follows an argument of Meigniez \cite{Meigniez}, that in turn takes inspiration from work of Thurston \cite{Thurston}.

Foliated bundles play a key role in the construction so we start by briefly recalling their definition in \Cref{sec:FoliatedBundles}. \Cref{sec:eliminating_closed_leaves} explains how to open closed leaves and \Cref{sec:eliminating_preserving_vanishingcycle} shows that the latter can moreover be made in a way that preserves a given Lagrangian vanishing cycle in the original symplectic foliation.
Lastly \Cref{sec:proof_of_codim1_examples} contains the proof of \Cref{thm:codim_1_examples}.


\subsection{Foliated bundles}
\label{sec:FoliatedBundles}

Consider manifolds $M$ and $X$, together with a representation $\rho:G \to \Diff(X)$, for a subgroup $G\subset \pi_1(M)$. This allows us to form a bundle with fiber $X$ by taking
\[ \wtd{M} \times_{G} X := \wtd{M} \times X/(p,x) \sim (g\cdot p, \rho(g)\cdot x),\quad \forall g\in G,\]
where $\wtd{M}\to M$ denotes the universal cover on which $\pi_1(M)$ acts by deck transformations. The projection $\pi\colon \wtd{M} \times_{G} X \to M$ is induced by $\wtd{M} \to M$.

Observe that $\wtd{M} \times X$ is foliated by $\wtd{M} \times \{x\}$ for all $x \in X$, and that this foliation passes to a foliation $\cF$ on the quotient whose leaves are transverse to the fibers of $\pi$. Moreover, if $M$ admits a symplectic form $\omega$, then $\Omega := \pi^*\omega$ makes $\cF$ into a strong symplectic foliation.

The two main cases of interest to us are the following.
\begin{itemize}
    \item \textbf{Round case:} when $X = \S^1$ we assume that the representation $\rho$ takes values in $\Diff_+(\S^1)$, the space of orientation preserving diffeomorphisms.
    \item \textbf{Straight case:} when $X = (-1,1)$ we assume that $\rho$ takes values in $\Diff_c(-1,1)$, the space of compactly supported diffeomorphisms.
\end{itemize}


\subsection{Eliminating closed leaves of symplectic foliations}
\label{sec:eliminating_closed_leaves}

The aim of this section is to show that, under certain conditions, it is possible to eliminate closed leaves from a symplectic foliation without changing the ambient manifold.
The idea is as follows.

Suppose we are given a symplectic foliation $(\cF,\omega)$ on $M$ containing a closed leaf $L$.
Then, the steps we perform are the following.
\begin{enumerate}
    \item We remove an open subset $H$ of $M$ creating a hole that intersects the closed leaf $L$ we wish to eliminate, so that the fundamental group $\pi_1(L\setminus H)$ has an extra generator.
    Seeing the foliation near $L$ as a foliated bundle, we use this new element in the fundamental group to change $(\cF\vert_{M\setminus H},\omega)$ locally around $L\setminus H$ so that the closed leaf gets eliminated.
    \item The previous step changes the induced characteristic foliation on the boundary of the holes. 
    It remains to be shown that the holes can be filled again without changing the topology of $M$. We achieve this by generalizing one of the key results from \cite{Meigniez} to the symplectic setting. 
    (To be slightly more precise, this step requires enlarging the hole $H$ to a bigger one $H'$, which is the hole that will be filled.)
\end{enumerate}

\medskip

The full details of this process will be described in the proof of \Cref{lem:RemovingClosedLeaf}, and we now introduce all the necessary preliminaries.

First, we give the precise definition of the type of holes involved in the construction. 
Suppose we are given the following data:
\begin{enumerate}
    \item a symplectic manifold $(X \times \D^2,\omega)$ where $X$ is closed;
    \item a representation $\rho\colon \pi_1(X) \to \Diff_c(-1,1)$;
    \item a distinguished element $\phi \in \Diff_c(-1,1)$.
\end{enumerate}
As in Section \ref{sec:FoliatedBundles} we consider two cases:
\begin{itemize}
    \item \textbf{Round hole:} 
    We identify $\Diff_c(-1,1)$ with $\Diff_+(\S^1,*)$, the orientation preserving diffeomorphisms of $\S^1$ fixing the basepoint. 
    Thus, $\rho$ and $\phi$ together define a representation $\lambda\colon \pi_1(X\times \S^1) \to \Diff_+(\S^1,*)$ of the fundamental group of $\partial(X\times \D^2)=X\times \S^1$.
    The suspension foliation $(\cF,\omega)$ associated to $\lambda$ defines a germ of symplectic foliation around the boundary of $X \times \D^2 \times \S^1$. 
    We call $(X\textbf{}\times \D^2 \times \S^1,\cF,\omega)$ the \emph{round hole} associated to $(\rho,\phi)$.

    \item \textbf{Straight hole:} The manifold $X\times \D^2 \times I$ has boundary and corners.
    Analogously to the round case, $\rho$ and $\phi$ define a representation $\lambda\colon \pi_1(X\times \S^1) \to \Diff_c(-1,1)$, 
    whose suspension defines a germ of symplectic foliation $(\cF,\omega)$ around the vertical boundary $X \times \S^1 \times I$. 
    It can naturally be extended over the horizontal boundary as the product foliation
    \[ \cF = \bigcup_{t \in \Op(\partial I)} X \times \D^2 \times \{t\} \;,\]
    with the leafwise symplectic form $\omega$. We refer to $(X\times \D^2\times I,\cF,\omega)$ as the \emph{straight hole} associated to $(\rho,\phi)$.
\end{itemize}

Adapting one of the key results from \cite{Meigniez}, we obtain a sufficient condition to fill such holes by a symplectic foliation.

\begin{lemma}\label{lemma:filling_hole}
    Given $\rho$ and $\phi$ as above, suppose that:
    \begin{enumerate}
        \item $\phi = [\alpha,\beta]$ is a commutator in $\Diff_c(-1,1)$,
        \item $\lambda(\gamma_1), \dotsc, \lambda(\gamma_r)$, where the $\gamma_i$ denote generators of $\pi_1(X)$, have pairwise disjoint supports, and the union of their supports is non-empty and disjoint from the supports of $\alpha$ and $\beta$.
        \item In the straight case we additionally assume that the union of the supports of the $\lambda(\gamma_i)$'s ``brackets'' the supports of $\alpha$ and $\beta$. 
    \end{enumerate}
    Then the hole $(H,\cF,\omega)$ associated to $\rho$ and $\phi$ can be filled with a symplectic foliation without interior closed leaves.
\end{lemma}

As in \cite[Vocabulary 3.2]{Meigniez}, here a subset $A\subset (-1,1)$ is said to \emph{bracket} another subset $B\subset (-1,1)$ if  every point of $B$ lies between two points of $A$.

\begin{proof}
    The proof being completely analogous for the straight and round case, we treat the first case only.

    The construction of the foliation $\cF$ filling the hole is exactly the same as in \cite[Lemma 3.3]{Meigniez}, except that we use $X$ instead of the $r$-torus and that, under our hypotheses, it can be endowed with a leafwise symplectic form inducing $\mu$ on the boundary. 
    Here are more details.
    
    The foliation in \cite[Proof of Lemma 3.3]{Meigniez} is constructed as the pullback of the suspension foliation $\mathcal{S}(\lambda(\gamma_1),\dots,\lambda(\gamma_r))$ on $X \times I$ under a map of the form
    \[F\colon X \times \D^2 \times I \to X\times I,\quad (x,y,z) \mapsto (x,f(y,z)) \;,\]
    where $f\colon\D^2 \times I \to I$ is a Morse function obtained from the height function, by adding two critical points, $(y_1,z_1)$ and $(y_2,z_2)$ of index $1$ and $2$ respectively, in canceling position. 
    The corresponding critical values are denoted $c_1, c_2 \in I$.

    We claim that there exists a vector field $V\in \X(X\times \D^2 \times I)$, which is transverse to $\cF$, equal to $\partial_z$ near the boundary,satisfies $\d z(V) = 1$ everywhere. Away from the slices $X \times (y_i,z_i)$, $i = 1,2$ the map $F$ is just a projection (since $f$ is the height function on $\D^2 \times I$). Thus, $\d F(\partial_z) = \partial_z$ which is transverse to $\cS$ implying $\partial_z$ is transverse to $\cF$. 
    Furthermore, we may assume without loss of generality that $c_i$ is not a fixed point of $\lambda(\gamma_1)$. 
    Since the generators of $\pi_1(X)$ can be represented by embedded curves, we can find a vector field $W \in \X(X)$ satisfying
    \[ W(\gamma_1(t)) = \dot{\gamma}_1(t)\;.\]
    Then, we can arrange that $\mathcal{S}$ is transverse to $W$ along $X \times \{c_i\}$.
    Hence, along $X \times (y_i,c_i)$ we have that $\d F(W + \partial_z)$ is transverse to $\mathcal{S}$. Since transversality is open, $W + \partial_z$ is transverse to $\cF$ on a neighborhood of the singular tori. A simple interpolation argument provides the existence of $V$.

    For the rest of the argument, we fix any vector field $V$ as above, and we denote by
    \[ \phi_t:X \times \D^2 \times \{0\} \to X \times \D^2 \times \{t\},\quad (x,y,0) \mapsto \phi_t(x,y,0),\]
    the image of the bottom face under its flow. 
    The induced map $X\times\D^2\times\{0\}\to X\times\D^2\times\{1\}$ is a priori not the identity;
    however, as this is just a time-$1$ flow, one can easily correct this issue as follows.
    We fix a smooth bump function $\tau:I \to [0,1]$ satisfying
    \[ \tau(z) = \begin{cases} z  & 0 \leq z \leq 1 - 2\varepsilon \\ 1-z & 1-\varepsilon \leq z \leq 1, \end{cases}\]
    where $\varepsilon >0$ is so small that $\cF$ equals the product foliation on $X \times \D^2 \times [1-2\varepsilon,1]$. 
    
    Using this data we can define a fibration $\pi:X \times \D^2 \times I \to X \times \D^2 \times \{0\}$ by setting
    \[ \pi(x,y,z) := \pr \circ \phi^{-1}_{\tau(z)}(x,y,\tau(z)),\]
    where $\pr:X \times \D^2 \times I \to X \times \D^2 \times \{0\}$ denotes the projection map.
    
    Observe that on $X \times \D^2 \times [0,1-2\varepsilon]$ the fibers of $\pi$ equal the flowlines of $V$, and hence are transverse to $\cF$. 
    On the other hand, $\cF$ is a product foliation on $X \times \D^2 \times [1-2\varepsilon,1]$. Since $V$ always has a positive $\partial_z$-component, the fibers of $\pi$ are transverse to $\cF$ also on this region. 
    Lastly, note that by construction $\pi$ equals $\pr$ on a neighborhood of the boundary of $X \times \D^2 \times I$. As such the pullback $\pi^*\omega$ defines a leafwise symplectic form on $\cF$, equal to $\pr^* \omega$ near the boundary.

\end{proof}

\medskip

In order to fill the hole with a symplectic foliation, we will need to first ``enlarge'' it; this will be done as follows.
Denote by
\[ (\D^1 \times \D^{2n},\cF_{st},\omega_{st}),\]
the $1$-handle foliated by standard Darboux discs $(\D^{2n},\omega_{st})$. 
For any straight hole $(H,\cF,\omega)$ the symplectic foliation is tangent to the top and bottom horizontal boundary pieces.
By taking the union with a handle, starting at the top boundary of $H$ and ending at the bottom one and attached along Darboux charts, 
one can obtain an enlarged hole
\begin{equation}\label{eq:enlarged_hole}
    (\widehat{H},\widehat{\cF},\widehat{\omega}) := (H,\cF,\omega) \cup (\D^1\times \D^{2n},\cF_{st},\omega_{st}).
\end{equation}

The following proposition follows from a straightforward adaptation of \cite[Proposition~3.5]{Meigniez} using the symplectic version of \cite[Lemmas 3.3 and 3.4]{Meigniez} that we described above, i.e.\ \Cref{lemma:filling_hole}:

\begin{proposition}\label{prop:filling_enlarged_hole}
    Let $(H,\cF,\omega)$ be the straight hole associated to the trivial representation $\rho:\pi_1(X) \to \Diff_c(-1,1)$, and a non-trivial element $\phi \in \Diff_c(-1,1)$.  
    Then, the enlarged hole $(\widehat{H},\widehat{\cF},\widehat{\omega})$ can be filled without interior leaves.
\end{proposition}

\bigskip

Let us now go back to our main purpose of finding a procedure to eliminate closed leaves from a symplectic foliation.
The overall idea to achieve this is the following: one should cut along a local transversal of codimension $1$ in the leaves, and glue back in a way that creates holonomy and hence opens up the closed leaf. 
This cut-and-paste procedure requires creating some holes, which one will then fill (in fact, after further enlarging them with several handles) using \Cref{prop:filling_enlarged_hole}.
We summarize the output of this strategy, which will be detailed further below, as the following result:

\begin{proposition}\label{prop:RemovingClosedLeaf}
    Let $L$ be a closed leaf of a codimension~$1$ symplectic foliation $(\cF,\omega)$ on $M^{2n+1}$ for $n \geq2$. 
    Suppose there exists a neighborhood $U$ of $L$ such that $\cF|_U$ has leaves accumulating on $L$ from both sides. 
    Then, there exists symplectic foliation $(\cF',\omega')$ on $M$ such that:
    \begin{itemize}
        \item $(\cF',\omega')$ and $(\cF,\omega)$ agree outside $U$,
        \item  every leaf of $\cF'$ intersecting $U$ is open.
    \end{itemize}
\end{proposition}

The previous proposition is an immediate consequence of the following, more technical, lemma, which says in particular that it is enough to modify $(\cF,\omega)$ in a foliated chart $B\coloneqq \D^1\times \D^{2n}$, where $\D^{2n}$ is a symplectic Darboux ball and the foliation is given by the second factor, union a neighborhood of a transverse loop intersecting $B$ several times.

\begin{lemma}\label{lem:RemovingClosedLeaf}
    Let $L$ be a closed leaf of a codimension-one symplectic foliation $(\cF,\omega)$ on $M^{2n+1}$ for $n \geq2$. 
    Suppose that $U$ is as in \Cref{prop:RemovingClosedLeaf}, and that $(V,\cF|_{V},\omega|_V)$ is an (arbitrarily small) foliated Darboux chart intersecting $L$ and contained in $U$. 
    Then there are a new symplectic foliation $(\cF',\omega')$ on $M$ and an embedded handlebody
    \[ X := B^{2n+1} \cup \bigcup_{1}^k \D^1 \times \D^{2n},\]
    contained in $U$, such that:
\begin{enumerate}
    \item the restriction $(M\setminus X, \cF'|_{M\setminus X}, \omega'|_{M \setminus X})$ admits an embedding into $(M,\cF,\omega)$,
    \item the restriction $\omega'\vert_X$ is leafwise exact,
    \item the leaves of the characteristic foliation $\cF' \cap \partial X$ are diffeomorphic to $\S^{2n-1}$ or a wedge of two copies of $\S^{2n-1}$. 
    \end{enumerate}
\end{lemma}

\begin{proof}
    Identify the foliated Darboux chart $(V,\cF|_V,\omega|_V) \subset U \subset M$ with the subset $[-2,2]\times \D^{2n}$ in
    \begin{equation}\label{eq:foliated_chart}
        (\R \times \R^{2n}, \cF = \bigcup_{t \in \R}\, \{t\} \times \R^{2n},\omega = \omega_{st}),
    \end{equation}
    such that the intersection with $L$ equals $\{0\} \times \R^{2n}$.
    We further identify $\R^{2n} = \R^{2n-1} \times \R$ and choose a domain $T\subset \D^{2n-1}\subset \R^{2n-1}\times \{0\}$ whose boundary $\partial T$ is a (smooth) connected, separating hypersurface in $\R^{2n-1}$ containing a copy of $\Z$ in its fundamental group.  
    For example, we can take $T$ to be a solid torus $\S^1 \times \D^{2n-2}$ given by a small neighborhood of an unknot in $\D^{2n-1}\subset\R^{2n-1}$.  
    Observe that 
    \[ \pi_1(\R^{2n} \setminus \partial T) = \Z \;,\]
    generated by a closed loop in $\R^{2n}\setminus \partial T$ having a single intersection with $T$.

    We fix a tubular neighborhood $\partial T \times \D^2$ of $\partial T \subset \R^{2n}$, and cut out a straight hole
    \[ H := [-\delta,\delta] \times \partial T \times \D^2 \subset \R \times \R^{2n} \;,\]
    where $\delta > 0$ is chosen such that $\{ \delta\} \times \R^{2n}$ in the foliated chart of \Cref{eq:foliated_chart} lies in a leaf of $\cF\vert_U$ accumulating on $L$. 
    The resulting straight hole $(H,\cF,\omega)$ has symplectic base $(\partial T \times \D^2,\omega_{st})$ and the associated representation $\rho: \pi_1(\partial T) \to \Diff_c(-\delta,\delta)$, and distinguished element $\phi \in \Diff_c(-\delta,\delta)$ are trivial.

    Choose $\wtd{\phi} \in \Diff([-\delta,\delta])$ satisfying:
    \begin{equation}\label{eq:infinite_orbit}
        \wtd{\phi}(t) < t,\quad \text{ for all $-\delta < t < \delta$}.
    \end{equation}

    Then, we can cut $V = \R \times \R^{2n}$ along
    \begin{equation}\label{eq:wall}
        W := [-\delta,\delta] \times T  \subset \R \times \R^{2n},
    \end{equation}
    and glue back using the map
        \[ \psi\colon W \to W,\quad (t,x)\mapsto (\wtd{\phi}(t),x) \;.\]
    It is easily checked that this yields a new symplectic foliation~$(\wtd{\cF},\wtd{\omega})$ which agrees with $(\cF,\omega)$ outside $W$. 
    Moreover, \Cref{eq:infinite_orbit} implies that every $t \in (-\delta,\delta)$ has infinite orbit.
    Together with the fact that $\{\pm \delta\} \times T$ lie in open leaves (since they accumulate on $L$), it follows that all leaves of $\wtd{\cF}$ intersecting $U$ are open.

    It remains to fill the straight hole $(H,\wtd{\cF},\wtd{\omega})$ without creating closed leaves.
    To this end we want to apply  \Cref{prop:filling_enlarged_hole} for which we first need to enlarge the hole as in \Cref{eq:enlarged_hole}.
    The main technical difficulty is to ensure that the resulting leafwise $\omega'$ is leafwise exact on $X$ given by $V$ with the handles attached. 
    In order to do so, we will need that the core of the $1$-handle is tangent to $\ker \Omega$, for some globally closed extension $\Omega$ of $\omega$ over $\wtd V$ obtained from $V$ by shrinking it in a specific way, away from the attaching regions of the handles.
    This will imply that one can find a globally closed extension $\wtd \Omega$ of $\omega\vert_{\wtd X}$, agreeing with $\Omega\vert_{\wtd V}$ on $\wtd V$, to all of $\wtd X$, where the latter denotes the union of $\wtd V$ and the handles attached to $V$. 
    As $X$ retracts onto $\wtd X$, this will give the conclusion.
    Here are the details.  
    \medskip

    We start by choosing an embedded curve $\gamma\colon I \to M \,\setminus\, H$ \emph{positively} transverse to $\wtd{\cF}$ and connecting the top and bottom plaques of $H$.
    This is possible since the top and bottom plaque of $H$ lie in leaves of $\cF$ accummulating on $L$, and that by construction the holonomy map $\phi$ accumulates everything onto $-\delta$.
    Without loss of generality we may assume that 
    each connected component $\gamma_i$ of $\gamma \cap V$ intersects the wall $W$ positively and transversely exactly once.
    Note that this means that $\gamma$ will intersect the foliated Darboux ball $V$ multiple times, qualitatively as shown in \Cref{fig:qualitative}.

\begin{figure}[ht]
     \centering
     \def\svgwidth{0.30\textwidth}
     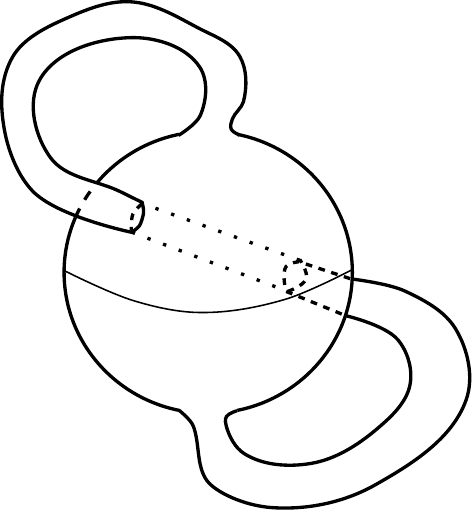
\caption{A qualitative embedded picture of the neighborhood of the curve $\gamma$ union the Darboux ball.}
        \label{fig:qualitative}
\end{figure}

    Next, we choose a nowhere vanishing vector field $Y$ near $V\subset  \R_t \times \R^{2n}$, obtained as perturbation of $\partial_t$ and satisfying that:
    \begin{itemize}
        \item $Y$ is tangent to $\gamma$ near $W$,
        \item $Y = \partial_t$ outside an (arbitrarily) small neighborhood of $W$, and near the boundary of the hole $\partial H$,
        \item $Y$ has positive $\partial_t$ component everywhere.
    \end{itemize}
    For a qualitative illustration, see \Cref{fig:wall_piercing_section} and \Cref{fig:wall_piercing_section_after_perturbation}.
    The first shows, near one of the intersection points, how the handle given by a thickening of $\gamma$ (in green) intersects the given chart near the wall $W$ (in red); the picture also indicates the vector field $\partial_t$ (in blue) and the profile of the neighborhood $V$ (in thickened black lines).
    The second picture shows how the vector field $Y$ (in blue) looks like in that same region, with the profile of the shrinked domain $\wtd V$ (in thickened black lines) around the wall $W$ (in red).

The flowlines of $Y$ define a projection $\pi_Y:\R \times \R^{2n} \to \{0\} \times \R^{2n}$ and, as in the proof of \Cref{lemma:filling_hole}, this allows us to replace $\omega_{st}$ by $\pi_Y^*\omega_{st}$ which is naturally defined on $M$ (and not only leafwise), ambiently closed and whose kernel is tangent to $\gamma$ in a neighborhood of $W$.

\begin{figure}[ht]
     \centering
     \def\svgwidth{0.30\textwidth}
    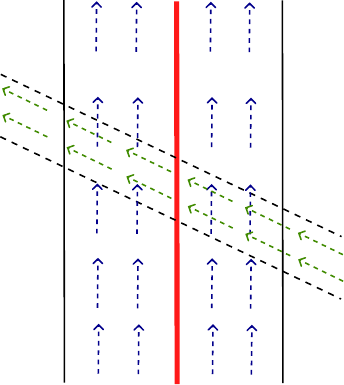
    \caption{The picture is in section. In red the wall, in blue the kernel direction of the Darboux ball, in green that of the handle.}
        \label{fig:wall_piercing_section}
\end{figure}

\begin{figure}[ht]
     \centering
     \def\svgwidth{0.40\textwidth}
    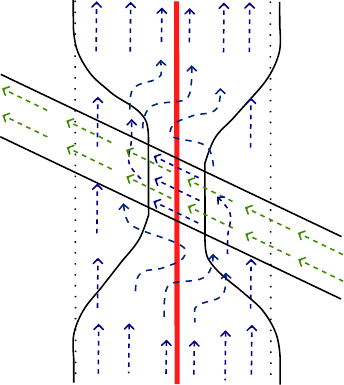
    \caption{The picture is in section. In red the wall, in blue the kernel direction of the Darboux ball, after perturbation, and in green that of the neighborhood of $\gamma$. 
    The neighborhood $B$ of the wall $W$ is smaller than $V$ in the previous picture: it is shrinked near the intersection point of $\gamma$ with $W$, but it is elsewhere the same; over the intersection of the neighborhood of $\gamma$ and $B$, the two kernel directions then coincide.}
        \label{fig:wall_piercing_section_after_perturbation}
\end{figure}

We now consider an embedded ball $B^{2n+1} \subset V$ around $W$ such that:
\begin{itemize}
    \item the leaves of the induced characteristic foliation $\wtd{\cF} \cap \partial B$ are diffeomorphic to $\S^{2n-1}$ or a point,
    \item $\partial B$ is transverse to $\gamma$,
    \item $\ker \pi_Y^*\omega_{st}$ is tangent to $\gamma$ at points where $\gamma$ intersects $\partial B$.
\end{itemize}

Such a $B$ can be obtained for instance as follows.
Recall that each connected component $\gamma_i$ of $\gamma \cap V$ intersects the wall $W$ exactly once; let us parametrize $\gamma_i\colon[-1,1]\to M$ so that $\gamma_i\cap W = \gamma_i(0)$. 
Then, $B$ is simply obtained by considering a slightly shrinked copy $V'$ of $V$, that is further shrinked by ``pushing its boundary'' towards $W$ along $\gamma_i$, in the direction of $\gamma_i'$ along $\gamma_i\vert_{[-1,-\epsilon)}$ and in the direction of $-\gamma_i'$ along $\gamma_i\vert_{(\epsilon,1]}$, for some $\epsilon>0$ very small.
See \Cref{fig:wall_piercing_section_after_perturbation} for a pictorial representation.

Note then that $\gamma$ may enter and exit $B$ several times. 
Hence the union of $B$ with the tubular neighborhood of $\gamma$ defined above can be interpreted as a ball $B$ with finitely many $1$-handles attached.
We denote the union by:
\[ \wtd{X} := B^{2n+1} \cup \bigcup_{1}^k \D^1 \times \D^{2n}.\]

Using a leafwise Moser argument, we see that (after shrinking the tubular neighborhood of $\gamma$ if necessary) on each of the handles we can write $(\wtd{\cF},\wtd{\omega})$ as
\[ \big(\D^1 \times \D^{2n},\wtd{\cF} = \bigcup_{s \in \D^1} \{s\} \times \D^{2n},\wtd{\omega} = \omega_{st}\big),\]
where $\omega_{st}$ denotes the standard symplectic form on $\D^{2n}$. 
Let $Z$ be a vector field on $M$, which in the above coordinates equals $\partial_s$. 
Then, we obtain a closed extension of the leafwise symplectic form (on $\D^1 \times \D^{2n})$ by taking $\pi^*_Z\omega_{st}$. 
By construction, $Z$ and $Y$ agree near $(\partial \D^1) \times \D^{2n}$ so that the extension $\wtd{\omega}$ is well-defined on the whole of $\wtd{X}$.

Since $\wtd{X}$ is a ball with finitely many $1$-handles attached it follows that $H^2(\wtd{X}) = 0$. Hence, since $\wtd{\omega}$ is closed, it is exact. This implies that its restriction to the leaves $\wtd{\omega}|_{\cF}$ is exact as well.

Let now $X$ be a thickening of $\wtd{X}$ with the following properties
\begin{itemize}
    \item the inclusion $\wtd{X}\hookrightarrow X$ is an $H^2$-isomorphism,
    \item the vector field $Y$ equals $\partial_t$ near the boundary of $X$.
\end{itemize}
Then, the first property implies that $\wtd\omega$ is leafwise exact on $X$, while the second property implies that $\wtd{\cF}$ and $\wtd{\omega}$ are unchanged on $M \setminus X$. In particular, the complement $(M\setminus X, \wtd{\cF}|_{M \setminus X},\wtd{\omega}_{M \setminus X})$ embeds into $(M,\cF,\omega)$.
\end{proof}

\subsection{Eliminating closed leaves while preserving a vanishing cycle}\label{sec:eliminating_preserving_vanishingcycle}

The results of the previous section allow us to eliminate closed leaves from a symplectic foliation. We now show that this construction preserves Lagrangian vanishing cycles.

\begin{lemma}\label{lem:closedleafvanishingcycle}
    Suppose we are in the setup of \Cref{prop:RemovingClosedLeaf}.
    Then $(\cF',\omega')$ can be arranged to satisfy the following property:
    if $(\cF,\omega)$ contains a Lagrangian vanishing cycle, then so does $(\cF',\omega')$.
\end{lemma}
\begin{proof}

    Since the $X$ constructed in \Cref{prop:RemovingClosedLeaf} can be arranged to be disjoint from a given Lagrangian vanishing cycle $\cC$ in $(M,\cF,\omega)$, it immediately follows that the embedding $\iota$ of $\cC$ into $M$ also defines an embedding into $(M,\cF',\omega')$ satisfying  \Cref{item:lvc_lagrangian_fol,item:lvc_core,item:lvc_boundary_extension} of \Cref{def:lagr_van_cycle}.

    For \Cref{item:lvc_sympl_aspherical_leaves}, 
    suppose that there is $f:\S^2 \to \cF'$ such that $\int_{\S^2} f^*\omega' > 0$.
Up to a homotopy of $f$, we can also assume that $f \pitchfork \partial X$. Hence we obtain a decomposition 
\[ \S^2 = \Sigma_{in} \cup \Sigma_{out},\]
where $\Sigma_{in} := f^{-1}(\overline{X})$ and $\Sigma_{out} := f^{-1}(M \setminus X)$ are surface with (common) boundary, consisting of a finite collection of circles $\gamma_1,\dots,\gamma_k \subset \S^2$, that we consider oriented as boundary of $\Sigma_{in}$.
Note that, as $\omega$ is leafwise exact on $X$, $\Sigma_{out}$ must be non-empty and $f$ cannot be constant on it.

Now, the topology of the leaves of $\cF' \cap \partial X$ implies that there exist $g_i : \D^2 \to \partial X$ whose image is contained in a leaf of $\cF' \cap \partial X$, and hence of $\cF' \cap X$, and with $g_i|_{\partial \D^2} = \gamma_i$.
Since $\omega'\vert_X$ is leafwise exact it also follows that
\[ \int_{\Sigma_{in}} f^*\omega' =  \int_{\D^2} g_i^*\omega'. \]
In turn this implies that
\[ \int_{\Sigma_{out}} f^*\omega' + \int_{\D^2} g_i^*\omega' = \int_{\S^2} f^*\omega' > 0.\]

    Consider now
    \[ \wtd{\Sigma}_{out} := \left( \bigsqcup_{i=1}^k D^2 \right) \bigcup_{\overline{\gamma_1},\ldots,\overline{\gamma_k}} \Sigma_{out} \;,\]
    i.e.\ the closed surface obtained by capping off every boundary component of $\Sigma_{out}$ with a disc. 
    Note that $\wtd{\Sigma}_{out}$ is homeomorphic to a union of $\S^2$'s.
    The map $f$ on $\Sigma_{out}$ can then be extended as $g_i$ on the disc capping $\overline{\gamma_i}$, which hence define a map $\wtd{f}:\wtd{\Sigma}_{out} \to M$. 
    Moreover, the symplectic area computation above implies that
    \[ \int_{\wtd{U\Sigma}_{out}} \wtd{f}^*\omega'> 0\]

    However, since the image of $\wtd{f}$ is contained in $M \setminus X$ and is non-constant (because $\Sigma_{out}$ is non-empty and $f$ is non-constant on it), this is not possible since it would contradict \Cref{item:lvc_sympl_aspherical_leaves} of \Cref{def:lagr_van_cycle} for the original foliation $(M,\cF,\omega)$. 
    We conclude $(M,\cF',\omega')$ satisfies \Cref{item:lvc_sympl_aspherical_leaves}.

    Lastly, the proof of \Cref{item:lvc_pi1_condition} is completely analogous, and we omit the details.
\end{proof}


\subsection{Proof of \Cref{thm:codim_1_examples}}\label{sec:proof_of_codim1_examples}

The fact that high-codimensional examples can be obtained from the codimension~$1$ ones is obvious; we hence deal with the codimension $1$ part of the statement.

Let $\cF_{Reeb}$ be the Reeb foliation on $\S^1 \times \S^2$, and $\Sigma^{2n}$ a product of $2$-dimensional closed surfaces of genus $\geq 1$. We consider the product manifold $M = \S^1 \times \S^2 \times \T^2 \times \Sigma$, endowed with the product foliation
\[ \cF  = \cF_{Reeb} \times T(\T^2 \times \Sigma).\]
Note that $\cF$ has only one closed leaf (diffeomorphic to $\T^2 \times \T^2 \times \Sigma)$, and carries a natural leafwise symplectic form $\omega$ coming from the sum of area forms on $\cF_{Reeb}$, $\T^2$, and the factors of $\Sigma$.

Observe that $(M,\cF,\omega)$ naturally admits a Lagrangian vanishing cycle $\cC=\D^2\times \T^{n+1}$.
This is simply given by the product of a $\D^2\subset (\S^1\times\S^2,\cF_{Reeb})$ with induced foliation given by concentric circles (i.e.\ a section of one of the two Reeb tubes $\S^1\times \D^2$ forming $(\S^1\times\S^2,\cF_{Reeb})$) together with an isotropic $\T^{n+1}$ inside $\T^2\times \Sigma$.
The latter exists because the symplectic structure there is just a product of symplectic structures of surfaces, and each surface factor clearly contains an isotropic loop.

Combining \Cref{prop:RemovingClosedLeaf} and \Cref{lem:closedleafvanishingcycle} guarantees that we can modify the above symplectic foliation by opening up the unique closed leaf while at the same time keeping the Lagrangian vanishing cycle.
Then, the new symplectic foliation doesn't satisfy any of the criteria in \Cref{obs:obstr_strongness} by construction.
Here, tautness follows from the fact that there are now only open leaves, and that one can always find a transverse loop passing through any given open leaf $F$ in any compact (cooriented) foliated manifold, as $F$ must necessarily intersect at least one foliated chart in a finite cover of the ambient manifold at least twice.
This concludes the proof. 
\hfill\qedsymbol

\bibliographystyle{alpha}
\bibliography{biblio}

\end{document}